\documentclass[%a4paper, 
11pt,  reqno]{amsart}
\usepackage[margin=1.2in,marginparwidth=1.5cm, marginparsep=0.5cm]{geometry}

\usepackage{amsmath,amssymb,amscd,amsthm,amsxtra, esint}

\usepackage{mathabx}

 \usepackage{slashed}
 \usepackage{float}

\usepackage{booktabs} % nice tables
\usepackage{microtype}
\usepackage{amssymb}
\usepackage{mathrsfs} 

\usepackage{upgreek} 

\usepackage{color}
\usepackage{hyperref}

\usepackage{xsavebox}
\usepackage{orcidlink}

\usepackage{cases}%%%

\definecolor{gr}{rgb}   {0.,   0.69,   0.23 }
\definecolor{bl}{rgb}   {0.,   0.5,   1. }
\definecolor{mg}{rgb}   {0.85,  0.,    0.85}
\definecolor{yl}{rgb}   {0.8,  0.7,   0.}
\definecolor{or}{rgb}  {0.7,0.2,0.2}

\usepackage{tikz} 
\usetikzlibrary {positioning}
\usepackage{marginnote}
\usepackage{scalerel} % to rescale the paraproduct symbols

\usetikzlibrary{shapes.misc}
\usetikzlibrary{shapes.symbols}
\usetikzlibrary{decorations}
\usetikzlibrary{decorations.markings}

\tikzset{
	dot/.style={circle,fill=black,draw=black,inner sep=0pt,minimum size=0.5mm},
	>=stealth,
	}

\tikzset{
	dot2/.style={circle,fill=black,draw=black,inner sep=0pt,minimum size=0.2mm},
	>=stealth,
	}

\tikzset{
	ddot/.style={circle,fill=black,draw=black,inner sep=0pt,minimum size=0.8mm},
	>=stealth,
	}

\tikzset{decision/.style={ % requires library shapes.geometric
        draw,
        diamond,
        aspect=1.5
    }}

\tikzset{dia2/.style
={diamond,fill=white,draw=black,inner sep=0pt,minimum size=1mm},
	>=stealth,
	}

\tikzset{dia/.style
={star,fill=black,draw=black,inner sep=0pt,minimum size=1mm},
	>=stealth,
	}

\tikzset{dia/.style
={diamond,fill=black,draw=black,inner sep=0pt,minimum size=1.3mm},
	>=stealth,
	}

\makeatletter
\def\DeclareSymbol#1#2#3{\xsavebox{#1}{\tikz[baseline=#2,scale=0.15]{#3}}}
\def\<#1>{\xusebox{#1}}
\makeatother

\newsavebox{\peA}
\newsavebox{\pneA}
\newsavebox{\plA}
\newsavebox{\pgA}
\newsavebox{\pleA}
\newsavebox{\pgeA}
\newsavebox{\pezA}

\savebox{\peA}{\tikz \draw (0,0) node[shape=circle,draw,inner sep=0pt,minimum size=8.5pt] {\scriptsize  $=$};}
\savebox{\pneA}{\tikz \draw (0,0) node[shape=circle,draw,inner sep=0pt,minimum size=8.5pt] {\footnotesize $\neq$};}
\savebox{\plA}{\tikz \draw (0,0) node[shape=circle,draw,inner sep=0pt,minimum size=8.5pt] {\scriptsize $<$};}
\savebox{\pgA}{\tikz \draw (0,0) node[shape=circle,draw,inner sep=0pt,minimum size=8.5pt] {\scriptsize $>$};}
\savebox{\pleA}{\tikz \draw (0,0) node[shape=circle,draw,inner sep=0pt,minimum size=8.5pt] {\scriptsize $\leqslant$};}
\savebox{\pgeA}{\tikz \draw (0,0) node[shape=circle,draw,inner sep=0pt,minimum size=8.5pt] {\scriptsize $\geqslant$};}
\savebox{\pezA}{\tikz \draw (0,0) node[shape=circle,draw,
fill=white, % color = white, 
inner sep=0pt,minimum size=8.5pt]{} ;}

\def \peB{\mathchoice
{\scalebox{.7}{{\usebox{\peA}}}}
{\scalebox{.7}{{\usebox{\peA}}}}
{\scalebox{.7}{{\usebox{\peA}}}}
{}
}

\def \pezB{\mathchoice
{\scalebox{.7}{{\usebox{\pezA}}}}
{\scalebox{.7}{{\usebox{\pezA}}}}
{\scalebox{.7}{{\usebox{\pezA}}}}
{}
}

\newcommand{\pe}{\mathbin{{\peB}}}

\newcommand{\pez}{\mathbin{{\pezB}}}

\usetikzlibrary{decorations.pathmorphing, patterns,shapes}

\DeclareSymbol{dot}{-2.7}{\draw (0,0) node[dot]{};}

\DeclareSymbol{nx}{0}{\draw (0, 0.85) node[]{$\nabla$}; \draw[line width=0.7pt] (-0.7,-0.45) -- (0.5, 1.85);}

\DeclareSymbol{1}{0}{\draw[white] (-.4,0) -- (.4,0); \draw (0,0)  -- (0,1.2) node[dot] {};}
\DeclareSymbol{2}{0}{\draw (-0.5,1.2) node[dot] {} -- (0,0) -- (0.5,1.2) node[dot] {};}
\DeclareSymbol{3}{0}{\draw (0,0) -- (0,1.2) node[dot] {}; \draw (-.7,1) node[dot] {} -- (0,0) -- (.7,1) node[dot] {};}
\DeclareSymbol{30}{-3}{\draw (0,0) -- (0,-1); \draw (0,0) -- (0,1.2) node[dot] {}; \draw (-.7,1) node[dot] {} -- (0,0) -- (.7,1) node[dot] {};}
\DeclareSymbol{31}{-3}{\draw (0,0) -- (0,-1) -- (1,0) node[dot] {}; \draw (0,0) -- (0,1.2) node[dot] {}; \draw (-.7,1) node[dot] {} -- (0,0) -- (.7,1) node[dot] {};}
\DeclareSymbol{32}{-3}{\draw (0,0) -- (0,-1) -- (1,0) node[dot] {}; \draw (0,0) -- (0,-1) -- (-1,0) node[dot] {}; \draw (0,0) -- (0,1.2) node[dot] {}; \draw (-.7,1) node[dot] {} -- (0,0) -- (.7,1) node[dot] {};}
\DeclareSymbol{20}{-3}{\draw (0,0) -- (0,-1);\draw (-.7,1) node[dot] {} -- (0,0) -- (.7,1) node[dot] {};}
\DeclareSymbol{22}{-3}{\draw (0,0.3) -- (0,-1) -- (1,0) node[dot] {}; \draw (0,0.3) -- (0,-1) -- (-1,0) node[dot] {};\draw (-.7,1) node[dot] {} -- (0,0.3) -- (.7,1) node[dot] {};}
\DeclareSymbol{31p}{-3}{\draw (0,0) -- (0,-1) -- (1,0) node[dot] {}; \draw (0,0) -- (0,1.2) node[dot] {}; \draw (-.7,1) node[dot] {} -- (0,0) -- (.7,1) node[dot] {}; 
\draw (0,-1) node{\scaleobj{0.5}{\pez}}; 
\draw (0,-1) node{\scaleobj{0.5}{\pe}}; }
\DeclareSymbol{32p}{-3}{\draw (0,0) -- (0,-1) -- (1,0) node[dot] {}; \draw (0,0) -- (0,-1) -- (-1,0) node[dot] {}; \draw (0,0) -- (0,1.2) node[dot] {}; \draw (-.7,1) node[dot] {} -- (0,0) -- (.7,1) node[dot] {}; 
\draw (0,-1) node{\scaleobj{0.5}{\pez}};
\draw (0,-1) node{\scaleobj{0.5}{\pe}};}
\DeclareSymbol{22p}{-3}{\draw (0,0.3) -- (0,-1) -- (1,0) node[dot] {}; \draw (0,0.3) -- (0,-1) -- (-1,0) node[dot] {};\draw (-.7,1) node[dot] {} -- (0,0.3) -- (.7,1) node[dot] {}; 
\draw (0,-1) node{\scaleobj{0.5}{\pez}};
\draw (0,-1) node{\scaleobj{0.5}{\pe}};}
\DeclareSymbol{21p}{-3}{\draw (0,0.3) -- (0,-1) -- (1,0) node[dot] {}; 
\draw (-.7,1) node[dot] {} -- (0,0.3) -- (.7,1) node[dot] {}; 
\draw (0,-1) node{\scaleobj{0.5}{\pez}};
\draw (0,-1) node{\scaleobj{0.5}{\pe}};}

\DeclareSymbol{21}{-3}{\draw (0,0.3) -- (0,-1) -- (1,0) node[dot] {}; 
\draw (-.7,1) node[dot] {} -- (0,0.3) -- (.7,1) node[dot] {}; 
}

\DeclareSymbol{11p}{-3}{\draw (0,0) node[dot2] {} -- (0,-1) -- (1,0) node[dot] {}; 
\draw (0,0) -- (0,1.2) node[dot] {};  %\draw (-.7,1) node[dot] {} -- (0,0) -- (.7,1) node[dot] {}; 
\draw (0,-1) node{\scaleobj{0.5}{\pez}}; 
\draw (0,-1) node{\scaleobj{0.5}{\pe}}; }
\DeclareSymbol{100}{-3}
{\draw[white] (-.4,0) -- (.4,0); 
\draw (0,0) node[dot2] {} -- (0,-1)  ; 
\draw (0,0) -- (0,1.2) node[dot] {};}  %\draw (-.7,1) node[dot] {} -- (0,0) -- (.7,1) node[dot] {}; 
\usetikzlibrary{mindmap,backgrounds,trees,arrows,decorations.pathmorphing,decorations.pathreplacing,positioning,shapes.geometric,shapes.misc,decorations.markings,decorations.fractals,calc,patterns}

\tikzset{>=stealth',
         cvertex/.style={circle,draw=black,inner sep=1pt,outer sep=3pt},
         vertex/.style={circle,fill=black,inner sep=1pt,outer sep=3pt},
         star/.style={circle,fill=yellow,inner sep=0.75pt,outer sep=0.75pt},
         tvertex/.style={inner sep=1pt,font=\scriptsize},
         gap/.style={inner sep=0.5pt,fill=white}}
\tikzstyle{mybox} = [draw=black, fill=blue!10, very thick,
    rectangle, rounded corners, inner sep=10pt, inner ysep=20pt]
\tikzstyle{boxtitle} =[fill=blue!50, text=white,rectangle,rounded corners]

\tikzstyle{decision} = [diamond, draw, fill=blue!20,
    text width=4.5em, text badly centered, node distance=2.5cm, inner sep=0pt]
\tikzstyle{block} = [rectangle, draw, fill=blue!20,
    text width=2.7cm, text centered, rounded corners, minimum height=3em]

\tikzstyle{line} = [draw, very thick, color=black!50, -latex']

\tikzstyle{cloud} = [draw, ellipse,fill=red!40, 
node distance=2.5cm,
    minimum height=1em]

\tikzstyle{cloud2} = [draw, ellipse,fill=red!30, text=white,text width=10em, node distance=2.5cm, text centered, minimum height=4em]

\tikzstyle{cloud3} = [draw, ellipse, fill=cyan!30, 
node distance=2.5cm,
    minimum height=3em, 
    text width=1.7cm]

\tikzstyle{cloud4} = [draw, ellipse,fill=orange!70, node distance=2.5cm,
   text width=2.1cm,  
           minimum height=3em]

\tikzstyle{cloud5} = [draw, ellipse,fill=red!20, node distance=2.5cm,
    text centered, 
    text width=3.0cm, 
    minimum height=3em]

\tikzstyle{cloud6} = [draw, ellipse,fill=red!20, node distance=2.5cm,
    text width=1.5cm, 
    text centered, 
    minimum height=3em]

\tikzset{
    position/.style args={#1:#2 from #3}{
        at=(#3.#1), anchor=#1+180, shift=(#1:#2)
    }
}
\newtheorem{theorem}{Theorem} [section]

\newtheorem{lemma}[theorem]{Lemma}
\newtheorem{proposition}[theorem]{Proposition}
\newtheorem{remark}[theorem]{Remark}

\newtheorem{definition}[theorem]{Definition}

\DeclareMathOperator*{\supp}{supp}

\newcommand{\1}{\hspace{0.2mm}\text{I}\hspace{0.2mm}}
\newcommand{\II}{\text{I \hspace{-2.8mm} I} }
\newcommand{\noi}{\noindent}
\newcommand{\Z}{\mathbb{Z}}
\newcommand{\R}{\mathbb{R}}
\newcommand{\T}{\mathbb{T}}
\newcommand{\bul}{\bullet}

\let\P= \undefined
\newcommand{\P}{\mathbf{P}}

\newcommand{\E}{\mathbb{E}}

\newcommand{\K}{\mathcal{K}}

\newcommand{\F}{\mathcal{F}}

\newcommand{\al}{\alpha}
\newcommand{\be}{\beta}
\newcommand{\dl}{\delta}

\newcommand{\nb}{\nabla}

\newcommand{\Dl}{\Delta}
\newcommand{\eps}{\varepsilon}
\newcommand{\kk}{\kappa}
\newcommand{\g}{\gamma}
\newcommand{\G}{\Gamma}
\newcommand{\ld}{\lambda}
\newcommand{\Ld}{\Lambda}
\newcommand{\s}{\sigma}

\newcommand{\ft}{\widehat}
\newcommand{\Ft}{{\mathcal{F}}}
\newcommand{\wt}{\widetilde}
\newcommand{\cj}{\overline}
\newcommand{\dx}{\partial_x}
\newcommand{\dt}{\partial_t}

\newcommand{\embeds}{\hookrightarrow}
\newcommand{\LRA}{\Longrightarrow}

\newcommand{\ta}{\theta}

\renewcommand{\l}{\ell}
\renewcommand{\o}{\omega}
\renewcommand{\O}{\Omega}

\newcommand{\les}{\lesssim}
\newcommand{\ges}{\gtrsim}

\newcommand{\jb}[1]
{\langle #1 \rangle}

\newcommand{\jbb}[1]{\bigl\langle #1 \bigr\rangle}

\newcommand{\fbb}[1]
{[\hspace{-0.6mm}[ #1 ]\hspace{-0.6mm}]}

\newcommand{\ind}{\mathbf 1}

\renewcommand{\S}{\mathcal{S}}

\newcommand{\N}{\mathbb{N}}

\renewcommand{\H}{\mathcal{H}}

\newtheorem*{ackno}{Acknowledgements}

\newcommand{\I}{\mathcal{I}}

\newcommand{\B}{\mathcal{B}}

\numberwithin{equation}{section}
\numberwithin{theorem}{section}

\newcommand{\justin}[1]{\marginpar{\color{blue} $\Leftarrow\Leftarrow\Leftarrow$}{\smallskip \noi\color{blue}Justin: #1}\smallskip}

\newcommand{\younes}[1]{\marginpar{\color{red} $\Leftarrow\Leftarrow\Leftarrow$}{\smallskip \noi\color{red}Younes: #1}\smallskip}

\newcommand{\Q}{\mathbf{Q}}
\newcommand{\PP}{\mathbb{P}}

\DeclareMathOperator{\Law}{Law}

\newcommand{\muu}{\vec{\mu}}

\newcommand{\rhoo}{\vec{\rho}}

\newcommand{\W}{\mathcal{W}}

\newcommand{\Dr}{\Theta}

\newcommand{\D}{\mathcal{D}}

\newcommand{\Pii}{\mathbf{\Pi}}

\usepackage{comment}

\newcommand{\Ta}{\Theta}

\newcommand{\mc}[1]
{\mathcal{#1}}

\newcommand{\Prob}{\mathbb{P}}

\newcommand{\vu}{\vec{u}}

\makeatletter
\@namedef{subjclassname@2020}{%
  \textup{2020} Mathematics Subject Classification}
\makeatother

\begin{document}

\baselineskip = 14pt

\title[Sinh-Gordon model]
{Invariant Gibbs dynamics for the\\
hyperbolic sinh-Gordon model}

%\author[T.~Oh, T.~Robert, and Y.~Zine]
%{Tadahiro Oh, Tristan Robert, and Younes Zine}

\author[J.~Forlano and Y.~Zine]
{Justin Forlano\orcidlink{0000-0002-8118-9911} and Younes Zine\orcidlink{0009-0001-7752-1205}}

\address{
Justin Forlano, School of Mathematics\\
Monash University\\
VIC 3800\\
Australia}

\email{justin.forlano@monash.edu}

%\address{
%Tristan Robert\\
%IECL - Site de Nancy\\
%Facult\'e des sciences et Technologies\\
%Campus, Boulevard des Aiguillettes\\
%54506 Vand{\oe}uvre-l\`es-Nancy\\
%France}
%
%\email{tristan.robert@univ-lorraine.fr}

\address{
Younes Zine\\
 \'Ecole Polytechnique F\'ed\'erale de Lausanne\\
1015 Lausanne\\ Switzerland}

\email{younes.zine@epfl.ch}



\begin{abstract}
We study the hyperbolic defocusing sinh-Gordon model with parameter $\be^2>0$ and its associated Gibbs dynamics on the two-dimensional torus. We establish global well-posedness of the model for a certain range of parameters $\beta^2>0$ with the corresponding Gibbs measure initial data and prove invariance of the Gibbs measure under the flow, thereby resolving a question posed by Oh, Robert, and Wang (2019). Our physical space approach hinges on developing a novel $L^\infty$-based well-posedness theory for wave equations with exponential-type nonlinearities, going beyond the classical $L^2$-based framework. This refinement allows us to fully leverage structural properties of Gaussian multiplicative chaos. As a by-product of our method, we also obtain an improved well-posedness theory for the hyperbolic Liouville model.
\end{abstract}

\maketitle
\tableofcontents
\section{Introduction}
\label{sec:intro}

\subsection{The sinh-Gordon model}\label{sec:liouville}
We consider
the stochastic damped sinh-Gordon equation (SdSG) on 
$\T^2 = (\R/2\pi\Z)^2$:
\begin{align}
\begin{cases}
\dt^2 u + \dt u + (1- \Dl)  u   + \iota \be \sinh(\be u) = \sqrt{2}\xi,\\
(u, \dt u) |_{t = 0} = (u_0, u_1) , 
\end{cases}
\qquad (t, x) \in \R_+\times\T^2,
\label{liouville}
\end{align}
where $\be, \iota \in \R\setminus \{0\}$, $\sinh(x) = \frac 12 (e^x -e^{-x})$ for $x\in \R$, and $\xi$ denotes space-time white noise on $\R_+ \times \T^2$. When $\iota>0$, we say that \eqref{liouville}
is defocusing, while if $\iota<0$, we say that it is focusing.

The stochastic partial differential equation (PDE)~\eqref{liouville} is the so-called canonical stochastic quantization equation~\cite{RSS} for the sinh-Gordon model with charge parameter $\be$:
\begin{align} d\rho_\be (u) \propto \exp\Big( - \iota \int_{\T^2} \cosh(\be u) dx - \int_{\T^2} \big(|\nb u|^2 + |u|^2\big) dx \Big) du
\label{mes_sinh}
\end{align}
Here, $du$ denotes the non-existent Lebesgue measure on $\D'(\T^2)$, rendering the expression~\eqref{mes_sinh} purely formal. The probability measure~\eqref{mes_sinh} is an instance of an Euclidean quantum field theory (QFT) and has been studied extensively in the mathematical physics community; we refer the interested reader to~\cite{FMS, KM, Lukyanov, BLeC, BDV, GTV} and the references therein for further discussion. The central idea of the stochastic quantization programme, initiated by Parisi and Wu~\cite{PW}, is to study the measure~\eqref{mes_sinh} through the analysis of the dynamics~\eqref{liouville}, to which PDE techniques may be employed.\footnote{In the original formulation of Parisi and Wu, stochastic quantization was introduced as a method to represent Euclidean gauge theories as equilibrium limits of stochastic dynamics, with the long-time behavior of the Langevin equation formally sampling the underlying field theory measure.} In this paper, we rigorously define the dynamics~\eqref{liouville} on the support of a rigorously defined version of the measure $\rho_\be \otimes \mu_0$, where $\mu_0$ denotes the white noise measure on $\D'(\T^2)$, for a suitable range of the parameter $\be$. We further prove invariance of $\rho_\be \otimes \mu_0$ under the resulting flow. This constitutes a {\it first step} toward completing the canonical stochastic quantization programme for~\eqref{mes_sinh}.

The standard stochastic quantization programme (that is, the study of QFTs of the form~\eqref{mes_sinh} via the analysis of a parabolic problem) has been highly successful in recent years, due to the introduction of Hairer’s theory of regularity structures~\cite{Hairer}, followed by the theory of paracontrolled distributions developed by Gubinelli, Imkeller, and Perkowski~\cite{GIP}. See also~\cite{Kupiainen, Duch, DGR} for works employing renormalization group ideas originating in mathematical physics. In particular, the well-posedness theory for stochastic quantization equations with both polynomial nonlinearities~\cite{DPD, Hairer, Kupiainen, MW1, MW2, CC, TW, GH2,  GH, Duch, DDJ, DGR, DHYZ} and non-polynomial nonlinearities~\cite{HS, CHS, HKK1, HKK2, BC, GHOZ} is now well understood. In the current dispersive context, however, the picture is different. The well-posedness theory in the polynomial setting is relatively well understood for both wave and Schr\"odinger models; see, for instance,~\cite{BO94, BO96, BT1, BT2, GKO, GKO2, GKOT, Tolo, Bring2, BDNY, OTh2, OTWZ, OWZ, DNY1, DNY2, LTW}. By contrast, the well-posedness theory of random dispersive equations with non-polynomial nonlinearities, such as $\sin(\be u)$, $e^{\be u}$, or $\sinh(\be u)$, remains much less developed, despite a few attempts~\cite{ORSW1, ORSW2, ORW, STz, ORTzW, Robert, Zine3}. In particular, in~\cite{Zine3}, the second author introduced a physical-space framework to analyze the hyperbolic sine-Gordon model. The present work seeks to extend these physical-space methods to equations with exponential-type nonlinearities such as \eqref{liouville}.

The main difficulty in studying \eqref{liouville} is due to the {\it singular} nature of \eqref{liouville}, that is, the roughness of the noise $\xi$ renders the solution $u$ merely a distribution a priori, and a renormalization process is needed.
As a result, their analysis falls outside the scope of standard wave PDE techniques.

We also point out that, unlike the situation for polynomial nonlinearities, the difficulty of the problems, both parabolic and hyperbolic, with non-polynomial nonlinearities depends on the parameters in the model. 
For the case of the sine nonlinearity on $\T^2$, we have the parabolic sine-Gordon model 
\begin{align}
\dt u +  (1-\Dl)u + \sin(\be u) = \sqrt 2 \xi \label{parasine}
\end{align}
and the hyperbolic sine-Gordon model
\begin{align}
\dt^2 u + \dt u + (1-\Dl)u + \sin(\be u) = \sqrt{2}\xi. \label{hypsine}
\end{align}
In the parabolic case, a first order expansion and Wick renormalization establishes local well-posedness for any $0<\be^2<4\pi$. Beyond that, there are an infinite number of thresholds $\be_{j}^{2}= \frac{j}{j+1}8\pi$, $j\in \N$, in which further renormalizations are required. Using the theory of regularity structures, Hairer-Shen \cite{HS} and Chandra-Hairer-Shen~\cite{CHS} proved the local well-posedness in the entire sub-critical regime $0<\be^2 <8\pi$. When $\be^2>8\pi$, \eqref{parasine} becomes super-critical and non-trivial dynamics is no longer expected to exist. 
Due to the lack of strong parabolic smoothing in the hyperbolic problem \eqref{hypsine}, the potential range of $\be^2$ is lower. In \cite{ORSW2}, local well-posedness of \eqref{hypsine} was established for $0<\be^2<2\pi$. By introducing a novel physical-space framework, the second author~\cite{Zine3} extended the first threshold for~\eqref{hypsine} to $\be^2 < 2\pi (1+\frac{3\sqrt{241}-41}{122})\approx 2.046\pi$. Moreover, due to a “variance blow-up” phenomenon, the second author identified $\be^{2}=6\pi$ as a critical threshold for~\eqref{hypsine}, which is perhaps surprisingly lower than the corresponding threshold $\be^2=8\pi$ for the parabolic problem~\eqref{parasine}.

We now return to discuss the setting of exponential-type nonlinearities. 
For the parabolic sinh-Gordon model
\begin{align}
\dt u + (1-\Dl)u +\iota \be \sinh(\be u) = \sqrt 2\xi, \label{parasinh}
\end{align}
Garban \cite{Garban} and Oh, Robert, and Wang \cite{ORW} proved local well-posedness for $0<\be^{2} <\frac{8\pi}{3+2\sqrt{2}} \approx 1.37\pi$. There are currently \textit{no} results for the hyperbolic sinh-Gordon model \eqref{liouville} and the main purpose of our work is to rectify this gap.

Another model of interest here is the parabolic Liouville model,
\begin{align}
\dt u+ (1- \Dl)  u   + \iota \be e^{\be u} = \sqrt 2 \xi, \label{paraliouville}
\end{align}
and the hyperbolic Liouville model 
\begin{align}
\dt^2 u+\dt u + (1- \Dl)  u   + \iota \be e^{\be u} = \sqrt{2}\xi. \label{hypliouville}
\end{align}
While any result for the sinh-Gordon model would be expected to apply for the Liouville model, the reverse is not necessarily true, since the defocusing Liouville model ($\iota >0$) benefits from having a sign-definite nonlinearity. This feature plays a key role, and it is now known that \eqref{paraliouville} is locally well-posed in the entire sub-critical regime $0<\be^{2} <8\pi$ \cite{HKK2}; see \cite{HKK1, ORW} for priori works. We point out that both the strong parabolic smoothing (as measured in $L^1$-based Besov spaces) and the maximum principle were exploited in \cite{HKK2}. The absence of both of these properties makes the hyperbolic case significantly more analytically challenging. To date, the only result for the defocusing equation \eqref{hypliouville} is due to~\cite{ORW}, which holds in the range
\begin{align}
0<\be^2 < \tfrac{32-16\sqrt{3}}{5}\pi \approx 0.86\pi. \label{ORWbe}
\end{align}
 We will discuss the strategy of \cite{ORW} and our new ideas in more depth in Section~\ref{SEC:proof}.

Lastly, we summarize that the above well-posedness theories for the sine-Gordon and sinh-Gordon/Liouville models exploit different structural features of their respective nonlinearities. At the analytical level, the sine nonlinearity may appear simpler to handle, since the function $\sin(x)$ is bounded, in contrast to the unbounded growth of $\sinh(x)$ or $e^{x}$. However, the singular nature of \eqref{parasine} and \eqref{hypsine} means that renormalization is required. Whilst renormalizaton destroys this boundedness property, after a first-order expansion, the contributions from the (smoother) remainder are, morally, straightforward to bound. For the exponential nonlinearities, there are still difficulties in controlling the contributions from the remainder  $e^{\pm \be v}$ and these are overcome if there is enough smoothing to ensure that the remainder belongs to $L^{\infty}_{t,x}$ (or that there is a sign-definite structure as for the defocusing Liouville models).

At the probabilistic level, both types of nonlinearities do not belong to any finite order Wiener chaos, which rules out the Wiener chaos estimate and so the main stochastic objects introduced require different strategies to study as compared to the polynomial setting. 
In the sine-Gordon case, the so-called imaginary Gaussian multiplicative chaos plays a fundamental role, and oscillations induce ``charge cancellations" \cite{HS, CHS} which allow to obtain high-moment estimates without altering the spatial-regularity. 
In contrast, in the exponential case, the Gaussian multiplicative chaos (GMC) (see \eqref{GMC}) does not exhibit such cancellations and instead displays the so-called \textit{intermittency} phenomenon: the spatial regularity depends both on the value of $\be^2$ and on the integrability exponent, and moreover, the GMC admits finite $L^p(\O)$-moments only for $1 \leq p < \frac{8\pi}{\be^2}$. See \cite{Garban} for further discussion of intermittency. This renders the well-posedness analysis particularly delicate. On the other hand, the GMC possesses a sign-definite structure, being a positive distribution, which can be exploited in the well-posedness theory, even at the local-in-time level.

\subsection{Setup and main result}\label{subsec:main}
In order to define \eqref{mes_sinh} rigorously, we first make sense of the quadratic part in \eqref{mes_sinh} as a Gaussian measure on distributions.
Given $ s \in \R$, 
let $\mu_s$ denote
a Gaussian measure, formally given by
\begin{align*}
 d \mu_s 
   =  Z_s^{-1}e^{-\frac 12 \| u\|_{{H}^{s}}^2} du
& = Z_s^{-1} \prod_{n \in \Z^2}   e^{-\frac 12 \jb{n}^{2s} |\ft u_n|^2}   d\ft u_n , 
 %\label{gauss0}
\end{align*}
where 
  $\jb{\,\cdot\,} =  \big(1+ |\,\cdot\,|^2\big)^{\frac{1}{2}}$ and $\ft u_n$ denotes the Fourier coefficient of $u$ at the frequency $n \in \Z^2$. We then define $\muu_s = \mu_s \otimes \mu_{s-1}.$
In particular, when $s = 1$, the measure $\muu_1$ is defined as the joint law of the random Fourier series:
\begin{align} \label{series}
\muu_1 = \Law(u_0,u_1)
\qquad \text{with} \qquad
u_0^\o  = \sum_{n \in \Z^2} \frac{g_n(\o)}{\jb{n}} e_n
\qquad \text{and} \qquad
u_1^\o  =  \sum_{n \in \Z^2} h_n(\o)e_n.
\end{align}
Here, 
 $e_n=(2\pi)^{-1}e^{i n\cdot x}$
 and $\{g_n,h_n\}_{n\in\Z^2}$ denotes  a family of independent, identically distributed, standard 
 complex-valued  Gaussian random variables such that $\cj{g_n}=g_{-n}$ and $\cj{h_n}=h_{-n}$, 
 $n \in \Z^2$.
It is easy to see that $\muu_1 = \mu_1\otimes\mu_0$ is supported on
\begin{align*}
\H^{s}(\T^2) := H^{s}(\T^2)\times H^{s - 1}(\T^2)
\end{align*}
\noi
for $s < 0$ but not for $s \geq 0$.

We now proceed to make sense of the $\cosh$ interaction appearing in \eqref{mes_sinh}, and we will do this through Fourier truncations.
Let $\chi\in C^{\infty}_{c}(\R^2)$ be a smooth, non-negative even function with $\supp \chi \subset \{ \xi\in \R^2 \, : \, |\xi|\leq 1\}$ and $\chi =1$ on $\{ \xi\in \R^2 \, :\, |\xi|\leq \frac 12\}$. Given $N\in \N$, we define 
\begin{align*}
\chi_{N}(n) = \chi(\tfrac{n}{N})
\end{align*}
and let $\Pii_{N}$ denote the smooth frequency projector onto the spatial frequencies $\{ n\in \Z^2 \, :\, |n|\leq N\}$ associated with the Fourier multiplier $\chi_{N}$. 
Given  $u_0 = u_0^\o$ as in \eqref{series}, 
i.e. $\operatorname{law}(u_0) = \mu_1$ and for $N\in \N$, 
set $\s_N$ as follows
 \begin{align}
 \label{sN}
 \s_N =  \E\Big[\big(\Pii_{N}u_0(x)\big)^2\Big] =\sum_{n\in\Z^2}\frac{\chi_N(n)^2}{\jb{n}^2}
 = \frac1{2\pi}\log N + o(1),
 \end{align}
\noi
as $N \to \infty$, 
uniformly in $x\in\T^2$. The truncated renormalized Gibbs measure is defined by
\begin{align}\label{gibbsN}
d \rhoo_N(u, v) = Z_{N}^{-1} \exp\Big( -\iota \g_N \int_{\T^2}\cosh(\be \Pii_{N}u) dx \Big) d \muu_1(u,v),
\end{align}
where $\g_N = \g_N(\be)$ is given by 
 \begin{align}
 \label{gN}
 \g_N(\be) = e^{-\frac{\be^2}{2}\s_N}
 \end{align}
 and satisfies $\g_{N}(-\be)=\g_{N}(\be)$ and $Z_{N}$ is a normalization constant.
In the defocusing case $\iota>0$, one then proves  the  existence of a  measure  $\rhoo$ such  that
\begin{align}
\lim_{N \to \infty} \rhoo_N = \rhoo,
\label{gibbs10}
\end{align}
in  the sense  of  total  variation. See Lemma \ref{lem:gibbs} below. The main point is that introducing the multiplicative constant $\g_N$ in \eqref{gibbsN} renormalizes the term $\cosh(\be \Pii_{N}u)$ since $\g_{N}\to 0$ as $N\to \infty$. As to how $\g_N$ appears here, we discuss that in the next section; see for instance \eqref{GMC}.

Next, we define the corresponding renormalized truncated sinh-Gordon dynamics:
\begin{align}
\dt^2 u_N   + \dt u_N  +(1-\Dl)  u_N 
+ \iota \be \g_N \Pii_{ N}\sinh\big(\be \Pii_{ N} u_N \big)  = \sqrt{2} \xi , 
\label{liouville2}
\end{align} 
with the Gibbs measure initial data $\rhoo_N$ \eqref{gibbsN}. Our main result below proves that, for some range of parameters $\be$, the sequence $(u_N, \dt u_N)_{N \in \N}$ converges to a non-trivial stochastic process $(u, \dt u)$ whose law is given by $\rhoo$ \eqref{gibbs10} at every time marginal. This process $u$ is hence formally interpreted as the solution to the following renormalized SdSG equation
\noi
\begin{align}
\dt^2 u + \dt u  +(1-\Dl)  u
+ \iota  \be 0 \cdot \sinh(\be u)  = \sqrt{2} \xi , 
\label{liouville3}
\end{align}

\smallskip

We first state a local well-posedness for~\eqref{liouville3} on the support of the Gaussian free field $\muu_1$.

\begin{theorem}\label{thm:0}
Let  $0 < \be^2 < \frac{6\pi}{5}$ and $\iota \neq 0$. Then,
the stochastic damped sinh-Gordon equation~\eqref{liouville3} is almost  surely locally  well-posed with respect to the Gaussian free field $\muu_1$ defined in~\eqref{series}. More  precisely, there exists an $\muu_1$-almost surely positive time $T>0$ and a process $(u,\dt u) \in C([0,T]; \mc H^{-\eps}(\T^2))$ for any  small $\eps > 0$ such  that the solution $(u_N, \dt u_N)$ to \eqref{liouville2} converges to $(u, \dt u)$ in $C([0,T]; \mc H^{-\eps}(\T^2))$ $\muu_1$-almost  surely as $N \to \infty$.
\end{theorem}

We note that in the defocusing case $\iota>0$, since the Gibbs measure $\rhoo$ is absolutely continuous with respect to $\muu_1$ (see Lemma~\ref{lem:gibbs}), Theorem~\ref{thm:0} also yields a local well-posedness result on the support of the Gibbs measure~$\rhoo$.

In the next result, restricting to the defocusing case $\iota>0$, we extend the local solution constructed in Theorem~\ref{thm:0} to a global-in-time solution by exploiting the invariance of the Gibbs measure, albeit for a more restricted range of the parameter $\be$.

\begin{theorem}\label{thm:main}
Let  $0 < \be^2 < \frac{2\pi}{11}$ and $\iota >0$. Then,
the stochastic damped sinh-Gordon equation~\eqref{liouville3} is almost  surely globally  well-posed with respect to the renormalized Gibbs  measure~$\rhoo$ defined in~\eqref{gibbs10} and the renormalized Gibbs measure $\rhoo$ is invariant under the dynamics. More  precisely, there exists a  process $(u,\dt u) \in C(\R_+; \mc H^{-\eps}(\T^2))$\footnote{Here, $C(\R_+; X)$ for a Banach space $X$ is the space of continuous functions from $\R_+$ to $X$, endowed with the compact-open topology.} for  any  small $\eps > 0$ such  that the solution $(u_N, \dt u_N)$ to \eqref{liouville2} converges to $(u, \dt u)$ in $C(\R_+; \mc H^{-\eps}(\T^2))$ $\rhoo$-almost  surely as $N \to \infty$. Moreover, for  each $t \ge 0$, the law of $(u(t), \dt u(t))$ is  given by  $\rhoo$.
\end{theorem}

To the best of our knowledge, Theorems~\ref{thm:0} and~\ref{thm:main} provide the first well-posedness results for the hyperbolic defocusing sinh–Gordon model. This resolves an open question raised in the closely related work~\cite{ORW} concerning the well-posedness (even locally in time) of the hyperbolic sinh-Gordon model. Indeed, in \cite[Remark 1.18]{ORW}, the authors state that they “do not know how to handle the hyperbolic sinh–Gordon equation for any $\be^2>0$.” See Section~\ref{SEC:proof} below for further discussion.

Our approach also covers the hyperbolic Liouville model \eqref{hypliouville}. Consider the following renormalized truncated Liouville dynamics
\begin{align}
\dt^2 u_N   + \dt u_N  +(1-\Dl)  u_N 
+ \iota \be \g_N \Q_{N}\big\{\exp\big(\be \Q u_N \big)\big\}   = \sqrt{2} \xi , 
\label{liouvilleUS}
\end{align} 
with $\iota >0$ (defocusing) and $\be \in \R$, and its associated renormalized Gibbs measure
\begin{align}
\label{gibbsNUS}
d \rhoo_{N, 2}(u, v) =Z_{N,2}^{-1}  \exp\Big( -\iota \g_N \int_{\T^2} \exp \big(\be \Q_{N} u\big) dx \Big) d \muu_1(u,v).
\end{align}
Here, $\Q_{N}$ is an approximation to the identity which has a non-negative convolution kernel. For example, given $\eta$ a smooth, non-negative, even function of compact support such that $\int_{\R^2}\eta(x) dx =1$, then for $N\in \N$, $\Q_{N}$ is defined by 
\begin{align*}
\Q_{N} h = \eta_{N} \ast_{x} h,
\end{align*}
where $\eta_{N}(x)= N^2 \eta(Nx)$. Notice that $\Q_{N}$ is not a projection onto finitely many frequencies unlike $\Pii_{N}$. Lastly, the renormalization constant $\g_N$ is as in~\eqref{gN}.\footnote{Changing the approximation procedure from $\Pii_N$ to $\Q_N$ in $\s_N$ only modifies $\g_N$ by an $O(1)$ factor.} By a variant of Lemma~\ref{lem:gibbs}, one shows the existence of a measure $\rhoo_2$ such that
\begin{align}
\lim_{N \to \infty} \rhoo_{N,2} = \rhoo_2,
\label{gibbs100}
\end{align}
in  the sense of total variation. 
With the measure in hand, we consider the associated invariant hyperbolic Liouville dynamics
\noi
\begin{align}
\dt^2 u + \dt u  +(1-\Dl)  u
+ \iota  \be 0 \cdot \exp(\be u) = \sqrt{2} \xi,
\label{liouvilleUS2}
\end{align}
which, as before, we interpret as the limit of the solution to the truncated renormalized Liouville equation~\eqref{liouvilleUS}.

\begin{theorem}\label{thm:2}
Let  $0 < \be^2 < \frac{6\pi}{5}$ and $\iota >0$. Then,
the stochastic damped Liouville equation~\eqref{liouvilleUS2} is almost  surely globally  well-posed with respect to the renormalized Gibbs  measure~$\rhoo_2$ defined in~\eqref{gibbs100} and the renormalized Gibbs measure $\rhoo_2$ is invariant under the dynamics. More  precisely, there exists a  process $(u,\dt u) \in C(\R_+; \mc H^{-\eps}(\T^2))$ for  any  small $\eps > 0$ such  that the solution $(u_N, \dt u_N)$ to \eqref{liouvilleUS} converges to $(u, \dt u)$ in $C(\R_+; \mc H^{-\eps}(\T^2))$ $\rhoo_2$-almost  surely as $N \to \infty$. Moreover, for  each $t \ge 0$, the law of $(u(t), \dt u(t))$ is  given by  $\rhoo_2$.
\end{theorem}

 Theorem~\ref{thm:2} therefore improves upon the parameter range \eqref{ORWbe} from \cite{ORW}. We point out that in both defocusing and focusing cases, we can establish local well-posedness for \eqref{liouvilleUS2} in the sense of Theorem~\ref{thm:0} for any $0<\be^2<\frac{6\pi}{5}$. 
The improved range of $\be^2$ in Theorem~\ref{thm:2} as compared to Theorem~\ref{thm:main} is due to a sign-definite structure hidden in \eqref{liouvilleUS2} which does not appear in \eqref{liouville}. In particular, relative to~\cite{ORW}, the use of an approximation procedure with a positive structure (namely, via the projectors $\Q_N$) is essential only in the globalization part of Theorem~\ref{thm:2}. A more detailed discussion of the differences between the two approaches is provided in Section~\ref{SEC:proof} below.

\subsection{Overview of the proofs} \label{SEC:proof}
First, we focus on the sinh-Gordon model \eqref{liouville3} and hence discuss the proofs of Theorem~\ref{thm:0} and~\ref{thm:main}. Then, we discuss the adjustments needed to obtain Theorem~\ref{thm:2} for the Liouville model \eqref{liouvilleUS2}.

In view of the mutual absolute continuity of $\rhoo$ and $\muu_1$ (see Lemma~\ref{lem:gibbs} below), we consider \eqref{liouville2} with initial data $(u_0, u_1)$ sampled from $\muu_1$.
Proceeding with a first order expansion as in \cite{BO96, DPD, ORW}, we have
\begin{align}
u_N = \Psi_{N} + v_N, \label{uNfull}
\end{align}
where $\Psi$ is the so-called stochastic convolution, $\Psi_{N}:= \Pii_{N}\Psi$ is the truncated stochastic convolution, and $v_{N}$ is a remainder.
\noi
The stochastic convolution $\Psi$ solves the following linear damped Klein-Gordon equation:
\begin{align*}
\begin{cases}
\dt^2 \Psi + \dt\Psi +(1-\Dl)\Psi  = \sqrt{2}\xi,\\
(\Psi,\dt\Psi) |_{t=0}=(u_0,u_1).
\end{cases}
%\label{sto_conv}
\end{align*}
Namely, $\Psi$ can be written as
\begin{align*}
\Psi(t)=\Psi(t; u_0,u_1,\o)=\dt \D(t) u_0 + \D(t) (u_0+u_1) + \sqrt{2} \int_{0}^{t} \D(t-t')d\B(t'),
%\label{stoconv1}
\end{align*}
where $\D(t)$ denotes the linear damped  Klein-Gordon  propagator  which is a Fourier multiplier operator with symbol
\begin{align}\label{prop1}
\D(t) = e^{-\frac{t}2}\frac{\sin(t \fbb{\nb})}{\fbb \nabla}, \quad t \in \R,
\end{align} 
where
\[ \fbb n = \big(\tfrac34 + |n|^2\big)^{\frac12} , \quad n \in \Z^2, \] and $\B$ is a cylindrical Wiener process on $L^2(\T^2)$: 
\begin{align*}
\B(t) = \sum_{n\in \Z^2} B_{n}(t) e_n,
\end{align*}
and where $\{B_{n}\}_{n\in \Z^2}$ is defined by $B_{n}(0)=0$ and $B_{n}(t) = \jb{ \xi , \ind_{[0,t]}\cdot e_n}_{t,x}$. Here, $\jb{\cdot ,\cdot }_{t,x}$ denotes the duality pairing on $\R_{+}\times \T^2$. Consequently, $\{B_{n}\}_{n\in \Z^2}$ is a family of mutually independent complex-valued Brownian motions such that $B_{-n}=\cj{B_n}$, $n\in \Z^2$, and normalized so that $\text{Var}(B_{n}(t))=t$.
The nonlinear remainder $v_N$ solves 
\noi
\begin{align}
\begin{split}
\dt^2 v_N   + \dt v_N  +(1-\Dl)  v_N & = - \iota \be \g_N \Pii_{ N} \big\{\sinh\big(\be \Pii_{ N} v_N + \be \Psi_N \big)\big\} \\
&=   - \iota \tfrac{1}{2}\be \g_N \Pii_{ N} \big\{ e^{\be \Pii_{N}v_N} e^{\be \Psi_{N}}   -e^{-\be \Pii_{N}v_{N}} e^{-\be \Psi_{N}}\big\},
\end{split}
\label{liouville4}
\end{align} 
with zero initial data, where $\g_N$ is as in \eqref{gN}.

We now explain further the multiplicative renormalization procedure. 
The main issue is that the term $e^{\be \Psi_{N}}$ does not have a well-defined limit as $N\to \infty$ for any $\be \in \R\setminus \{0\}$.
This can be traced to the Taylor expansion 
 \begin{align*}
e^{\be \Psi_{N}} = \sum_{k=0}^{\infty} \frac{\be^k}{k!} \Psi_{N}^{k}
\end{align*}
and the fact that the products $\Psi_{N}^{k}$, $k\geq 2$, do not converge as $N\to \infty$.
The remedy is to replace these powers by their Wick products $:\Psi_{N}^{k}:$ defined as
\begin{align}
: \Psi_{N}^{k}:  \stackrel{\text{def}}{=} H_{k}(\Psi_{N}; \s_N), \label{Wickk}
\end{align}
where $\s_{N}$ is as in \eqref{sN} and $H_{k}$ is the $k$th Hermite polynomial defined through the generating function
\begin{align}
e^{tx- \frac{\s^2}{2}t^2} = \sum_{k=0}^{\infty} \frac{t^k}{k!} H_{k}(x;\s). \label{generating}
\end{align}
The Wick products $:\Psi_{N}^{k}:$ in \eqref{Wickk} do have a non-trivial limiting behaviour as $N\to \infty$ and this suggests a renormalization of $e^{\be \Psi_{N}}$ as 
\begin{align*}
\Dr_{N}(t,x ;\be)  \stackrel{\text{def}}{=} \sum_{k=0}^{\infty} \frac{\be^k}{k!} :\Psi_{N}(t,x)^{k}:.
\end{align*}
By using \eqref{generating}, we find that 
\begin{align}
\Dr_{N}(t,x;\be) = e^{\be \Psi_{N}} e^{-\frac{\be^2}{2}\s_{N}} = \g_{N} e^{\be \Psi_{N}}. \label{GMC}
\end{align}
Note that from \eqref{sN} and \eqref{gN}, $\g_{N}\to 0$ as $N\to \infty$. 
 For $0<\be^2<4\pi$, $\Theta_N$ converges, in the sense of distributions, to a nontrivial limit (see Lemma~\ref{LEM:GMC}) commonly referred to as Gaussian multiplicative chaos in the mathematical physics literature. 
We refer the interested reader to~\cite{Bere, LRV, RV, Shamov} and references therein for more information on GMC.

Using the notation,
\begin{align}
\begin{split}
\Dr_{N}^{\pm}(t,x) &: = \Dr_{N}(t,x;\pm\be) ,
\end{split} 
\label{GMCcs}
\end{align}
we rewrite \eqref{liouville4} as 
\begin{align}
\dt^2 v_{N} + \dt v_N +(1-\Dl)v_N = -\iota \tfrac{1}{2} \be \Pii_{N}\big\{ e^{\be \Pii_{N}v_N} \Dr^{+}_{N}-e^{-\be \Pii_{N}v_N} \Dr^{-}_{N}\big\}. \label{liouville42}
\end{align}
Now, we recall the linear wave propagator $\S$, given by
\begin{align}
\label{S}
\S(t) = e^{-\frac{t}2}\frac{\sin(t |\nb|)}{|\nb|}, \quad t \in \R,
\end{align} 
where $\frac{\sin(t|\nb|)}{|\nb|}$ is  a convolution operator with {\it explicit} kernel given by
\begin{align}
\mc W_{\T^2}(t,x) = \frac{\ind_{B(0,|t|)}(x)}{\sqrt{t^2-|x|^2}},
\label{ker00}
\end{align}
for any $x \in \T^2$ and $|t| \le \pi$; see Subsection~\ref{sec:notations} for more details.
Moreover, the difference $\mc D(t) - \mc S(t)$, where $\mathcal{D}(t)$ is as in \eqref{prop1}, is smoothing in the sense of Lemma~\ref{LEM:diffprop} below.
Aiming to leverage the explicit formulation of the kernel~\eqref{ker00}, we write $v_N = X_N + Y_N$, where
\begin{align}
\begin{split}
X_N & = - \iota \tfrac{1}{2} \be \Pii_{ N} \I_{\textup{wave}} \big(  f^{+}( \Pii_{N}(X_N+Y_N)) \Dr^{+}_{N}- f^{-}(\Pii_{N}(X_N+Y_N)) \Dr^{-}_{N} \big),\\
Y_N & = - \iota \tfrac{1}{2} \be \Pii_{ N} \I_{\textup{KG} - \textup{wave}} \big(  f^{+}( \Pii_{N}(X_N+Y_N)) \Dr^{+}_{N}-  f^{-}( \Pii_{N}(X_N+Y_N)) \Dr^{-}_{N} \big),
\end{split}
\label{liouville5}
\end{align}
on $\R_+ \times \T^2$, where we defined
\begin{align}
f^{\pm}(y) = e^{\pm \be y}, 
\label{fcfs}
\end{align}
and  $\I_{\textup{wave}}$ and $\I_{\textup{KG} - \textup{wave}}$ are the Duhamel operators given by
\begin{align}
\begin{split}
\I_{\textup{wave}}(F)(t) & = \int_{0}^t \mc S(t-t') F(t') dt', \\
\I_{\textup{KG} - \textup{wave}}(F)(t) & = \int_{0}^t (\mc D - \mc S)(t-t') F(t') dt'
\end{split} \label{duhamels}
\end{align}
for any $t \in \R$. Our goal is to solve the system~\eqref{liouville5} uniformly in $N$ in some function space and obtain the convergence of the solution $(X_N,Y_N)$ to a well-defined limit as $N \to \infty$.
We point out that such a decomposition akin to \eqref{liouville5} first appeared in \cite{ORW} in the context of the defocusing hyperbolic Liouville model. However, the way we exploit the structure in the $X$-equation is different from \cite{ORW} and the main novel idea introduced in this paper, as we will discuss later.

In the parabolic setting~\cite{HKK1, HKK2}, a crucial feature of Gaussian multiplicative chaos is the inequality
\begin{align}
\|\jb \nb ^{-s} (fg)\|_{L^p} \les \|f\|_{L^\infty} \|g\|_{W^{-s,p}},
\label{pos_bdd}
\end{align}
whenever $g$ is a \textit{positive} distribution and $0\leq s\leq 1$; see Lemma~\ref{LEM:pos}. 

The estimate~\eqref{pos_bdd} yields a derivative gain compared to the standard product bound:
\begin{align*}
\|\jb \nb ^{-s} (fg)\|_{L^p} \les \|f\|_{W^{s,r}} \|g\|_{W^{-s,q}},
\end{align*}
whenever $1<p,q,r<\infty$ with $\frac1p = \frac1r + \frac1q$ and $\frac1r + \frac1q = \frac1p + \frac s2$; see~\cite{GKO}.\footnote{For instance, consider the case where $g$ corresponds to the GMC in the standard wave $L^2$ setting, with $p \approx q \approx 2$ and $r \approx \infty$. Then Lemma~\ref{LEM:GMC} yields $s \approx \frac{\be^2}{4\pi}$. Thus,~\eqref{pos_bdd} provides an approximate $\frac{\be^2}{4\pi}$-derivative gain compared to the standard product bound, which is non-negligible for large values of $\be^2$.} Our approach seeks to exploit this improved estimate~\eqref{pos_bdd}—or rather, a variation thereof adapted to the current wave setting; see~\eqref{kerbdd2} below.

\medskip

\noi
{\it $\bul$ An $L^\infty$-framework for a wave equation.} We consider the following damped wave equation:
\begin{align}
\begin{cases}
\dt^2 u + \dt u -\Dl  u = F(u) \\
(u, \dt u) |_{t = 0} = 0, 
\end{cases}
\qquad (t, x) \in \R_+\times\T^2,
\label{toy}
\end{align}
with a nonlinearity $F(u) = \Theta e^{\be u} $, for some positive distribution $\Theta$. Note that~\eqref{toy} essentially corresponds to the $X$-equation in~\eqref{liouville5}. Our goal is to construct a solution theory for~\eqref{toy} in a space $\mc X$ that embeds into $L^\infty(\R \times \T^2)$. 
In the parabolic setting of \eqref{parasinh} in \cite{Garban, ORW}, the Schauder estimates ensure that one can take $L^{\infty}_{T}W^{\frac{2}{p}+,p}_{x}$ for some $p\geq 2$.\footnote{Which is chosen in order to optimize the range of $\be^2$, leading to $\be^{2} <\frac{8\pi}{3+2\sqrt{2}}$.}  In the hyperbolic case, the weaker smoothing property does not seem to be enough to do this without introducing new ideas.

Our choice of space $\mathcal{X}$ is motivated by two considerations. First, to take advantage of the bound~\eqref{pos_bdd}, we require $e^{\be u}$ to belong to $L^\infty(\R \times \T^2)$. Secondly, $L^\infty$-type spaces are stable under exponentiation (i.e. $u \in L^\infty \LRA e^u \in L^\infty$), which is particularly useful for treating the exponential nonlinearities present in the problem.

In the current rough setting, $\Theta$ belongs to $L^2([0,1]; H^{-\frac{\be^2}{4\pi}-\eps}(\T^2))$ for any $\eps >0$, see Lemma~\ref{LEM:GMC}. Consequently, by the standard non-homogeneous linear estimate
\begin{align}
\|\I_{\textup{wave}}(u) \|_{L^\infty_t H^{s+1}_x} \les \|u\|_{L^1_t H^{s}_x},
\label{basic}
\end{align}
valid for any $s \in \R$, we can only expect $v$ to lie in $L^\infty_t H^{1-\frac{\be^2}{4\pi}-\eps}_x$. For every $\be ^2 >0$, this regularity remains strictly below the Sobolev embedding threshold $H^{1+\eps}(\T^2) \hookrightarrow L^\infty(\T^2)$, for any $\eps >0$.\footnote{For small values of $\be^2$, one may combine \eqref{basic} with standard refinements such as Strichartz estimates (see \cite{ORW}) to place $v$ in $L^\infty_x$ though only after some time averaging; that is, one obtains $v \in L^p_t L^\infty_x$ for some $1 \le p < \infty$.} Therefore, standard PDE techniques are insufficient in this regime, and new ideas are required. 

First, as the difference $\mathcal{D}-\mathcal{S}$ is smoothing of order $2$ (Lemma~\ref{LEM:diffprop}) this ensures, by Sobolev embedding, that $Y\in L^{\infty}_{t,x}$. Thus, it remains to consider the variable $X$.
To this end, we take advantage of the explicit formula~\eqref{ker00} for the convolution kernel of the multiplier $\mc S$ in the Duhamel operator $\I_{\textup{wave}}$, which is our main motivation for the decomposition \eqref{liouville5}. 
We essentially have
\[ \I_{\textup{wave}}(F)(t) = \int_{\R} K(t, t', \cdot) \ast_x F(t') dt',  \]
where $\ast_x$ denotes the spatial convolution, and the kernel $K$ is given by
\[ K(t,t',x) = e^{-\frac{t-t'}{2}} \frac{\ind_{B(0,|t-t'|)}(x) \ind_{[0,t]}(t')}{\sqrt{|t-t'|^2-|x|^2}}. \]
See Subsection~\ref{subsec:ker} and, specifically, \eqref{ker1b} and \eqref{duha3} for more precise definitions.
Our strategy is to evaluate the operator $\I_{\textup{wave}}$ in space-time anisotropic Sobolev spaces $\Ld^{s,b}_p(\R \times \T^2)$ defined by the norm
\[ \|u\|_{\Ld^{s,b}_p} = \| \jb \dt ^b \jb{\nb_x}^s u\|_{L^p_{t,x}} \]
for $s,b \in \R$ and $1 \le p < \infty$. We develop the relevant properties of these spaces in Subsection~\ref{subsec:spaces}. In particular, by Sobolev embeddings in both space and time,
\begin{align}
\Ld_p^{s,b}(\R \times \T^2) \hookrightarrow L^{\infty}(\R \times \T^2)
\label{embd}
\end{align}
whenever 
\begin{align}
s>\tfrac{2}{p} \quad \text{and} \quad b>\tfrac 1p. \label{LinftySob}
\end{align} See Lemma~\ref{LEM:Sobolev}.
 Thus, the space $\Ld_p^{s,b}(\R \times \T^2)$ with $(s,b)$ satisfying \eqref{LinftySob} is admissible for our purposes.\footnote{Strictly speaking, we use a version of $\Ld_p^{s,b}(\R \times \T^2)$ localised to a time interval $[0,T]$, for $T>0$, but we ignore this in the following discussion.} 

\begin{remark}\rm
\label{rmk:space}
We point out that the spaces $\Ld^{s,b}_p(\R \times \R^2)$ are not, in general, well suited to the study of wave equations, in the sense that the linear propagator $\mc S(t)$ is not expected to be bounded on them. For instance, in the Euclidean setting (replacing $\T^2$ with $\R^2$) and taking $s=b=0$ and $p \ge 4$, the boundedness of $\mc S(t)$ on $L^p([1,2] \times \R^2)$ comes at the expense of a loss of derivatives; see~\cite{GWZ}. Thus, we must crucially exploit the positivity properties of the nonlinear term $F(u)$ in order to study the problem~\eqref{toy} in these anisotropic Sobolev spaces.
\end{remark}

To estimate the operator $\I_{\textup{wave}}$ in $\Ld^{s,b}(\R \times \T^2)$, we need to control fractional derivatives of the kernel $K$  in both the $t$-and $x$-variables. The least integrable component of $K$ is the factor $||t-t'|- |x||^{-\frac12}$. Thus, after placing the derivatives $\jb \dt ^b$ and $\jb{\nb_x}^s$ directly on $K$, we expect a bound of the form
\begin{align}
\big|\jb{\dt} ^b \jb{\nb_x} ^s K (t,t',x)\big| \les | |t-t'| + |x| |^{-\frac12} ||t-t'| - |x||^{-\frac12-s-b}.
\label{kerbdd1}
\end{align}
In practice, however, the non-differentiability of the indicator functions appearing in $K$ prevent a direct computation of such derivatives. We address this technical issue via an approximation argument; see Subsection~\ref{subsec:ker}.

From \eqref{kerbdd1}, we expect an estimate of the form
\begin{align}
\|\I_{\text{wave}}(F)\|_{\Ld^{s,b}_p} \les \|\mc R \ast_{t,x} |F| \|_{L^p_{t,x}}, \quad s+b<\tfrac 12 \quad \text{and} \quad 1\leq p<\infty,
\label{kerbdd2}
\end{align}
where $\mc R$ denotes the integrable kernel
\begin{align}
\mc R(t,x) = | |t| + |x| |^{-\frac12} ||t| - |x||^{-\frac12-s-b}. \label{Rkernel}
\end{align}
We select a triplet 
\begin{align*}
(s,b,p) \approx (\tfrac13, \tfrac16, 6), 
% \label{triplet}
\end{align*}
so that both~\eqref{kerbdd2} and the embedding~\eqref{embd} apply. Then, specializing to $F(u) = \Theta e^{\be u}$,  \eqref{kerbdd2} implies
\begin{align}
\|\I_{\text{wave}}(F)\|_{\Ld^{s,b}_p} \les \|\mc R \ast_{t,x} \Theta \|_{L^p_{t,x}} e^{\be \|u\|_{\Ld^{s,b}_p}}.
\label{kerbdd3}
\end{align}
The bound \eqref{kerbdd3} (or mild refinements thereof) is sufficient to establish a satisfactory well-posedness theory for \eqref{toy}, \textit{irrespective of the sign of} $\iota$, provided we can control (a slight variant of) the quantity
\begin{align}
\|\mc R \ast_{t,x} \Theta \|_{L^p_{t,x}}.
\label{stoint}
\end{align}
We address the estimate for \eqref{stoint} next.

\begin{remark}\rm We mention here that, in order to avoid additional technicalities, we actually compute a space–time isotropic version of the anisotropic derivative 
$\jb{\dt} ^b \jb{\nb_x} ^s K (t,t',x)$;
see Lemma~\ref{lem:ker03}. 
Moreover, the right-hand side bound~\eqref{kerbdd1} (and hence the kernel $\mc R$) also contains additional less singular terms; see Proposition~\ref{prop:kery}. 
\end{remark}

\medskip

\noi
{\it $\bul$ On the modified GMC.} We now aim to control the quantity \eqref{stoint} for $\Theta=\Theta_N$, the GMC random field, with bounds uniform in $N$. This essentially reduces to analyzing the modified GMC expression
\begin{align}
\int_{\T^2}\int_{0}^{1}  \frac{\Dr_{N}(t',y;\be)}{| |t-t'|-|x-y||^{\frac 12+s+b} |x-y|^{\frac 12}} dt' dy.\label{modGMC}
\end{align}
Ignoring the hyperbolic singularity term,\footnote{That is, one would use the time integral to integrate out the hyperbolic singularity since $\frac12 + s + b < 1$—which would be valid if, for instance, $\Theta_N$ were constant in time.} we see that the modified GMC is roughly similar to
\begin{align*}
\jb{\nb_x}^{-\frac32} \Dr_{N}
\end{align*}
where we recall that the Bessel potential operator $\jb{\nb_x}^{-\al}$ has an integral kernel on $\T^2$ that behaves like $|x|^{-2+\al}$ modulo a smooth function; see \cite[Lemma 2.2]{ORSW1}. This spatial smoothing is the reason for the $L^p$-integrability of \eqref{modGMC} uniformly in $N$.

Broadly speaking, our estimates for \eqref{modGMC} rely on Kahane’s convexity inequality (Lemma~\ref{LEM:Kahane}) together with detailed structural properties of the GMC measure (Lemma~\ref{LEM:GMCmeas}), following the approach in \cite{ORW}. 
Unlike in \cite{ORW} though, handling the hyperbolic singularity in \eqref{modGMC} requires us 
 to use refined information on the \textit{space–time correlations} of the stochastic convolution
\begin{align}
\G_{N}(t_1,t_2,x_1,x_2): = \E[ \Psi_{N}(t_1,x_1)\Psi_{N}(t_2,x_2)] \label{GN}
\end{align}
where $t_1 \neq t_2$. 
Such sharp estimates were previously developed in~\cite{Zine3}; see Lemma~\ref{LEM:GN}. 

As for the numerology in Theorem~\ref{thm:0}, the condition $s+b<\frac 12$ is used to ensure that the hyperbolic singularity is locally integrable in time. We are then able to establish that the modified GMC object belongs to $L^{p}_{t,x}$ uniformly in $N$ for the natural range of parameters: $0<\be^2 <8\pi$, $1\leq p< \frac{8\pi}{\be^2}$, and
 $\frac{(p-1)\be^2}{4\pi}<\frac 12$.  
Specialising $p\approx 6$ yields the condition $\be^{2} <\frac{6\pi}{5}$ in Theorem~\ref{thm:0}.

\medskip

\noi
{\it $\bul$ The global well-posedness argument.} We now discuss our globalization argument leading to Theorem~\ref{thm:main}. The estimates \eqref{kerbdd2} and \eqref{kerbdd3} combined with the probabilistic estimates on the modified GMC object \eqref{stoint} are used to obtain the (uniform in $N$) local well-posedness 
for the truncated system $(X_N,Y_N)$ in \eqref{liouville5}, under the improved condition $0<\be^2 <\frac{6\pi}{5}$, thereby proving Theorem~\ref{thm:0}. It then remains to establish the global well-posedness with respect to the Gibbs measure  $\rhoo$, for which we require the restricted range $0<\be^2 <\frac{2\pi}{11}$.

The well understood strategy here is to apply Bourgain invariant measure argument \cite{BO94, BO96}. In the singular stochastic setting, this argument was detailed in \cite{ORTz} for the two dimensional singular stochastic wave equations with polynomial nonlinearity on compact manifolds; see also \cite{HM, GKOT}. In our setting of a singular problem with an exponential nonlinearity, we encounter two fundamental difficulties in attempting to apply Bourgain's invariant measure argument, as we now explain.

First, we note that in the setting of random initial data, Bourgain's invariant measure argument was abstracted in \cite[Theorem 6.1]{FTo}. The key inputs here are (i) a sub-critical local well-posedness theory which provides a (uniform) bound on solutions in a norm $\|\cdot \|_{\mathcal{Z}}$ for a short time $\tau( \|u_0 \|_{\mathcal{Z}})>0$ and (ii) exponential integrability for the norm $\|\cdot \|_{\mathcal{Z}}$ in the sense that  
\begin{align}
\E[ \exp( \|u_0\|_{\mathcal{Z}}^{\dl})] <\infty \label{Boexp}
\end{align}
for some $\dl>0$. These ingredients come together in an estimate for the probability that the solution measured in $\|\cdot \|_{C_{T}\mathcal{Z}}$ is larger than some threshold $M\gg 1$ is bounded above by 
\begin{align}
T\tau(M)^{-1} \Prob( \| u_0\|_{\mathcal{Z}} >M) \les T\tau(M)^{-1}  e^{-M^{\dl}}. \label{Boprob}
\end{align}
For polynomial nonlinearities, typically $\tau(M)^{-1}\les M^{k}$ for some $k\geq 1$, so that the exponential decay in \eqref{Boprob} ensures that the final probability can be made small.
For exponential-type nonlinearities, $\tau(M)^{-1}\lesssim e^{cM}$. Consequently, using \eqref{Boprob} requires either $\delta>1$, or $\delta=1$ with $c>0$ chosen sufficiently small. In particular, this is essentially the case studied in \cite{STz, Robert} for wave and Schr\"{o}dinger-type equations with exponential nonlinearities, respectively. Here, the previously cited works concern the \textit{non-singular} regime where the initial data is almost surely bounded and thus $e^{u_0}$ is controlled in $L^{\infty}$ and one can establish \eqref{Boexp} for any $\dl\leq 2$.

In the \textit{singular regime}, the random initial data no longer belongs to $L^{\infty}$ almost surely and one needs a renormalization in order to make sense of $e^{u_0}$. This is precisely the case for our setting where $e^{\be \Psi}$ is understood as the Gaussian multiplicative chaos and constructed probabilistically as a single object. 
 Then the intermittency phenomenon rules out the exponential integrability assumption~\eqref{Boexp} used in Bourgain’s invariant measure argument.
 %, and a different approach is therefore required.

Secondly, our local well-posedness theory for the system \eqref{liouville5} does not seem suitable for iteration from short to long time intervals. Indeed, to control the $X$ variable,
we use the space $\Ld^{\frac 13-,\frac{1}{6}-}_{6+}$. As discussed in Remark~\ref{rmk:space}, the linear operator $\S(t)$ is not expected to even bounded on this space, and, in any case, the Strichartz estimates for the wave equation \cite{GV,  KeelTao, GKO} impose too strong of a spatial regularity assumption to recover the full range of $\be^2$ in Theorem~\ref{thm:main}.

In order to resolve these issues, we instead partially adapt the strategy in \cite{OOTol, OOTol2} which was developed in order to extend Bourgain's invariant measure argument to settings where the Gibbs measure is singular with respect to the base Gaussian measure. See also \cite{Bring2, BDNY} for other such approaches. In our setting, the Gibbs measure and the base Gaussian measure are actually equivalent.
Nonetheless, these arguments are well suited to address the aforementioned issues. First for the stability argument—namely, the construction of a solution to the infinite-dimensional problem ($N=\infty$) from well-posedness in the finite-dimensional setting—we employ an auxiliary time-weighted norm, following~\cite{OOTol, OOTol2}. This allows us to construct the infinite dimensional solutions on a long-time interval in a {\it single step}, without appealing to an iteration over short time intervals. 
To exploit the gain coming from the additional smallness in the time-weighted norm, we need to introduce further modified GMC objects and keep track of their additional time-decay. We point out that the stability argument works for the entire range $0<\be^2 <\frac{6\pi}{5}$.

The main new part in our argument is in establishing the probabilistic a priori bounds which are uniform in $N$ and over an arbitrary time interval $[0,T]$. See Lemma~\ref{LEM:unifNMbds}. The approaches in \cite{OOTol, OOTol2, BDNY, LTW} are essentially based on the argument described above leading to \eqref{Boprob}. Namely, for $M\geq 1$ controlling the relevant norms of the enhanced data on $[0,T]$, one breaks the time interval $[0,T]$ into $O(\frac{T}{\tau(M)})$-many subintervals of width $\tau(M)$. The local well-posedness theory controls the norms of the nonlinear solutions on each such interval with high-probability and then one sums over all of the intervals using the invariance of the truncated measures under the truncated flows to ensure that the norms actually do not grow.

The condition for the total probability to be small is precisely \eqref{Boprob}. In our setting, however, in order to cover the full range of $\be^2$ in the local theory of Theorem~\ref{thm:0}—which we do recover in Theorem~\ref{thm:2} for the defocusing Liouville model—we have $\tau(M)^{-1}\approx M^{-0+}$, while intermittency implies that the exceptional probability of the enhanced data set decays at most at the \textit{polynomial rate} $M^{-\frac{1}{q}}$.%In our setting, at best we can hope to take $\tau(M)^{-1}\approx M^{\frac{1}{q}+}$ for any fixed $q>6$ (see \eqref{timegain2}),

Instead, we observe that the right-hand sides of \eqref{liouville5} are \textit{only} functions of the full truncated solution $u_N$ which preserves the truncated measures $\rhoo_{N}$, so that $X_{N}=X_{N}[u_N]$ and $Y_{N}=Y_{N}[u_N]$. 
To control the $\Ld^{\frac{2}{p}+,\frac{1}{p}+}_{p}([0,T])$ norm of $X_{N}$, it suffices to obtain high-probability bounds for both $X_{N}$ and $Y_{N}$ in $L^{\infty}_{T,x}$. For $Y_{N}$, a straightforward argument based on Minkowski's inequality, the two degrees of smoothing enjoyed by $Y_{N}$, and invariance yields control of the relevant norms, including $L^{\infty}_{T,x}$.

In establishing that $X_{N}\in L^{\infty}_{T,x}$, we cannot directly apply the Kahane approach to bound the corresponding moments, as was done for \eqref{modGMC}, since $u_{N}$ is not Gaussian. This again suggests exploiting invariance. To do so, the additional time integral in the Duhamel formula forces us to first apply Minkowski's inequality in order to move it outside the expectation. In the worst case, the remaining spatial integral must be used to compensate for the hyperbolic singularity. We extract as much as possible at this stage by establishing moment bounds for the GMC measure on thin annuli; see Lemma~\ref{LEM:GMCannuli}, which may be of independent interest. To further weaken the singularity, we pass to an $L^{12+}$-theory for the GMC\footnote{Rather than the $L^{6+}$-theory we use for Theorem~\ref{thm:0}.}, leading to the sharper restriction $\be^{2}<\frac{2\pi}{11}$; see Lemma~\ref{LEM:PuN}.

This overall strategy proves to be more effective in the defocusing hyperbolic Liouville model, which we discuss next.

\medskip

\noi
{\it $\bul$ On the hyperbolic Liouville model.} 
The argument in \cite{ORW} for studying \eqref{liouvilleUS} locally-in-time begins similar to ours for \eqref{liouville} involving a refined first order expansion: 
\begin{align*}
u_{N} = \Psi_{N} + X_{N} +Y_{N}
\end{align*} 
where $X_{N}$ and $Y_{N}$ are unknowns satisfying the system:
\begin{align}
\begin{split}
X_{N}& = -\iota \be \Q_{N} \mathcal{I}_{\text{wave}}[ \Dr_{N} e^{\be \Q_{N}(X_N +Y_N)}],\\
Y_{N}& = -\iota \be \Q_{N} \mathcal{I}_{\text{KG-wave}}[ \Dr_{N} e^{\be\Q_{N}(X_N +Y_N)}],
\end{split} \label{XYorw}
\end{align}
with zero initial data. As in our approach, the authors in \cite{ORW} seek to exploit the positivity of $\Dr_{N}$. However, since $X_{N}$ is not expected to lie in $H^{1+\eps}_{x}(\T^2)$, they observe that by making the \textit{defocusing} assumption $\iota>0$, it follows that $X_{N}\leq 0$ a.s. i.e. in the defocusing case, $X_N$ has a \textit{good} sign so that $e^{\Q_{N}X_{N}} \in L^{\infty}(\T^2)$. To ensure this sign, key roles are played by (i) the defocusing assumption $\iota>0$, (ii) the positivity of $\Dr_{N}$, (iii) the choice of $\Q_{N}$ which is crucial as it preserves the sign of its input function;
 namely, if $h\geq 0$, then $\Q_{N}h \geq 0$, and (iv) the fact that the kernel for $\mathcal{S}(t)$ in \eqref{S} is non-negative, see \eqref{ker00}.
 Then, a uniform-in-$N$ local well-posedness is established for the system $(X_N,Y_N)$ using \eqref{pos_bdd} and Strichartz estimates, assuming that $\iota>0$ and $\be^2 < \frac{32-16\sqrt{3}}{5}\pi \approx 0.86\pi$.
This approach no longer works in the \textit{focusing} case ($\iota <0$) because now $X\geq 0$ a.s. and the authors in \cite{ORW} left this situation open.

In contrast, our approach applies equally well in the focusing case, thereby resolving this open problem, and it even yields the improved range $\be^{2} <\frac{6\pi}{5}$. At the level of the local theory, the key difference lies in how we exploit the structure of the equation for $X_{N}$. We rely on (i) the positivity of $\Dr_{N}$ and (ii) the explicit formula for the kernel of $\mathcal{S}(t)$, and \textit{not} on its sign-definite structure. This enables us to carry out physical-space estimates on the kernel and derive \eqref{kerbdd2} without any assumption on the sign of $\iota$. In particular, our local-in-time analysis also applies with the frequency projections $\Pii_{N}$.

For the global well-posedness result in Theorem~\ref{thm:2}, we likewise improve on the range obtained in \cite{ORW}. Moreover, unlike Theorem~\ref{thm:main}, we do not need to impose any additional restriction on $\be^2$ beyond what is already required for the local theory, so that the global result matches the range $\be^{2} <\frac{6\pi}{5}$. The main reason is that, in obtaining uniform-in-$N$ and long-time bounds for $X_N$ and $Y_{N}$, we can exploit the sign structure of $\be X_{N}$ to observe that our approach for bounding the norm $\|X_{N}\|_{\Ld^{2/p+,1/p+}_{p}([0,T])}$ does \textit{not} involve the $L^{\infty}_{T,x}$-norm of $X_{N}$. Consequently, we do not need an analogue of Lemma~\ref{LEM:PuN}, which led to the tighter restriction $\be^2 <\frac{2\pi}{11}$. We refer to Remark~\ref{rmk:liouvilleUS} for further details. In summary, while our local well-posedness theory for the Liouville model does not rely on a sign-definite structure, the global well-posedness theory does exploit it.
\begin{remark}\rm \label{RMK:gwp}
Note that we do not expect the range $\be^{2}<\frac{2\pi}{11}$ to be optimal for Theorem~\ref{thm:main} to hold, even within the current framework for the local theory of Theorem~\ref{thm:0}. Modest improvements to this range may be possible but would require additional technical inputs that do not seem to overcome the main obstructions discussed above.
\end{remark}

\begin{remark}\rm \label{RMK:focusing}
Note that in the focusing case $\iota<0$, the Gibbs measure $\rhoo_{N}$ in \eqref{gibbsN} is not normalizable, since $Z_{N}\to \infty$ as $N\to \infty$ when $\iota<0$; see \cite[Proposition A.1]{ORTzW}. Thus, while Theorem~\ref{thm:0} establishes a local-in-time existence and uniqueness theory for \eqref{liouville} with $\iota<0$, we do not expect an analogue of Theorem~\ref{thm:main} in this regime.
\end{remark}

The remainder of this paper is organized as follows. In Section~\ref{sec:preliminaries}, we introduce notation, present key formulas, and study properties of the function spaces $\Ld_{p}^{s,b}$. The main focus of Section~\ref{sec:det} is the derivation of kernel estimates in Section~\ref{subsec:ker}, culminating in the main estimate of Proposition~\ref{prop:kery}. Section~\ref{sec:sto} contains the principal probabilistic input, concerning properties of the GMC and the modified GMC. Finally, in Section~\ref{sec:wp}, we state and prove the main local existence and uniqueness result for~\eqref{liouville} (Proposition~\ref{PROP:LWP}), which leads to the proof of Theorem~\ref{thm:0}. We then complete the proofs of Theorems~\ref{thm:main} and~\ref{thm:2} by employing a modified version of Bourgain's invariant measure argument.

\section{Preliminaries}\label{sec:preliminaries}
\subsection{Notations}\label{sec:notations} We write  $ A \les B $ to  denote an estimate  of the form  $A \leq CB$.  
 Similarly, we write $A \sim B$ to denote $A \les B $ and $ B\les A$ and use $A \ll B$ 
when we have $A \leq c B$ for small $c > 0$. We may write  $A \les_\ta B$  for  $A \leq C B$ with $C = C(\ta)$  
if we want to emphasize the dependence of the implicit constant on some parameter $\ta$.  
We   use $C, c > 0$, etc.~to denote  various constants  whose  values   may change  line by   line.

For $x, y \in \T^2 \cong [-\pi, \pi)^2$,  set
\[ |x - y|_{\T^2  } =  \min_{k \in  2\pi  \Z^2} |x - y + k|_{\R^2},  \]
 where   $|\cdot |_{\R^2}$   denotes the   standard Euclidean   norm   on $\R^2$. When there is no confusion, 
we  simply  use $|\cdot|$  for both  $|\cdot|_{\T^2}$ and $|\cdot |_{\R^d}$ for any $d \ge 1$.

\medskip

\noi
{\bf $\bullet$ Fourier transforms.}\quad  We respectively denote  by  $\Ft_{\R^2}$ and $\Ft_{\R^2}^{-1}$
the spatial   Fourier transform on $\R^2$ and its  inverse,
which  are given by
\noi
\begin{align*}
\Ft_{\R^2} (f)(\xi) = \frac{1}{2\pi} \int_{\R^2} f(x) e^{-i \xi \cdot x} dx, \qquad \Ft_{\R^2} ^{-1}(f)(x) = \frac{1}{2 \pi} \int_{\R^2} f(\xi) e^{i \xi \cdot x} d \xi.
%\label{fourier1}
\end{align*}

\noi
%where $a \cdot b$ denotes the usual scalar product of the vectors $a$ and $b$ in $\R^2$. 
We   then define the convolution product on  $\R^2$ via
\begin{align*}
(f \ast g)(x) =   \frac{1}{2 \pi}   \int_{\R^2}  f(y)   g(x-y) dy
%\label{fourier2}
\end{align*}
such that $\Ft_{\R^2} (f \ast g) = \Ft_{\R^2} (f) \Ft_{\R^2}(g)$. Similarly, 
the  Fourier   transform $\Ft_{\T^2}$  on the   torus $\T^2$ is given by
\begin{align*}
\Ft_{\T^2} (f)( n  )   = \int_{\T^2} f(x)   \cj{e_n (x)} dx,  \quad n \in \Z^2,
%\label{fourier3}
\end{align*}

\noi
where 
\begin{align*}
e_n (x)= \frac{1}{2 \pi}e^{i n \cdot x}. 
%\label{exp0}
\end{align*}

\noi
Then, the Fourier inversion formula reads as 
\begin{align*}
f(x) = \sum_{n \in \Z^2} \Ft_{\T^2} (f)(n)  e_n (x),
%\label{fouriersum}
\end{align*}

\noi
Define the convolution product on $\T^2$ via 
\begin{align*}
(f \ast g)(x) = \frac{1}{2 \pi} \int_{\T^2} f(y) g(x-y) dy,
%\label{fourier5}
\end{align*}
such that $\Ft_{\T^2} (f \ast g) = \Ft_{\T^2} (f) \Ft_{\T^2}(g)$.

In this work, when considering space-time functions of the form $u : \R \times \T^2 \to \R$ and $v : \R \times \T^2 \to \R$, we will use subscripts to indicate the variable with respect to which the convolution is taken. For instance, we write 
\begin{align*}
(u \ast_{t,x} v) (t,x) & = \frac{1}{(2\pi)^3} \int_{\R \times \T^2} u(t-t', x-y) v(t',y) dt' dy, \\
(u \ast_{x} v)(t,x) & = \frac{1}{(2\pi)^2} \int_{\T^2} u(t, x-y) v(t,y)
\end{align*}
for any $(t,x) \in \R \times \T^2$.

In the following, 
when it is clear from the context, 
 we write $\Ft(f)$ and $\ft f$ for the Fourier transform of a function $f$ defined either on 
  $\R^2$,  $\T^2$, and $\R \times \R^2$. 
  A similar comment applies to $\F^{-1}(f)$ and $\widecheck f$.

\medskip

Next, we recall the Poisson summation formula; see \cite[Theorem 3.2.8]{Grafakos1}. 
Let $f \in L^1 (\R^2)$ such that (i)
there exists  $\eta >0$ such that 
 $|f(x)| \les \jb{x}^{-2 - \eta}$ for any $x \in \R^2$, and (ii)~$\sum_{n \in \Z^2} |\F_{\R^2} (f) (n)| < \infty$. Then, we have
\begin{align}
\sum_{n \in \Z^2} \F_{\R^2} (f) (n) e_n (x) = \sum_{k \in \Z^2} f(x + 2 \pi k)
\label{poisson}
\end{align}

\noi
for any $x \in \R ^2$.

In the following, we define a number of frequency projections. Let $\varphi\in C_{c}^{\infty}([0,1])$ be a smooth symmetric bump function such that $\varphi(\xi)=1$ for $|\xi| \leq \frac{5}{4}$ and $\supp \varphi \subseteq \{ |\xi|\leq \frac{8}{5}\}$. We then define 
$\phi(\xi):= \varphi(\xi) - \varphi(2\xi)$ and for $N\in 2^{\N_0}$, we set 
\begin{align*}
\phi_{1}(\xi)  :=  \varphi(\xi), \qquad 
\phi_{N}(\xi) : = \phi(\tfrac{\xi}{N}) \quad \text{for} \quad N\in 2^{\N}.
\end{align*}
Next, for dyadic numbers $N\in 2^{\N_{0}}$, we define the following Littlewood-Payley frequency projectors: $\F_{x}( \P_{N}u)(n)  =\phi_{N}(n) \ft u(n)$
for $n\in \Z^2$.  Note that we have 
\begin{align*}
\sum_{N\in 2^{\N_0}}\P_{N} =   \text{Id}.
\end{align*}

%We recall some results regarding the projectors $\Pii_{N}^{+}$: 
%for any $s\geq 0$ and $1<p<\infty$, 
%\begin{align}
%\| \Pii_{N}^{+}f\|_{W^{s,p}(\T^2)} \les N^{s} \| \Pii_{N}^{+}f\|_{L^{p}(\T^2)} \label{QNbd}
%\end{align}
%and
%for any $0<\s-s<1$ and $1<p<\infty$, it holds that 
%\begin{align}
%\|( \Pii_{N}^{+}-\text{Id})f\|_{W^{s,p}(\T^2)} \les N^{s-\s}\|f\|_{W^{\s,p}(\T^2)}. \label{QNdiff}
%\end{align}
%Both \eqref{QNbd} and \eqref{QNdiff} follow from transference and the Mihlin-H\"{o}rmander theorem; see \cite[(5.25)]{ORW} for a proof of \eqref{QNdiff}.
%\medskip

\noi
{\bf $\bullet$ Poisson's formula.}  Consider   the  Fourier   multiplier on $\R^2$   given by   $ \frac{\sin( t |\nb|)}{|\nb|}$   for  $t \in \R_+$. Then,   from \cite[(27) on p.\,74]{Evans}, it   admits the following physical   space   representation as   a convolution kernel:
\begin{align}
\frac{\sin( t |\nb|)}{ |\nb| }  f  =   \mc W_{\R^2}(t, \cdot) * f
\label{poisson2}
\end{align}
\noi
for any $t \in \R$ and where the wave kernel $\W_{\R^2}$ on $\R^2$ is defined as
\begin{align}
\mc W_{\R^2}(t,x) = \frac{\ind_{B(0,|t|)}(x)}{\sqrt{t^2-|x|^2}}
\label{poisson3}
\end{align}
for any $(t,x)\in  \R\times \R^2$. The identity \eqref{poisson2} is often referred to as Poisson's formula. Note  that for a fixed function $f$, the function   $g = \frac{\sin( t |\nb|)}{|\nb|} f  $   is   the   solution  to the  Cauchy  problem   for  the   linear    wave  equation:
\begin{align*}
\begin{cases}
\dt^2 g   -  \Dl  g     =   0,\\
(g, \dt g) |_{t   =   0} = (0, f), 
\end{cases}
\qquad (t, x) \in \R_+\times\R^2.
\end{align*}
We denote by $\mc W_{\T^2}$, the function given by
\begin{align}
\mc W_{\T^2}(t,x) = \sum_{m\in \Z^2} \mc W_{\R^2} (t,x+2\pi m)
\label{poisson4}
\end{align}
for any $(t,x) \in \R \times \T^2 \cong \R \times [-\pi, \pi)^2$. Then, by Poisson's summation formula~\eqref{poisson} and an approximation argument; see for instance~\cite[Lemma 2.5]{ORW}, we have that
\begin{align*}
\sum_{n \in \Z^2} \frac{\sin(t |n|)}{|n|} \ft f(n) e_n = \mc W_{\T^2} \ast_x f.
%\label{poisson5}
\end{align*}
Note that for $|t| < \pi$, the spatial of support of $\mc W_{\R^2}(t,\cdot)$ is included in $B(0,\pi) \subset [-\pi, \pi)^2 \cong \T^2$ and hence $\mc W_{\R^2}(t,\cdot) \equiv \mc W_{\T^2}(t, \cdot)$. 

\medskip

\noi
{\bf $\bullet$ Fractional derivations.}
Let $d \in \N$ and $\Dl_{\R^d}$ be the Laplacian on $\R^d$. Fix $0 < \s < 1.$ Then, we have the following standard formula:
\begin{align}
(-\Dl_{\R^d})^{\frac{\s}{2}}f(z) = c_{d,b} \int_{\R^d} \frac{f(z+h)-f(z)}{|h|^{d+\s}}dh
\label{dtb}
\end{align}
for any $z \in \R^d$ and any smooth function $f \in \mc S(\T^d)$. See for instance~\cite[Theorem 1]{Stinga}.

By~\eqref{dtb}, a standard approximation procedure and the Poisson summation formula; see for instance~\cite[Lemma 2.5]{ORW}, we have the following integral representation for the Laplacian $\Dl_{\R \times \T^d}$ on $\R \times \T^d$:
\begin{align}
(-\Dl_{\R\times\T^d})^{\frac{\s}{2}}f(t,x) = c_{d,b} \int_{\R \times \T^d} (f(t+h_1, x+h_2) - f(t,x)) \mc K_{d,\s}(h_1,h_2) dh_1 dh_2
\label{dtb2}
\end{align}
for any $(t,x) \in \R \times \T^d$ and $f \in \mc S(\R \times \T^d)$. Here, the kernel $\mc K_{d,\s}$ is given by
\begin{align}
\mc K_{d,\s}(t,x) = \sum_{m \in \Z^d} \frac{1}{| (t,x + 2\pi m)|^{1+d +\s}} \label{Laplaceker}
\end{align}
for any $(t,x) \in \R \times \T^d \cong \R \times [-\pi, \pi)^d$.

\medskip

\noi
{\bf $\bullet$ Sobolev spaces}  Given $s\in \R$, the $L^2$-based Sobolev space $H^{s}(\T^2)$ is defined by norm 
\begin{align*}
\| f\|_{H^s} = \| \jb{\nb}^{s}f\|_{L^{2}_{x}} = \| \jb{n}^{s} \ft f(n)\|_{\l^2_{n}}.
\end{align*}
We use the notation $\mc H^{s} (\T^2) : = H^{s}(\T^2)\times H^{s-1}(\T^2)$.
Given $s\in \R$ and $1\leq p\leq \infty$, we defined the $L^p$-based Sobolev spaces $W^{s,p}(\T^2)$ by the norm 
\begin{align*}
\| f\|_{W^{s,p}} = \| \jb{\nb}^{s}f\|_{L^p} =\| \mathcal{F}^{-1}[ \jb{n}^{s} \ft f(n)] \|_{L^p_x}.
\end{align*}
We define the (time) Sobolev spaces $W^{s,p}(\R)$ similarly.

\subsection{Function spaces, embeddings and linear estimates}\label{subsec:spaces}

Recall that $\D(t)$, defined in \eqref{prop1} is the propagator for the linear damped Klein-Gordon and $\S(t)$, defined in \eqref{S}, is the propagator for the linear wave equation.
We recall the following lemma regarding the smoothing property of the difference operator $\D(t) -  \S(t)$.

\begin{lemma} \cite[Lemma 2.6]{ORW}\label{LEM:diffprop}
Let $t\geq 0$ and $s\in \R$. Then: \\
\noi
\textup{(i)} The operator $\D(t) -  \S(t)$ is bounded from $H^{s}(\T^2)$ to $H^{s+2}(\T^2)$.\\
\textup{(ii)}The operator $\dt(\D(t) -  \S(t))$ is bounded from $H^{s}(\T^2)$ to $H^{s+1}(\T^2)$.
\end{lemma}

The key working function space for us is the following anisotropic Sobolev spaces.

\begin{definition} \label{DEF:spaces2}
\rm 

Let $(s, b) \in \R^2 $ and $1 \le p \le \infty$. We define the space  $\Ld^{s,b}_{p}(\R \times \T^2)$   as the completion   of $\mathcal{S} (\R \times \R^2)$   under the   norm
\noi
\begin{align*}
\|u\|_{\Ld^{s,b}_{p}(\R  \times  \T^2)} :=   \big\|  \jb{\nb_x}^s \jb \dt ^b u \big\|_{L^p_{t,x}(\R\times \T^2)}.
% \label{ld1}
\end{align*} 
Given an interval $I\subset \R$, we define the restriction $\Ld_{p}^{s,b}(I)$ of the space $\Ld_{p}^{s,b}(\R\times \T^2)$ onto the interval $I$ via the norm:
\begin{align*}
\| u\|_{\Ld^{s,b}_{p}(I)} : = \inf\{ \| v\|_{\Ld^{s,b}_{p}} \, : \, v\vert_{I\times \T^2} =u\}.
\end{align*}
When $I=[0,T]$ for some $T>0$, we use the shorthand notation $\Ld^{s,b}_{p}(T)$ for $\Ld_{p}^{s,b}([0,T])$.
\end{definition}

In the remainder of this section, we state and prove some basic results regarding these anisotropic Sobolev spaces, paying particular attention to their properties under time localisation. In the special case $p=2$, a number of these results are well-known and can be proved using the Plancherel theorem and Fourier analytic arguments; see for instance \cite{TAO}. In our setting, we need to work with $p\neq 2$. In view of Subsection~\ref{subsec:ker}, we use a physical side approach for these estimates. We will extensively use the formula \eqref{dtb} and the the following equivalence of norms:
\begin{align}
\| u\|_{\Ld^{s,b}_{p}(I)} \sim \| u\|_{\Ld^{s,0}_{p}(I)} + \| (-\dt^{2})^{\frac{b}{2}} u\|_{\Ld^{s,0}_p (I)} \label{dtnorm}
\end{align}
for any interval $I\subseteq \R$. Note that \eqref{dtnorm} follows from the Mikhlin-H\"{o}rmander theorem \cite[Theorem 6.2.7]{Grafakos1}.

\begin{lemma}\label{LEM:timegain}
Let $1<p<\infty$, $s\in \R$ and $b_1\leq b_2 <\frac{1}{p}<b_3$. Let $0<\tau \leq 1$. Then, for all $u\in \Ld_{p}^{s,b_2}$ and $v\in \Ld^{s,b_2}_{p}([0,\tau])$, it holds that
\begin{align}
\| \varphi(t/\tau) u\|_{\Ld^{s,b_1}_{p}} \les \tau^{b_2 -b_1} \|u\|_{\Ld^{s,b_{2}}_{p}} \quad \text{and} \quad \|v\|_{\Ld_{p}^{s,b_1}([0,\tau])} \les \tau^{b_2-b_1} \|v\|_{\Ld_{p}^{s,b_2}([0,\tau])}\label{timegain1}
\end{align}
and for all $v\in \Ld^{s,b_3}_{p}$, 
\begin{align}
\| \varphi(t/\tau) v\|_{\Ld^{s,b_3}_{p}(\R)} \les \tau^{\frac 1p - b_3}\| v\|_{\Ld_{p}^{s,b_3}(\R)}. \label{timegain3}
\end{align}
Moreover, if $u(0,x)= 0$ for all $x\in \T^2$, then 
\begin{align}
\| u\|_{\Ld^{s,b_1}_{p}(I)} \les |I|^{b_2-b_1} \|u\|_{\Ld^{s,b_2}_{p}(I)}. \label{timegain2}
\end{align}
for any $\frac{1}{p}<b_1\leq b_2 <\min(b_1+\frac{1}{p},1)$ and closed interval $I\subset\R$ such that $0<|I| <1$.
\end{lemma}

\begin{proof}
The first estimate in \eqref{timegain1} follows from direct computations using the fractional Leibniz rule; see Lemma~\ref{LEM:prod} (i) (in particular, the proof of \eqref{timegain2} is a refinement of this argument). The second inequality in \eqref{timegain1} then follows from the first one in \eqref{timegain1} since for any $u\in \Ld^{s,b_2}_{p}(\R)$ such that $u\vert_{[0,\tau]\times \T^2} = v$, and $\varphi(t/\tau)u\vert_{[0,\tau]\times \T^2} =v$. Here, $\varphi$ is a smooth bump function defined in Subsection~\ref{sec:preliminaries} which equals $1$ on a neighborhood of the origin and is supported on a slightly larger set. Thus, 
\begin{align*}
  \|v\|_{\Ld_{p, \tau}^{s,b_1}} \leq \| \varphi(t/\tau)u\|_{\Ld^{s,b_1}_{p}} \les \tau^{b_2-b_1}\| u\|_{\Ld^{s,b_1}_{p}}.
\end{align*}
As $u$ was an arbitrary extension of $v$, the result follows. 
Next, \eqref{timegain3} is a consequence of the fact that $W^{b_3,p}(\R)$ is an algebra since $b_3>\frac 1p$ and \eqref{dtnorm}.

It remains to show \eqref{timegain2}.
To this end, it is enough to assume that $I=[0,\tau]$ for some $0<\tau<1$ and to establish
\begin{align}
\| \jb{\dt}^{b_1}[ \varphi(t/\tau) f]\|_{L^{p}(\R)} \les \tau^{b_2-b_1} \| \jb{\dt }^{b_2}f\|_{L^{p}(\R)}. \label{timegain21}
\end{align}
for $f\in C_c^{\infty}(\R)$ satisfying $f(0)=0$.
By Sobolev embedding, 
\begin{align}
\| \varphi(t/\tau) f\|_{L^{p}(\R)} \les \tau^{\frac{1}{p}} \|f\|_{L^{\infty}} \les \tau^{\frac{1}{p}}\|f\|_{W^{b_2,p}}. \label{timegain22}
\end{align}
Next, using \eqref{dtb}, we write 
\begin{align}
(-\dt^2)^{\frac{b_1}{2}} \big[ \eta(t/\tau)f](t)  =c_{1,b_1} \int_{\R} \frac{[\varphi(\tfrac{t+h}{\tau})-\varphi(\tfrac{t}{\tau})]f(t+h)}{|h|^{1+b_1}}dh + c_{1,b_1}\varphi(\tfrac{t}{\tau}) (-\dt^2)^{\frac{b_1}{2}}f(t).
\label{timegain20}
\end{align}
Taking the $L^p(\R)$ norm and using the triangle inequality leaves us to control each of these terms separately. The second term in \eqref{timegain20} is controlled using H\"{o}lder's inequality, which requires $b_2-b_1<\frac 1p$, and Sobolev embedding: 
\begin{align}
\begin{split}
\|\varphi(\tfrac{t}{\tau}) (-\dt^2)^{\frac{b_1}{2}}f(t)\|_{L^{p}}& \les \|\varphi(\tfrac{t}{\tau})\|_{L^{ \frac{1}{b_2-b_1}  }}  \| (-\dt^2)^{\frac{b_1}{2}}f\|_{L^{\frac{p}{1-p(b_2-b_1)}}}  \les \tau^{b_2-b_1} \| f\|_{W^{b_2,p}}.
\end{split}
\label{timegain23}
\end{align}
Now we control the first term in \eqref{timegain20}. By a change of variables, we  have 
\begin{align*}
\int_{\R} \frac{[\varphi(\tfrac{t+h}{\tau})-\varphi(\tfrac{t}{\tau})]f(t+h)}{|h|^{1+b_1}}dh  = \tau^{-b_1} \int_{\R}\frac{[\varphi(\tfrac{t}{\tau}+h)-\varphi(\tfrac{t}{\tau})]f(t+\tau h)}{|h|^{1+b_1}}dh  =: \1 (t)+ \II(t),
\end{align*}
where the integral in $h$ in $\1(t)$ is restricted to $|h|\les 1$ and the integral in $\II(t)$ is restricted to $|h|\gg 1$. 
To estimate $\1(t)$, we note that $|t|\les \tau$ since otherwise the integrand vanishes due to the compact support of $\varphi$. In particular, $|t+\tau h|\les \tau$. 
Using that $f(0)=0$, and Morrey's inequality \cite[Theorem 8.2]{Hitchhiker}, we have
\begin{align}
|f(t+\tau h)| = |f(t+\tau h)-f(0)| \les \tau^{b_2-\frac 1p}  \sup_{  \substack{ t_1,t_2 \in [0,5\tau] \\ t_1\neq t_2} } \frac{| f(t_1)-f(t_2)|}{|t_1-t_2|^{b_2-\frac{1}{p}}}  \les \tau^{b_2-\frac 1p} \|f\|_{W^{b_2,p}}. \label{timegainf}
\end{align}
Meanwhile, the fundamental theorem of calculus gives
\begin{align}
|\varphi(\tfrac{t}{\tau}+h)-\varphi(\tfrac{t}{\tau})| \les |h| \int_{0}^{1} |(\dt \varphi)(\tfrac{t}{\tau}+sh)|ds.
\label{ftc}
\end{align}
Thus, by 
Minkowski's inequality, \eqref{timegainf} and a change of variables, we have  
\begin{align}
\begin{split}
\| \1(t)\|_{L^p}  &\les \tau^{-b_1}  \tau^{b_2-\frac{1}{p}} \| f\|_{W^{b_2,p}} \bigg( \int_{|h|\les 1}|h|^{-b_1}dh\bigg) \sup_{y\in \R} \| (\dt \varphi)( \tfrac{t}{\tau}+y)\|_{L_t^{p}}  \\
 & \les \tau^{b_2-b_1}\|f\|_{W^{b_2,p}}.
 \end{split} \label{timegainIbd} 
\end{align}
Now for $\II(t)$, if $|t+\tau h|\les \tau$, the same argument as for $\1(t)$ applies using \eqref{timegainf} and Minkowski's inequality and using
\begin{align*}
\sup_{h\in \R}\| \eta(\tfrac{t}{\tau}+h)-\eta(\tfrac{t}{\tau})\|_{L^{p}}\les \tau^{\frac 1p}
\end{align*}
 instead of \eqref{ftc}. 
 Thus, it remains to consider when $|t+\tau h|\gg \tau$. Here we have $\varphi(\tfrac{t}{\tau}+h)=0$.
 Then, as in \eqref{timegainf}, we have
 \begin{align*}
|f(t+\tau h) | \les |t+\tau h|^{b_2-\frac 1p} \| f\|_{W^{b_2,p}} \les (|t|^{b_2-\frac 1p} + |\tau h|^{b_2-\frac 1p}) \| f\|_{W^{b_2,p}}.
\end{align*}
Then, for this portion of $\II(t)$, we have the bound
\begin{align*}
& \tau^{-b_1} \| f\|_{W^{b_2,p}}\Bigg\{ \| \varphi(\tfrac{t}{\tau}) |t|^{b_2-\frac{1}{p}} \|_{L^p}  \bigg( \int_{\R}\frac{1}{\jb{h}^{1+b_1}}dh\bigg) +  \tau^{b_2-\frac{1}{p}}\|\varphi(\tfrac{t}{\tau})\|_{L^{p}}  \bigg( \int_{\R} \frac{1}{ \jb{h}^{1+\frac 1p -(b_2-b_1)}  }dh\bigg)  \Bigg\} \\
 & \les \tau^{b_2-b_1} \|f\|_{W^{b_2,p}}.
\end{align*}
 Now combining this with  \eqref{timegain22}, \eqref{timegain20}, \eqref{timegain23}, \eqref{timegainIbd} and \eqref{dtnorm}, proves \eqref{timegain21}.
\end{proof}

%
%
%\begin{lemma}
%Let $s\in \R$, $1<p<\infty$ and $b<\frac 1p$. Then, we have for any interval $I\subseteq \R$ and any $u\in \Ld_{p}^{s,b}(I)$, 
%\begin{align*}
%\| u\|_{\Ld_{p}^{s,b}(I)} \sim \| \ind_{I}(t) u\|_{\Ld^{s,b}_{p}(\R)}
%\end{align*}
%and for any $v\in \Ld_{p}^{s,b}$, 
%\begin{align*}
%\| \ind_{I}(t) v \|_{\Ld^{s,b}_{p}} \les \|v\|_{\Ld^{s,b}_{p}}.
%\end{align*}
%%Finally, if $t_0 = \inf I$, then the map 
%%\begin{align*}
%%t \mapsto \| \ind_{[t_0 , t]} v\|_{\Ld^{s,b}_{p}}
%%\end{align*}
%%is continuous. 
%\end{lemma}
%See for instance \cite[Lemma 4.4]{

Next, we give some important embeddings regarding $L^{\infty}$-based spaces.

\begin{lemma}[Embeddings] \label{LEM:Sobolev}
Let $1\leq p \leq \infty$, $s>\frac2p$ and $b>\frac 1p$. Then, $\Ld^{s,b}_{p}(\R\times \T^2) \embeds C(\R\times \T^2)$, and we have
\begin{align}
\| u\|_{L^{\infty}(\R\times \T^2)} &\les \| u\|_{\Ld^{s,b}_{p}(\R\times \T^2)} \label{sobolev1} \\
\| u\|_{L^{\infty}_{t}(\R; W^{s,p}_x(\T^2))}& \les \| u\|_{\Ld^{s,b}_{p}(\R\times \T^2)} \label{sobolev12}
\end{align}
Furthermore, for any interval $I\subseteq \R$, $\Ld^{s,b}_{p}(I) \embeds C(I\times \T^2)$ and 
\begin{align}
\| u\|_{L^{\infty}(I\times \T^2)} &\les \| u\|_{\Ld^{s,b}_{p}(I)}, \label{sobolev2} \\
\| u\|_{L^{\infty}_{t}(I; W^{s,p}_x(\T^2))}& \les \| u\|_{\Ld^{s,b}_{p}(I)}  \label{sobolev21}
\end{align}
\end{lemma}
\begin{proof}
The first claim \eqref{sobolev1} is a simple consequence of Littlewood-Payley theory and Bernstein inequalities. Moreover, \eqref{sobolev12} follows from Minkowski's inequality and the Sobolev embedding in time (for each fixed $x\in \T^2$). The time localised estimates \eqref{sobolev2} and \eqref{sobolev21} then follow from \eqref{sobolev1} and \eqref{sobolev12}, respectively, and considering extensions, similar to the proof of Lemma~\ref{LEM:timegain}.
\end{proof}

In our proof of the kernel estimates in Section~\ref{subsec:ker}, it turns out to be more convenient to establish bounds using the isotropic fractional Laplacian $(-\Dl_{\R\times \T^2})^{\g}$ rather than the anisotropic operators $\jb{\nb_{x}}^{a}\jb{\dt}^{b}$. The next result plays an important role in connecting these two kinds of spaces.

\begin{lemma}
For any $1<p<\infty$ and $s,b\geq 0$ such that $0\leq s+b\leq 1$, it holds that
\begin{align}
\| u\|_{\Ld^{s,b}_{p} (\R\times \T^2)} \les  \| u\|_{L^{p}_{t,x}(\R\times \T^2)} + \| (-\Dl_{\R\times \T^2})^{\frac{s+b}{2}} u\|_{L^{p}_{t,x}(\R\times \T^2)}.
\label{iso}
\end{align}
\end{lemma}
\begin{proof}
 It suffices to prove the following on $\R\times \R^2$:
\begin{align}
\| \jb{\dx}^{b}\jb{\nb_x}^{s} u\|_{L^{p}_{t,x}(\R\times \R^2)} \les \|\jb{\nb_{t,x}}^{s+b}u\|_{L^{p}_{t,x}(\R\times\R^2)}. \label{iso2}
\end{align}
Indeed, assuming \eqref{iso2}, by the Mikhlin-H\"{o}rmander multiplier theorem  \cite[Theorem 6.2.7]{Grafakos1}, we obtain 
\begin{align}
\| \jb{\dt}^{b}\jb{\nb_x}^{s} u\|_{L^{p}_{t,x}(\R\times \R^2)} \les \|u\|_{L^p_{t,x}(\R\times \R^2)} + \| (-\Dl_{\R\times \R^2})^{\frac{s+b}{2}} u\|_{L^{p}_{t,x}(\R\times \R^2)}. 
\label{iso3}
\end{align}
Now, \eqref{iso} follows from \eqref{iso3} and transference.\footnote{Strictly speaking, we need a transference result from multipliers on $\R\times \R^2$ to multipliers on $\R\times \T^2$. We were not able to find a ready-made result in this mixed-type context in the literature. Nonetheless, one may easily adapt the transference proof in \cite[Theorem 3]{FS} to this setting. We leave the details to the reader.}
To prove \eqref{iso2}, it is enough to establish $L^p_{t,x}(\R\times \R^2)$ boundedness for the multiplier operators $\jb{\dt}^{a}\jb{\nb_{t,x}}^{-a}$ and $\jb{\nb_{x}}^{a}\jb{\nb_{t,x}}^{-a}$ for $0\leq a\leq 1$. It is straightforward to check that the corresponding multipliers satisfy the assumptions of the Marcinkiewicz multiplier theorem \cite[Corollary 6.2.5]{Grafakos1} and are thus $L^p_{t,x}$ bounded for $1<p<\infty$.
\end{proof}

In order to prove a stability result, which is key for the globalisation argument using Bourgain's invariant measure argument, we employ an idea from \cite{OOTol, OOTol2} which amounts to considering exponentially weighted-in-time norms. Due to the definition of the space $\Ld^{s,b}_{p}(I)$, we need to proceed carefully to obtain the required adaptions. 
Let $\ld\geq \frac 12$ and  define the time-weighted spaces
\begin{align}
 \| u\|_{\Ld^{s,b}_{p, \ld}} =  \| e^{-\ld |t|} u\|_{\Ld^{s,b}_{p}} \quad \text{and} \quad   \| u\|_{\Ld^{s,b}_{p, \ld}(I)} =  \| e^{-\ld |t|} u\|_{\Ld^{s,b}_{p}(I)}. 
\label{Ldweight}
\end{align}
Then, we have the following estimate:

\begin{lemma}\label{LEM:ldup}
Let $s\in \R$, $0<b<1$, and $1<p<\infty$. Then,
\begin{align*}
\| u\|_{\Ld^{s,b}_{p}([0,T])} \les  \ld e^{2\ld T} \|u\|_{\Ld^{s,b}_{p,\ld}([0,T])}
\end{align*}
for any $T>0$ and $\ld \geq \frac 12$.
\end{lemma}

\begin{proof}
It is clear that 
\begin{align*}
\| u\|_{\Ld^{s,0}_{p}([0,T])} \leq e^{\ld T}\| u\|_{\Ld^{s,0}_{p,\ld}([0,T])}.
\end{align*}
Thus by \eqref{dtnorm} it remains to show
\begin{align}
\| (-\dt^{2})^{\frac{b}{2}} u\|_{\Ld^{s,0}_p (I)} \les \ld e^{2 \ld T} \|  u\|_{\Ld^{s,b}_{p,\ld} (I)}. \label{weight00}
\end{align}
Let $v$ be any extension of $ \jb{\nb_{x}}^{s} u$ such that $v|_{[0,T]\times \T^2}= \jb{\nb_{x}}^{s} u$. Then, 
\begin{align*}
\| (-\dt^{2})^{\frac{b}{2}} u\|_{\Ld^{s,0}_p (I)} & \leq \| (-\dt^{2})^{\frac{b}{2}}  [ \varphi(t/T) v(t)]\|_{L^{p}(\R\times \T^2)}  \\
& = \big\| \| (-\dt^{2})^{\frac{b}{2}}  [  \varphi_{\ld,T}(t) ( e^{-\ld |t|}v(t))]\|_{L^{p}(\R)} \big\|_{L^{p}(\T^2)},
\end{align*}
where $ \varphi_{\ld,T}(t): = e^{\ld |t|}\varphi(t/T)$. Here, $\varphi$ is a smooth bump function defined in Subsection~\ref{sec:preliminaries} which equals $1$ on a neighborhood of the origin and is supported on a slightly larger set.
In the following, we fix $x\in \T^d$ and omit the dependence of $v$ on $x$ for clarity. 
Using \eqref{dtb}, we have 
\begin{align*}
(-\dt^{2})^{\frac{b}{2}}&  [  \varphi_{\ld,T}(t) ( e^{-\ld |t|}v(t))]  \\
& =  \varphi_{\ld,T}(t)(-\dt^{2})^{\frac{b}{2}}[ e^{-\ld |t|}v(t)] + c_{1,b}\int_{\R} \frac{\varphi_{\ld,T}(t+h)- \varphi_{\ld,T}(t) }{|h|^{1+b}}  (e^{-\ld |t+h|}v(t+h)) dh 
\end{align*}
With \eqref{dtnorm}, the first term yields
\begin{align}
\|\varphi_{\ld,T}(t)(-\dt^{2})^{\frac{b}{2}}[ e^{-\ld |t|}v(t)] \|_{L^{p}(\R\times \T^2)} & \les e^{2\ld T} \| e^{-\ld |t|}v\|_{\Ld^{0,b}_{p}}. \label{weight0}
\end{align}
For the second term, we split the integral into two parts. When $|h|>1$, Minkowski's inequality and a change of variables yields 
\begin{align}
\Big\|\int_{|h|>1} \frac{\varphi_{\ld,T}(t+h)- \varphi_{\ld,T}(t) }{|h|^{1+b}}  (e^{-\ld |t+h|}v(t+h)) dh  \Big\|_{L^p (\R)} \les e^{2\ld T} \| e^{-\ld|t|}v(t)\|_{L^{p}(\R)}. \label{weight1}
\end{align}
When $|h|\leq 1$, we first use the mean value theorem to get
\begin{align}
\Big\|\int_{|h|\leq 1} \frac{\varphi_{\ld,T}(t+h)- \varphi_{\ld,T}(t) }{|h|^{1+b}}  (e^{-\ld |t+h|}v(t+h)) dh  \Big\|_{L^p (\R)} \les \ld e^{2\ld T} \| e^{-\ld|t|}v(t)\|_{L^{p}(\R)}. \label{weight2}
\end{align} 
Combining \eqref{weight0}, \eqref{weight1} and \eqref{weight2} gives 
\begin{align*}
\| (-\dt^{2})^{\frac{b}{2}} u\|_{\Ld^{s,0}_p (I)}  \les (1+\ld)e^{2\ld T} \| e^{-\ld|t|}v\|_{L^p(\R\times \T^2)}.
\end{align*}
Then \eqref{weight00} follows by taking an infimum over the extensions $v$.
\end{proof}

\section{Analytic estimates}\label{sec:det}

\subsection{Basic product estimates}\label{subsec:prod}

We recall the following product estimate (fractional Leibniz rule). 
See \cite{CW, GO} on $\R^d$ and
\cite{BOZ} on $\T^d$.

\begin{lemma}\label{LEM:prod}
 Let $0\le s \le 1$. Let $\mathcal{M}=\R^d$ or $\T^d$.
Suppose that 
 $1<p_j,q_j,r < \infty$, $\frac1{p_j} + \frac1{q_j}= \frac1r$, $j = 1, 2$. 
 Then, we have  
\begin{equation*}  
\| \jb{\nb}^s (fg) \|_{L^r(\mathcal{M})} 
\les \| f \|_{L^{p_1}(\mathcal{M})} 
\| \jb{\nb}^s g \|_{L^{q_1}(\mathcal{M})} + \| \jb{\nb}^s f \|_{L^{p_2}(\mathcal{M})} 
\|  g \|_{L^{q_2}(\mathcal{M})}.
\end{equation*}

\smallskip

\noi
\textup{(ii)} 
Suppose that  
 $1<p,q,r < \infty$ satisfy the scaling condition:
$\frac1p+\frac1q\leq \frac1r + \frac{s}d $.
Then, we have
\begin{align*}
\big\| \jb{\nb}^{-s} (fg) \big\|_{L^r(\mathcal{M})} \les \big\| \jb{\nb}^{-s} f \big\|_{L^p(\mathcal{M}) } 
\big\| \jb{\nb}^s g \big\|_{L^q(\mathcal{M})}.  
\end{align*}
Moreover, it holds that 
\begin{align*}
\big\| \jb{\nb}^{-s} (fg) \big\|_{L^{2}(\mathcal{M})} \les \big\| \jb{\nb}^{-s} f \big\|_{L^2(\mathcal{M}) } 
\big\| \jb{\nb}^s g \big\|_{L^{\infty}(\mathcal{M})}. 
%  \label{productneg}
\end{align*}
\end{lemma} 

Next, we recall the fractional chain rule from \cite{Gatto}. The fractional chain rule on $\R^d$ was essentially proved in \cite{CW}.
%\footnote{As pointed out in \cite{Staffilani}, 
%the proof in \cite{CW} needs
%a small correction, which  yields the fractional chain  rule in a 
%less general context.
%See \cite{Kato, Staffilani, Taylor}.}
As for the estimate on $\T^d$, see~\cite{Gatto}.

\begin{lemma} \label{LEM:Lip}
Let $0<s<1$. 
%\textup{(i)} Suppose the $F$ is a Lipschitz function with Lipschitz constant $L>0$. Then, for any $1<p<\infty$, we have 
%\begin{align*}
%\| |\nb|^s F(u) \|_{L^{p}(\T^d)} \les L \||\nb|^{s} u\|_{L^{p}(\T^d)}.
%\end{align*}
Suppose that $F\in C^1 (\R)$ satisfies
\begin{align*}
|F'(\tau x+(1-\tau)y)|\leq c(\tau)( |F'(x)|+|F'(y)|),
\end{align*}
for every $\tau \in [0,1]$ and $x,y\in \R$, where $c\in L^1([0,1])$. Then, for every $1<p,q,r<\infty$ with $\frac{1}{p}+\frac{1}{q}=\frac{1}{r}$, we have
\begin{align*}
\| |\nb|^{s} F(u)\|_{L^{r}(\T^d)} \leq \|F'(u)\|_{L^{p}(\T^d)} \||\nb|^{s} u\|_{L^q (\T^d)}.
\end{align*}
\end{lemma}

We also recall the following product estimate regarding the product with a positive distribution from \cite[Lemma 2.14]{ORW}.

\begin{lemma} \label{LEM:pos}
Let $0\leq s\leq 1$ and $1<p<\infty$. Then, we have 
\begin{align*}
\| \jb{\nb}^{-s}(fg)\|_{L^p(\T^d)} \les \|f\|_{L^{\infty}(\T^d)} \|\jb{\nb}^{-s}g\|_{L^{p}(\T^d)}
\end{align*}
for any $f\in L^{\infty}(\T^d)$ and any positive distribution $g\in W^{-s,p}(\T^d)$, satisfying one of the following two conditions: \textup{(i)} $f\in C(\T^d)$ or \textup{(ii)} $f\in W^{s,q}(\T^d)$ for some $1<q<\infty$ such that $\frac{1}{p}+\frac 1q <1+\frac{s}{d}$.
\end{lemma}

\subsection{Kernel estimates}\label{subsec:ker}
In this section, we establish the key kernel estimates and corresponding inhomogenous estimate tailored to our $L^{\infty}$-framework.
Define the kernel $K^\star$ by
\begin{align}
K^{\star} (t,s, x) =  e^{-\frac{1}{2}(t-s)} \chi(t) \ind_{[s, \infty)}(t) \ind_{[0, \infty)}(s)  \mc W_{\T^2}(t-s,x), \quad (s,t,x) \in \R^2 \times \T^2.
\label{ker1b}
\end{align}
Here, $\mc W_{\T^2}$ is as in~\eqref{poisson3} and $\chi\in C^{\infty}_{0}(\R)$ with $\chi=1$ on $[-\frac 34, \frac 34]$ and $\supp \chi \subset [-1,1]$.  Note that $\mathcal{I}_{\text{wave}}^{\star}$ is a properly time-localised (with respect to $t$) version of $\mathcal{I}_{\text{wave}}$ from \eqref{duhamels}.
Consequently, since $\chi$ is supported on $[-1,1]$ and $0 \le s \le t$ on the support of $K^\star$, we may, as discussed in Section~\ref{sec:preliminaries}, replace $\mc W_{\T^2}$ by $\mc W_{\R^2}$ in~\eqref{ker1b}. This convention will be adopted throughout this section without further comment.

\begin{comment}
\younes{This is the kernel we want to work with as it is 
\begin{itemize}

\item compactly supported in a small ball in the time variables $(t,s)$ s.t. $\mc W$ on the torus coincides with $\mc W$ on $\R^2$

\item there is a time singularity at $t=s$, but not at $t=0$, which would be an issue — it's easy to introduce such a singularity when dealing with an integral from $0$ to $t$.

\end{itemize}}

\justin{Do we need to replace $\chi(t)$ in \eqref{ker1b} by $\chi(t/T)$, or is there a simpler way to handle time localisation?
I'm thinking this because the first thing for estimates seems to be:
\begin{align*}
\| \I_{\textup{wave},\ld}^\star(F) (t)\|_{\Ld^{s,b}_{p}([0,T])} \les \| \chi(t/T)\I_{\textup{wave},\ld}^\star( \chi(\cdot/T)  F) (t)\|_{\Ld^{s,b}_{p}(\R)}
\end{align*}
}
\younes{Yes, but there's no way to keep an 'exact'—i.e. up to $T$—time localisation as you apply fractional derivatives since the corresponding singular integrals are not time localised. Since the bound you've stated is basically the best we can do, we'll have to integrate things over the real line and hence have some decay in time for our estimates.
Hence, I think we should keep $K^\star_\ld$ as is and replace the RHS of~\eqref{kerygoal} with an expression of the form 
\[\Big\| \int_{\R} \mc R^\star_\s (t,s,\cdot) *_x |F(s)| ds \Big\|_{L^p(\R \times \T^2)},\]
where $\mc R_\s^\star(t,s,x)$ looks like the current $\mc R_\s^\star$, but decays in $\jb{t}$.}

\justin{I see, thanks! I think I'm okay with this as we can still have $\ind_{[0,T]}$ on the function $F$, $0<T\leq 1$.}
\end{comment}

Now, define the following Duhamel operator:
\begin{align}
\I_{\textup{wave}}^\star(F) (t) = \int_{\R} K^{\star}(t,s,\cdot) \ast_x F(s) ds,
\label{duha3}
\end{align}
and for $0\le \s<\frac12$, the sublinear operator $\mc A_\s$ given by
\begin{align}
\mc A_\s(F) (t) = \int_{\R} \mc R_\s(t,s,\cdot) \ast_x |F(s)| ds,
\label{subop}
\end{align}
where the kernel $\mc R_\s$ is defined as
\begin{align*}
\mc R_{\s}(t,s,x) & = \jb{t}^{-3-\s} \big(1 +   |t-s|^{-\frac12-\s} |x|^{-\frac12} \\
& \quad + (|t-s| + |x|)^{-\frac12} ||t-s| - |x||^{-\frac12-\s}\big)  \jbb{\log(||t-s|- |x||)}^2 \, \jb{\log(|x|)}^3.
\end{align*}

The main goal of this section is to prove the next estimate.

\begin{proposition}\label{prop:kery}
Let $1 \le p < \infty$, $0\le \s<\frac12$, and $\I^{\star}_{\textup{wave}}$ and $\mc A_\s$ be as in~\eqref{duha3} and~\eqref{subop}, respectively. Then, the following bound holds:
\begin{align}
\|(-\Dl_{\R \times \T^2})^{\frac\s2} \I^{\star}_{\textup{wave}}(F) \|_{L^p(\R \times \T^2)} \les \|A_\s(F)\|_{L^p(\R \times \T^2)}.
\label{kerygoal}
\end{align}
\end{proposition}

Let $\nu^1 \in C^{1}_c(\R; [0,1])$ and $\nu^2 \in C^1_c(\R^2;\R)$ be functions supported in $[-1,1]$ and $B(0,1)$, respectively, which integrate to $1$. Then, for $N \ge 1$, set $\nu^1_N = N \nu(N \cdot)$ and $\nu^2_N = N^2 \nu^2(N\cdot)$. Define the spatially periodic function $K^\star_{N}$ via 
\begin{align}
K^\star_{N}(t,s,x) = \int_{\R \times \T^2} \nu^1_N(t-t') \nu^2_N(x-y) K^\star(t',s,y) dt' dy
\label{kery1}
\end{align}
for all $(t,s,x) \in \R^2 \times \T^2$.

\begin{lemma}\label{lem:ker01}
Fix $N \ge 1$, $a \in \R\setminus\{0\}$ and $0 <\g \le \frac12$. Let $\nu \in L^{\infty}(\R; \R)$ be a compactly supported function and set $\nu_N = N \nu(N \cdot)$. Let $S_{\g,a}$ be the function the given by

\noi
\begin{align*}
S_{\g,a}(t) & = \big|t^2 - a^2 \big|^{-\g}
\end{align*}
and set $S_{N,\g,a} := S_{\g, a} *_t \nu_N$. Then, the following bound holds:

\noi
\begin{align}
\begin{split}
 |S_{N,\g,a} (t)|& \les \big| t^2 - a^2 \big|^{-\g} \jb{\log |a|}
\end{split}
\label{kery2}
\end{align}
for any $t \in \R$, where the implicit constant is uniform in the parameter $N \ge 1$.
\end{lemma}

\begin{proof}
Fix $t \in \R$ and consider
\[ S_{N,\g,a} (t) = \int_{\R} \nu_N(t-s)  \big|s^2 - a^2 \big|^{-\g} ds. \]
Let $s$ be in the support of $\nu_N(t-\cdot)$.

\medskip

\noi
{\bf $\bul$ Case 1: $|t| \les |s|$.}  If $\big|s^2 - a^2 \big| \ges \big| t^2 - a^2 \big|$ then~\eqref{kery2} holds. Hence, we assume that $\big|s^2 - a^2 \big| \ll \big| t^2 - a^2 \big|$. Under this assumption, we have
\begin{align*}
  \big| t^2 - a^2 \big| \sim \big|t^2 - a^2 - (s^2 - a^2) \big| \sim ||t|-|s|| \cdot ||t|+|s|| \les N^{-1} ||t| + |s||.
\end{align*}
Therefore, we have
\begin{align*}
\big|t^2 - a^2 \big|^{\g} \int_{\R} |\nu_N(t-s)|  \big|s^2 - a^2 \big|^{-\g} ds & \les N^{1-\g} \int_\R \ind_{|t-s| \les  N^{-1}} \frac{|t| + |s|}{|s| + |a|} ||s|-|a||^{-\g} \\
& \les N^{1-\g} \int_\R \ind_{|t-s| \les  N^{-1}}  ||s|-|a||^{-\g} ds \\
& \les 1,
\end{align*}
as desired in~\eqref{kery2}. 

\medskip

\noi
{\bf $\bul$ Case 2: $|s| \ll |t|$ and $|t| \ll |a|$.} In this case, we simply bound
\begin{align*}
 \int_{\R} |\nu_N(t-s)| \big|s^2 - a^2 \big|^{-\g} \ind_{|s| \ll |t| \ll |a|} ds &  \les |a|^{-2\g}  \int_\R |\nu_N(t-s)| ds \les |a|^{-2\g}  \les  \big|t^2 - a^2\big|^{-\g},
\end{align*}
since $|t| \ll |a|$. This is acceptable in view of~\eqref{kery2}.

\medskip

\noi
{\bf $\bul$ Case 3: $|s| \ll |t|$ and $|t| \ges |a|$.} We consider two scenarios.

\medskip

\noi
{\bf $\bul$ Subcase 3.1: $|a| \ll |s| \ll |t|$.} In this situation, we have
\begin{align*}
 \int_{\R} |\nu_N(t-s)| \big|s^2 - a^2 \big|^{-\g} \ind_{|a| \ll |s| \ll |t|} ds 
&  \les N  \int_\R |s|^{-2\g} \ind_{|a|\ll |s| \les N^{-1}} ds \\
&  \les N \big( |\log |a|| \cdot \ind_{\g = \frac12} + N^{2\g-1} \cdot \ind_{0 < \g < \frac12}\big) \\
&  \les N^{2\g} |\log |a||\\
& \les |t|^{-2\g} |\log |a|| \\
&  \les \big|t^2 - a^2\big|\jb{\log |a|}.
\end{align*}
since $N \les |t|^{-1}$ and $|t| \gg |a|$. This shows~\eqref{kery2}. 

\medskip

\noi
{\bf $\bul$ Subcase 3.2: $|s| \ll |t|$, $|s| \les |a|$ and $|t| \ges |a|$.} In this situation, we have
\begin{align*}
 \int_{\R} |\nu_N(t-s)| \big|s^2 - a^2 \big|^{-\g} \ind_{|s| \les |a|} ds & \les |a|^{-\g} N \int_\R \ind_{|s| \les  |a|}  ||s|-|a||^{-\g} ds \\
&  \les |a|^{1-2\g} N \\
&  \les |t|^{-2\g} \\
&  \les \big|t^2 - a^2\big|,
\end{align*}
since $N \les |t|^{-1}$ and $|a| \les |t|$. Here, we used that $\g \le \frac12$. This agrees with~\eqref{kery2}.
\end{proof}

Next, we prove the following variant of \cite[Lemma 3.4]{Zine3}.

\begin{lemma}\label{lem:ker02}
Fix $N \ge 1$, $0 <\g \le \frac12$ and $k \in \{1,2\}$. Let $\nu \in L^{\infty}(\T^2; \R)$ be a compactly supported function and set $\nu_N = N^2 \nu(N \cdot)$. Let $t \in \R\setminus\{0\}$, and define the function $H_{t,\g}$ by

\noi
\begin{align*}
H_{t,\g}(x) & = \big|t^2 - |x|^2 \big|^{-\g} \jb{\log |x|}^k
\end{align*}
Set $ H_{N,t, \g} := H_{t,\g} *_x \nu_N$. Then, the following bound holds:

\noi
\begin{align}
\begin{split}
 |H_{N,t,\g} (x)|& \les \big| t^2 - |x|^2 \big|^{-\g} \jbb{\log(||t|- |x||)}  \jb{\log|x|}^{k}
\end{split}
\label{CC1}
\end{align}
for any $x \in \T^2\setminus\{0\}$, where the implicit constant is uniform in the parameter $N \ge 1$.
\end{lemma}
\begin{proof}
Let $x \in \T^2 \setminus\{0\}$. We have
\begin{align}
H_{N,t,\g}(x) = \int_{\T^2} \nu_N(x-y) H_{t,\g}(y) dy.
\label{kery20}
\end{align}
Let $y$ be in the support of the integrand of the integral in~\eqref{kery20}.

\medskip

\noi
{\bf $\bul$ Case 1: $|y| \ges |x|$.} Then, from~\cite[Lemma 3.4]{Zine3}, we have
\begin{align*}
|H_{N,t,\g}(x)| & \les \jb{\log |x|}^{k} \, \int_{\T^2} |\nu_N(x-y)| \big|t^2 - |y|^2\big|^{-\g} dy \\
& \les \jb{\log |x|}^{k} \, \big| t^2 - |x|^2 \big|^{-\g} \, \jbb{\log(||t|- |x||)},
\end{align*}
as desired in~\eqref{CC1}. 

\medskip

\noi
{\bf $\bul$ Case 2: $|y| \ll |x|$.} Note that in this situation, we necessarily have $|x| \les N^{-1}$. We consider three different scenarios.

\medskip

\noi
{\bf $\bul$ Subcase 2.1: $|y| \gg |t|$.} In this case, we simply have
\begin{align*}
H_{t,\g}(y) \les |y|^{-2\g} \, \jb{\log |y|}^k \, \ind_{|y|\les N^{-1}}.
\end{align*}
Hence, by a polar change of coordinates and integration by parts, we have
\begin{align*}
H_{N,t,\g}(x)  \les N^2 \int_{\T^2} |y|^{-2\g} \, \jb{\log |y|}^k dy & \les N^2 \int_{0}^{10 N^{-1}} r^{1-2\g} |\log(r)|^k dr \\
& \les N^{2\g} \log(N)^k\\
& \les |x|^{-2\g} \, \jb{\log |x|}^k \\
& \les \big|t^2 - |x|^2\big|^{-\g} \, \jb{\log |x|}^k,
\end{align*}
since $|x| \les N^{-1}$ and $|t| \ll |x|$. This shows~\eqref{CC1}.

\medskip

\noi
{\bf $\bul$ Subcase 2.2: $|y| \sim |t|$.} In this case, we have 
\begin{align*}
H_{t,\g}(y) \les |t|^{-\g} ||t| - |y||^{-\g} \, \jb{\log |y|}^k \, \ind_{|y|\sim |t|}.
\end{align*}
Hence, by a polar change of coordinates and integration by parts, we have
\begin{align*}
H_{N,t,\g}(x) & \les N^2 |t|^{-\g} \int_{\T^2} ||t| - |y||^{-\g} \, \jb{\log |y|}^k \, \ind_{|y|\sim |t|} dy \\
& \les N^2 |t|^{1-\g} \int_{10^{-1} |t|}^{10 |t|} ||t|-r|^{-\g} |\log(r)|^k dr \\
& \les N^{2} \cdot |t|^{2-2\g} |\log |t||^k \\
& \les N^{2\g} \log(N)^k \\
& \les \big|t^2 - |x|^2\big|^{-\g} \, \jb{\log |x|}^k.
\end{align*}
Note that in the above, we used that for $\al >0$ the map $r \mapsto r^{\al} (-\log(r))^k$ is non-decreasing in a neighborhood of the origin, and that $|t| \les N^{-1}$ and $|x| \les N^{-1}$. This shows~\eqref{CC1}.

\medskip

\noi
{\bf $\bul$ Subcase 2.3: $|y| \ll |t|$.} In this case, we have
\begin{align*}
H_{t,\g}(y) \les |t|^{-2\g} \jb{\log|y|}^k \ind_{|y| \les \min(N^{-1},|t|)}.
\end{align*}
Therefore, by a polar change of coordinates and integration by parts, we have
\begin{align*}
|H_{N,t,\g}(x)| & \les |t|^{-2\g} N^2 \cdot \min \! \big(N^{-2} \log(N)^k, |t|^2 \, \jb{\log|t|}^k\big) \ \\
& \les \min \! \big( |t|^{-2\g} \log(N)^k, |t|^{2-2\g} N^2 \, \jb{\log|t|}^k\big) \\
& \les \min \! \big( |t|^{-2\g} \, \jb{\log |x|}^k, |t|^{2-2\g} |x|^{-2} \, \jb{\log|t|}^k\big) \\
& \les \big| t^2 - |x|^2 \big|^{-\g} \, \jb{\log |x|}^k
\end{align*}
In the above, we again used that $N \les |x|^{-1}$ and the fact that the map $r \mapsto r^{2-2\g} (-\log(r))^k$ is non-decreasing around the origin. This completes the proof of~\eqref{CC1}. 
\end{proof}

Before stating our next result, we introduce some notation. Given $s_1, s_2,s_3 \geq 0$, set
\[ \mc R_{s_1, s_2, s_3}(t,x) := |t|^{-s_1} (|t| + |x|)^{-s_2} | |t| - |x||^{-s_3} \jbb{\log(||t|- |x||)} \, \jb{\log(|x|)} \]
for all $(t,x) \in \R \times \T^2$.
We note that by first applying Lemma~\ref{lem:ker02}, and then Lemma~\ref{lem:ker03}, we have the following useful estimate:
\begin{align}
 |K^\star_{N}(t,s,x)| \les \mc R_{0,\frac12,\frac12}(t-s,x). \label{keryK}
\end{align}

\begin{comment}
\younes{In what follows, we essentially use Lemmas~\ref{lem:ker01} and~\ref{lem:ker02} to show that 
\[ |K^\star_{\ld, N}(t,s,x)| \les \mc R_{0,\frac12,\frac12}(t,s,x)\]
We could write a lemma/corollary to state this once and for all or use both lemmas everytime we need the bound in the above—that's what I'm doing down below. What do you think?
}
\justin{I vote for stating it here and then citing the mathdisplay later}

\younes{Actually, on a second thought we don't exactly use the bound \[ |K^\star_{\ld, N}(t,s,x)| \les \mc R_{0,\frac12,\frac12}(t,s,x),\]
but a similar bound with absolute values inside on the bump functions or their derivatives. It's hard to capture all of this in a single statement so I suggest we keep things this way.}
\end{comment}

\begin{lemma}\label{lem:ker03}
Fix $N\ge 1$ and let $K^\star_{N}$ be as in~\eqref{kery1}. Then, the following estimates hold:
\begin{align}
| \partial_t K^\star_{N}(t,s,x) | & \les \mc R_{1, \frac12, \frac12}(t-s, x) + \mc R_{0, \frac12, \frac32}(t-s, x),\label{kery5} \\
| \nb_x K^\star_{ N}(t,s,x) | & \les \mc R_{0, \frac12, \frac32}(t-s, x), \label{kery6}
\end{align}
for all $(t,s,x) \in \R^2 \times \T^2$. Here, the implicit constants are uniform in $N \ge 1$.
\end{lemma}

\begin{proof}
We focus on~\eqref{kery5} as~\eqref{kery6} follows from slightly simpler arguments. We also note that the bound~\eqref{kery6} follows from a (variant of)~\cite[Lemma 3.11]{Zine3}. 
Fix $(t,s,x) \in \R^2 \times \T^2$. We divide our analysis into different cases depending on the value of $(t,s,x)$.

\medskip

\noi
{\bf $\bul$ Case 1: $||t-s|-|x|| \les N^{-1}$.} In this case, we move time derivative on $K^\star_{N}$ to the corresponding bump function $\nu^1_N$ so that by~\eqref{kery1}, we have
\begin{align*}
\partial_t K^\star_{N}(t,s,x) = \int_{\R \times \T^2} \dt \nu^1_N(t-t') \nu^2_N(x-y) K^\star(t',s,y) dt' dy.
\end{align*}
Hence, by \eqref{keryK}, we deduce that
\begin{align*}
|\partial_t K^\star_{N}(t,s,x)| & \les N \cdot \mc R_{0, \frac12, \frac12}(t-s,x) \les \mc R_{0, \frac12, \frac32}(t-s,x),
\end{align*}
as desired in~\eqref{kery6}.

\medskip

\noi
{\bf $\bul$ Case 2: $||t-s|-|x|| \gg N^{-1}$.} In this case, we consider again two scenarios.

\medskip

\noi
{\bf $\bul$ Subcase 2.1: $||t-s|-|x|| \gg N^{-1}$ and $|t-s| \les N^{-1}$.}  We again move the time derivative on $K^\star_N$ to its corresponding bump function $\nu^1_N$. As in Case 1, this gives
\begin{align*}
|\partial_t K^\star_{N}(t,s,x)| & \les N \cdot \mc R_{0, \frac12, \frac12}(t-s,x) \les \mc R_{1, \frac12, \frac12}(t-s,x)
\end{align*}
for any $\eps >0$, as desired in~\eqref{kery5}.

\medskip

\noi
{\bf $\bul$ Subcase 2.2: $|t-s| \gg N^{-1}$ and $||t-s|-|x|| \gg N^{-1}$.} Let $(t',y)$ be in the support of the function $(t',y) \mapsto \nu^1_N(t-t') \nu^2_N(x-y) K^\star(t',s,y)$. Then, the following bound holds:
\begin{align}
\big| ||t'-s| + \s |y|| - ||t-s| + \s|x|| \big| \le |t-t'| + |x-y| \les N^{-1}
\label{kery8}
\end{align}
for $\s \in \{-1,0,+1\}$. Since $||t-s|-|x|| \gg N^{-1}$,~the last estimate implies that 
\begin{align}
| |t'-s| + \s |y| | \sim | |t-s| + \s |x| |
\label{kery9}
\end{align}
for $\s \in \{-1,0,+1\}$. Moreover, for any fixed $y \in \T^2$, the function $t' \mapsto K^\star(t',s,y)$ is smooth on the support of $(t',y) \mapsto \nu^1_N(t-t') \nu^2_N(x-y) K^\star(t',s,y)$,\footnote{Note that this support is a closed interval.} with derivative bounded by 
$$\mc R_{1,\frac12,\frac12}(t'-s,x)+\mc R_{0,\frac12,\frac32}(t'-s,x)$$
Hence, by moving the time derivative on $K^\star_{N}$ to the kernel $K^\star$ and from~\eqref{kery9}, this leads to the bound
\begin{align*}
|\partial_t K^\star_{N}(t,s,x)| & = \Big| \int_{\R \times \T^2} \nu^1_N(t-t')  \nu^2_N(x-y) \partial_{t'}  K^\star(t',s,y) dt' dy \Big| \\
& \les \int_{\R \times \T^2} | \nu^1_N(t-t')|  |\nu^2_N(x-y)| \\
& \qquad \qquad \quad \times  \big(\mc R_{1,\frac12,\frac12}(t-s,x) + \mc R_{0,\frac12,\frac32}(t-s,x)\big) dt' dy \\
& \les \mc R_{1,\frac12,\frac12}(t-s,x) + \mc R_{0,\frac12,\frac32}(t-s,x),
\end{align*}
as desired in~\eqref{kery5}.
\end{proof}

We now put these previous results together to conclude an estimate on the fractional derivatives of $K^\star_{N}$.

\begin{lemma}\label{lem:ker1}
Let $N \ge 1$, $K^\star_{N}$ be as in~\eqref{kery1}, and $0 < \s < \frac12$. Then, the following estimate holds:
\begin{align}
\begin{split}
 |(-\Dl_{\R \times \T^2})^{\frac \s2} K^\star_{N} (t,s,x) \big| & \les \jb t^{-3 -\s} \big( 1+ \mc R_{\s, \frac12, \frac12}(t-s,x) + \mc R_{0, \frac12, \frac12 +\s}(t-s,x)\big)\\
 & \qquad \quad \times \jbb{\log(||t-s| - |x||)} \, \jb{\log|x|}
 \end{split}
 \label{kery12}
\end{align}
for all $(t,s,x) \in \R^2 \times \T^2$. Here, $(-\Dl_{\R \times \T^2})^{\frac \s2}$ denotes the fractional Laplacian on $\R \times \T^2$~\eqref{dtb2}, with respect to the first and third variables of $K^\star_{N}$, and the implicit constant is uniform in $N\ge 1$.
\end{lemma}

\noi
Note that while $K_{N}^{\star}(t,s,x)$ is compactly supported in the variables $(t,s)$, this is no longer the case for the fractional derivatives $(-\Dl_{\R \times \T^2})^{\frac \s2} K^\star_{N}$.

\begin{proof}
Let us fix $(t,s,x) \in \R^2 \times \T^2$. From~\eqref{dtb}, we get
\begin{align}
(-\Dl_{\R \times \T^2})^{\frac \s2} K^\star_{N} (t,s,x) = c_{3,\s} \int_{\R\times\T^2} K^\star_{N}(t,s,x;h_1,h_2) \mc K_{3,\s}(h_1, h_2) dh_1 dh_2,
\label{kery13}
\end{align}
where
\begin{align*}
 K^\star_{ N}(t,s,x;h_1,h_2)  :=  K^\star_{N}(t+h_1,s,x+h_2) - K^\star_{N}(t,s,x).
\end{align*}
 and $\mc K_{3,\s}$ is as in \eqref{Laplaceker}. Note that in this case, by construction of $K^\star$ and $K^\star_{N}$, the contribution of $K^\star_{N}(t+h_1, s, x+h_2)$ to~\eqref{kery13} is non-zero if  we have $|t+h_1|, |s| \les 1$.

We divide our analysis into several cases.

\medskip
\noi
{\bf $\bul$ Case 1: $|h_1| \gg 1$.} In this case, it follows from \eqref{Laplaceker} that
\begin{align*}
|\mc K_{3,\s}(h_1, h_2)| \les |h_1|^{-3-\s},
\end{align*}
for any $ h_2\in \T^2 \cong [-\pi, \pi)^2$.
Hence, in view of the support properties of $K^\star_{N}(t+h_1, s, x+h_2)$ and ~\eqref{keryK}, we have that
\begin{align}
\begin{split}
&  \int_{\R\times\T^2} |K^\star_{N}(t,s,x;h_1,h_2)| |\mc K_{3,\s}(h_1, h_2)| \ind_{|h_1| \gg 1} dh_1 dh_2 \\
&\quad  \les \int_{\R \times \T^2} |h_1|^{-3-\s} \mc R_{0, \frac12,\frac12}(t+h_1-s,x+h_2) \ind_{|h_1| \gg 1} \ind_{|t+h_1| \les 1} dh_1 dh_2 \\
& \quad \qquad + \int_{\R \times \T^2} |h_1|^{-3-\s} |K^\star_{ N}(t,s,x)| \ind_{|h_1| \gg 1} dh_1 dh_2 \\
& \quad =: \1 + \II.
\end{split}
\label{kery105}
\end{align}
On the one hand, by a translation in $h_2$ and since $|s| \les 1$ and $|h_1| \sim |t| \gg 1$, the first term in~\eqref{kery105} is estimated by
\begin{align}
\begin{split}
\1 & \les \jb{t}^{-3-\s} \int_{|h_1+t|\les 1} dh_1 \frac{\jb{\log(|t-s+h_1|)} }{  |t-s+h_1|^{\frac12} }    \int_{\T^2} dh_2 \frac{\jb{\log(|h_2|)}}{  ||t-s+h_1| - |h_2| |^{\frac12} } \\
& \les \jb{t}^{-3-\s} \int_{|h_1+t-s|\les 1} \frac{1 }{  |t-s+h_1|^{\frac12+\eps} }dh_1 \\
&\les \jb{t}^{-3-\s}.
\end{split}
\label{kery106}
\end{align}
On the other hand, from~\eqref{keryK} and recalling that $|t| \les 1$ on the support of $K^\star_{ N}(t,s,x)$, we have
\begin{align}
\II \les \jb t ^{-10} \, \mc R_{0,\frac12,\frac12}(t-s,x).
\label{kery107}
\end{align}
Therefore, by~\eqref{kery105},~\eqref{kery106} and~\eqref{kery107}, we deduce that
\begin{align*}
  \int_{\R\times\T^2} |K^\star_{N}(t,s,x;h_1,h_2)| |\mc K_{3,\s}(h_1, h_2)| \ind_{|h_1| \gg 1} dh_1 dh_2   \les \jb t^{-3-\s} \big( 1+ \mc R_{0,\frac12, \frac12}(t-s,x)\big),
\end{align*}
as desired in~\eqref{kery12}.

\medskip

\noi
{\bf $\bul$ Case 2: $|h_1| \les 1$.} Note that in this case, by construction of $K^\star$ and $K^\star_{N}$, we have $||t-s| - |x||, |t|, |s|, |h_1| \les 1$ on the support of the integrand in~\eqref{kery13}. It hence suffices to prove~\eqref{kery12} in the case $|t| \les 1$ since the bound is otherwise trivial. 
Note by \eqref{Laplaceker}, we have
\begin{align}
|\mc K_{3,\s}(h_1, h_2)| \les  \frac{1}{|(h_1, h_2)|^{3+\s}}
\label{kery13b}
\end{align}
for all $|(h_1,h_2)|\les 1$.

\medskip

\noi
{\bf $\bul$ Subcase 2.1: $\min(||t-s| - |x||, |t-s|) \les |(h_1,h_2)|$.}  From~\eqref{kery13b} and a dyadic decomposition, we bound
\begin{align}
\begin{split}
& \int_{\R\times\T^2} |K^\star_{N}(t,s,x;h_1,h_2)||\mc K_{3,\s}(h_1, h_2)| \ind_{\min(||t-s| - |x||, |t-s|) \les |(h_1,h_2)| \les 1} \ind_{|h_1|\les 1} dh_1 dh_2 \\
& \quad  \les \sum_{\l=0}^{\infty} \ind_{2^l \les \max(||t-s| - |x||^{-1}, |t-s|^{-1})} \int_{\R\times\T^2} \frac{|K^\star_{ N}(t,s,x;h_1,h_2)|}{|(h_1,h_2)|^{3+\s}} \ind_{|(h_1,h_2)| \sim 2^{-\l}} dh_1 dh_2 \\
& \quad  \les \sum_{\l=0}^{\infty} \ind_{2^l \les \max(||t-s| - |x||^{-1}, |t-s|^{-1})} 2^{\l \s} \\
& \quad \qquad \qquad \qquad \times \int_{\R\times\T^2} | K^\star_{N}(t,s,x;h_1,h_2)| 2^{3\l} \ind_{|h_1| \les 2^{-\l}} \ind_{|h_2| \les 2^{-\l}} dh_1 dh_2.
\end{split}
\label{kery13c}
\end{align}
From~\eqref{keryK}, and recalling that $|t+h_1| \les 1$ on the support of $K_{N}^{\star}(t+h_1,s,x+h_2)$, the contribution from $K_{N}^{\star}(t+h_1,s,x+h_2)$ to the integral
\begin{align} \int_{\R\times\T^2} | K^\star_{N}(t,s,x;h_1,h_2)| 2^{3\l} \ind_{|h_1| \les 2^{-\l}} \ind_{|h_2| \les 2^{-\l}} dh_1 dh_2
\label{kery13cc}
\end{align}
is bounded by
\begin{align}
\begin{split}
& \int_{\R\times\T^2} | K^\star_{N}(t+h_1,s,x+h_2)| 2^{3\l} \ind_{|h_1| \les 2^{-\l}} \ind_{|h_2| \les 2^{-\l}} dh_1 dh_2 \\
& \quad \les \int_{\R \times \T^2} \frac{\jbb{\log(||t-s+h_1| - |x+h_2||)}}{\big|(t-s+h_1)^2 - |x+h_2|^2\big|^{\frac12}} \ind_{|t-s+h_1| \les 1} \\
& \qquad \qquad \quad \times \jb{\log|x+h_2|} \, 2^{3\l} \ind_{|h_1| \les 2^{-\l}} \ind_{|h_2| \les 2^{-\l}} dh_1 dh_2.
\end{split}
\label{kery13d}
\end{align}
Let us denote by $\1_\l$ and $\II_\l$ the respective contributions of $||t-s+h_1| - |x+h_2|| < 2^{-100 \l}$ and $||t-s+h_1| - |x+h_2|| \ge 2^{-100 \l}$ to the right-hand side of~\eqref{kery13d}. On the one hand, by using the bound $|\log (x)| \les |x|^{-\frac14}$ for any $0<|x| \les 1$ and a polar change of coordinates, we first estimate $\1_\l$ by
\begin{align}
\begin{split}
\1_\l & \les 2^{3\l} \int_{\R} dh_1 \frac{\ind_{|t-s+h_1| \les 1}}{|t-s+h_1|^{\frac12}} \int_{\T^2} dh_2 \frac{\jb{\log|x+h_2|} }{||t-s+h_1| - |x+h_2||^{\frac34}}  \ind_{||t-s+h_1| - |x+h_2|| < 2^{-100 \l}} \\
& \les 2^{3\l} \int_{\R} dh_1 \ind_{|h_1| \les 1} \int_{0}^{10} dr \frac{\jb{\log(r)} }{||h_1| - r|^{\frac34}}  \ind_{||h_1| - r| < 2^{-100 \l}} \\
& \les 2^{-10\l}.
\end{split}
\label{kery100}
\end{align}
On the other hand, the term $\jbb{\log(||t-s+h_1| - |x+h_2||)} \ind_{||t-s+h_1| - |x+h_2|| \ge 2^{-100 \l}}$ is bounded up to a constant by $ \l$. Hence, applying Lemma~\ref{lem:ker01} and Lemma~\ref{lem:ker02} with $k=2$, yields
\begin{align}
\begin{split}
\II_\l & \les \l  \int_{\T^2} dh_2 2^{2\l} \ind_{|h_2| \les 2^{-\l}} \int_{\R} dh_1 \frac{\jb{\log|x+h_2|}}{\big|(t-s+h_1)^2 - |x+h_2|^2\big|^{\frac12}}  2^{\l} \ind_{|h_1| \les 2^{-\l}} \\
& \les  \l \int_{\T^2}  2^{2\l} \ind_{|h_2| \les 2^{-\l}}  \frac{\jb{\log|x+h_2|}^2}{\big|(t-s)^2 - |x+h_2|^2\big|^{\frac12}}dh_2 \\
& \les \l \cdot \mc R_{0,\frac12,\frac12}(t-s,x) \, \jb{\log |x|}
\end{split}
\label{kery101}
\end{align}
Lastly, the contribution from $K_{\ld,N}^{\star}(t,s,x)$ to~\eqref{kery13cc} is simply bounded by $\mc R_{0,\frac12,\frac12}(t-s,x)$. Therefore, from the above,~\eqref{kery13c},~\eqref{kery13d}, \eqref{kery100} and~\eqref{kery101}, we deduce that
\begin{align*}
& \int_{\R\times\T^2} |K^\star_{ N}(t,s,x;h_1,h_2)||\mc K_{3,\s}(h_1, h_2)| \ind_{\min(||t-s| - |x||, |t-s|) \les |(h_1,h_2)| \les 1} dh_1 dh_2 \\
& \quad  \les \sum_{\l=0}^{\infty} \ind_{2^l \les \max(||t-s| - |x||^{-1}, |t-s|^{-1})} 2^{\l \s} \big( \l \cdot \mc R_{0,\frac12, \frac12}(t-s,x) \, \jb{\log|x|} + 2^{-10\l} \big) \\
& \quad \les 1+ \big(\mc R_{\s,\frac12, \frac12}(t-s,x) + \mc R_{0, \frac12, \frac12+\s}(t-s,x)\big)\, \jbb{\log(||t-s| - |x||)} \, \jb{\log|x|}.
\end{align*}
This shows~\eqref{kery12} since $|t| \les 1$.
%
%
%\justin{NB: By the support properties of $K_{\ld, N}^{\star}$, in this case we must have $|t|\les 1$ since otherwise if $|t|\gg 1$, as $0\leq s, |x|\les 1$, we'd have $\min(||t-s| - |x||, |t-s|) \sim |t|$, a contradiction to the Case 2 assumption.}
%
%\younes{Yep, so we get a factor $\jb t^{-3 - \s}$ for free in front of our bounds.}

\medskip

\noi
{\bf $\bul$ Subcase 2.2: $|(h_1,h_2)| \ll \min(||t-s| - |x||, |t-s|)$.} In this case, we use~\eqref{kery13b} and the mean value theorem and~\eqref{kery5}-\eqref{kery6} in Lemma~\ref{lem:ker03} to obtain
\begin{align}
\begin{split}
& \int_{\R\times\T^2} |K^\star_{N}(t,s,x;h_1,h_2)| |\mc K_{3,\s}(h_1, h_2)| \ind_{|(h_1,h_2)| \ll \min(||t-s| - |x||, |t-s|)} dh_1 dh_2 \\
& \qquad \les \int_{\R\times\T^2} dh_1 dh_2 \frac{\ind_{|(h_1,h_2)| \ll \min(||t-s| - |x||, |t-s|)}}{|(h_1,h_2)|^{3+\s}} \\
& \qquad \qquad \quad \times \int_0^1 |\nb_{t,x} K^\star_{N}(t + \tau h_1, s,x+\tau h_2)| |(h_1,h_2)| d \tau \\
& \qquad  \les \int_{\R\times\T^2} dh_1 dh_2  \frac{\ind_{|(h_1,h_2)| \ll \min(||t-s| - |x||, |t-s|)}}{|(h_1,h_2)|^{2+\s}} \\
& \qquad \qquad  \times \int_0^1 \big(\mc R_{1, \frac12, \frac12}(t+\tau h_1-s, x+\tau h_2) + \mc R_{0, \frac12, \frac32}(t+\tau h_1-s, x+\tau h_2) \big) d \tau \\
& \qquad \les \big(\mc R_{1, \frac12, \frac12}(t-s, x) + \mc R_{0, \frac12, \frac32}(t-s, x) \big) \int_{\R\times\T^2}  \frac{\ind_{|(h_1,h_2)| \ll \min(||t-s| - |x||, |t-s|)}}{|(h_1,h_2)|^{2+\s}} dh_1 dh_2 \\
& \qquad \les \mc R_{\s, \frac12, \frac12}(t-s, x) + \mc R_{0, \frac12, \frac12+\s}(t-s, x).
\end{split}
\label{kery15}
\end{align}
In~\eqref{kery15}, we used that for any $s_1,s_2,s_3 \geq 0$, we have
\begin{align}
\mc R_{s_1,s_2,s_3}(t+h_1-s,x+h_2) \sim \mc R_{s_1,s_2,s_3}(t-s,x) \label{kerynoh}
\end{align}
for any $|(h_1,h_2)| \ll \min(||t-s| - |x||, |t-s|)$, which can be observed from a variant of~\eqref{kery8}-\eqref{kery9}. This finishes the proof of~\eqref{kery12}.
%\justin{From comment after Lemma \ref{lem:ker03}, we have $|s|\les 1$ and $|t+\tau h_1|\les 1$. Since $|h_1|\les 1$, we have $|t|\les 1$ in this case.}
%\younes{Yes, I'm already using this fact—see the first sentences in Case 2. This is why we have the bound~\eqref{kery13b} and we only have to decompose the variables $(h_1,h_2)$ along dyadic values which are $\les 1$ in the integrals of Subcase 2.1}
\end{proof}

We may now prove Proposition~\ref{prop:kery}.

\begin{proof}[Proof of Proposition~\ref{prop:kery}]
The bound~\eqref{kerygoal} for $\s=0$ follows from the triangle inequality with~\eqref{ker1b} and~\eqref{poisson4} and we thus consider the case $0<\s<\frac12$ in what follows. Recall the approximations of unity $\{\nu^1_N\}_{N\in \N}$ and $\{\nu^2_N\}_{N\in \N}$ used to define $\{K^\star_{\ld, N}\}_{N \in \N}$ in~\eqref{kery1}. Then, 
\begin{align*}
\nu^1_{N}\nu^2_{N} \ast_{t,x} G \to G 
\end{align*}
in $L^p(\R \times \T^2)$ as $N \to \infty$. Thus, by~\eqref{dtb2} and~\eqref{duha3}, in order to deduce~\eqref{kerygoal}, it suffices to prove 
\begin{align}
\Big\| \int_{\R} \{(-\Dl_{\R \times \T^2})^{\frac\s2} K^\star_{N}(\cdot, s,\cdot) \} \ast_x F(s,\cdot) ds\Big\|_{L^p(\R \times \T^2)} \les_\eps \|\mc A_{\s}(F)\|_{L^p(\R \times \T^2)},
\label{kerygoal2}
\end{align}
with an implicit constant uniform in $N \ge 1$. The bound~\eqref{kerygoal2} follows immediately from the pointwise estimate in Lemma~\ref{lem:ker1} and noting that, for any $s_1, s_2, s_3>0$, we either have
\begin{align*}
 \mc R_{s_1,s_2,s_3}(t,x) \les \mc R_{0,s_2,s_1+s_3}(t,x),
\end{align*}
if $|t| \ges |x|$, or 
\begin{align*}
 \mc R_{s_1,s_2,s_3}(t,x) \les |t|^{-s_1 - s_2} |x|^{-s_3} \, \jb{\log(|x|)}^3,
\end{align*}
if $|t| \ll |x|$.
\end{proof}

\subsection{Kernel estimates with additional damping}
We revisit the results of the previous subsection in the context of a stronger damping factor $e^{-\ld|\cdot |}$. The main reason for tracking this additional parameter is is the stability lemma (Lemma~\ref{LEM:stab}) appearing later on in the global well-posedness part. We stress that the results of the previous subsection suffice completely for proving the local well-posedness of \eqref{liouville} (Proposition~\ref{PROP:LWP}).

For $\ld\geq \frac 12$, we define the weighted Duhamel operator:
\begin{align}
\I_{\textup{wave},\ld}^\star(F) (t) = \int_{\R} e^{-\ld|t-s|} K^{\star}(t,s,\cdot) \ast_x F(s) ds,
\label{duha3ld}
\end{align}
the significance of which is that from \eqref{Ldweight}, we have 
\begin{align}
\| \I_{\textup{wave}}^\star(F) \|_{\Ld^{s,b}_{p,\ld}} = \| \I_{\textup{wave},\ld}^\star( e^{-\ld |\cdot|}F) \|_{\Ld^{s,b}_{p}}.
\label{ldnorm}
\end{align}
For $0\le \s<\frac12$, we define the sublinear operator $\mc A_{\s, \ld}$ given by
\begin{align}
\mc A_{\s,\ld}(F) (t) = \int_{\R} \mc R_{\s,\ld}(t,s,\cdot) \ast_x |F(s)| ds,
\label{subopld}
\end{align}
where the kernel $\mc R_{\s, \ld}$ is defined as
\begin{align}
\begin{split}
\mc R_{\s,\ld}(t,s,x) & =e^{-\frac{\ld}{200}|t-s|}\mc R_{\s}(t,s,x) +\jb{t}^{-3-\s} \ld^{-\frac 14}  \\
 & \quad +\jb{t}^{-3-\s} \ld^{-\dl} \frac{\jb{\log |x|}}{ |x|^{\frac 12}} \big( \mc R_{\frac 12+\s+2\dl,0,0}(t-s,x) + \mc R_{0, 0,\frac 12+\s +2\dl}(t-s,x)\big)
 \end{split} \label{Rsld}
\end{align}
for $0<\dl\leq \frac 12$.
We now give a version of Proposition~\ref{prop:kery} for $\I^{\star}_{\textup{wave},\ld}$.

\begin{proposition}\label{prop:keryld}
Let $\ld \ge \frac12$, $1 \le p < \infty$, $0\le \s,\dl<\frac12$, and $\I^{\star}_{\textup{wave},\ld}$ and $\mc A_{\s,\ld}$ be as in~\eqref{duha3ld} and~\eqref{subopld}, respectively. Then, 
\begin{align}
\|(-\Dl_{\R \times \T^2})^{\frac\s2} \I^{\star}_{\textup{wave},\ld}(F) \|_{L^p(\R \times \T^2)} \les \|\mc A_{\s,\ld}(F)\|_{L^p(\R \times \T^2)}.
\label{kerygoalld}
\end{align}
\end{proposition}

\begin{proof} 
It suffices to prove the following version of \eqref{kerygoal}: 
\begin{align}
\begin{split}
(-&\Dl_{\R \times \T^2})^{\frac \s2}[e^{-\ld |t-s|}K^\star_{N} (t,s,x)]   \\
&\les e^{-\frac{
\ld}{200}|t-s|} [\text{RHS} \eqref{kery12}] +\jb{t}^{-3-\s} \ld^{-\frac 14}  \\
 & \quad +\jb{t}^{-3-\s} \ld^{-\dl} \jb{\log |x|} |x|^{-\frac 12} \big( \mc R_{\frac 12+\s+2\dl,0,0}(t-s,x) + \mc R_{0, 0,\frac 12+\s +2\dl}(t-s,x)\big)
 \end{split} \label{kery12ld}
\end{align}
for all $(t,s,x) \in \R^2 \times \T^2$ and uniformly in $N\geq 1$. Indeed, once we have \eqref{kery12ld}, \eqref{kerygoalld} follows by the same arguments as in the proof of Proposition~\ref{prop:kery}.

Fix $(t,s,x)\in \R^2\times \T^2$ and let $f_{\ld}(t):= e^{-\ld|t|}$. Then, by \eqref{dtb}, we have
\begin{align}
\begin{split}
(-&\Dl_{\R \times \T^2})^{\frac \s2}[f_{\ld}(t-s) K^\star_{N} (t,s,x)]  \\
&= -c_{3,\s} \int_{\R\times\T^2} \big[ f_{\ld}(t-s+h_1)K^\star_{N}(t+h_1,s,x+h_2) \\
& \hphantom{XXXXXXXX}-f_{\ld}(t-s)K^{\star}_{N}(t,s,x)\big]  \mc K_{3,\s}(h_1, h_2) dh_1 dh_2
\end{split} \label{prodK}
 \end{align}
We follow the same case separation as in the proof of Lemma~\ref{lem:ker1}.

\smallskip
\noi
{\bf $\bul$ Case 1: $|h_1| \gg 1$.}
We write 
\begin{align}
\begin{split}
&\eqref{prodK}= -c_{3,\s} \int_{\R\times\T^2}f_{\ld}(t-s) K^\star_{N}(t,s,x;h_1,h_2) \mc K_{3,\s}(h_1, h_2) dh_1 dh_2 \\
& +c_{3,\s} \int_{\R\times\T^2} [f_{\ld}(t-s+h_1)-  f_{\ld}(t-s)] K^\star_{N}(t+h_1,s,x+h_2) \mc K_{3,\s}(h_1, h_2) dh_1 dh_2.
\end{split} \label{prodK2}
\end{align}
For the first term, we have precisely the same bound as in \eqref{kery12} with the additional factor $f_{\ld}(t-s)$. It then remains to control the second term.
By the support properties of $K^\star_{N}(t+h_1,s,x+h_2)$ and since $|h_1|\gg 1$, we have $|h_1|\sim |t|\gg |s|$. We reduce down to estimating
\begin{align}
\int_{\R\times\T^2} f_{\ld}(t-s+h_1) K^\star_{N}(t+h_1,s,x+h_2) \mc K_{3,\s}(h_1, h_2) dh_1 dh_2. \label{prodK3}
\end{align}
By arguing as in~\eqref{kery106}, this term is bounded by
\begin{align*}
\frac{1}{\jb{t}^{3+\s}} \int_{|t-s+h_1|\les 1} \frac{ e^{-\ld|t-s+h_1|}}{ |t-s+h_1|^{\frac 12+\eps}} dh_1 \les \frac{1}{\jb{t}^{3+\s}} \frac{1}{\ld^{\frac 14}},
\end{align*}
where we used H\"{o}lder in the second inequality.
\smallskip

\noi
{\bf $\bul$ Case 2: $|h_1| \les 1$.}
\smallskip

\noi
{\bf $\bul$ Subcase 2.1: $\min(||t-s| - |x||, |t-s|) \les |(h_1,h_2)|$.} 

As in Case 1, we reduce to controlling the term \eqref{prodK3} appearing in \eqref{prodK2}, as the remaining terms differ from those already considered in \eqref{kery13c} by the additional $(h_1,h_2)$-independent factor $f_{\ld}(t-s)$.
Moreover, we may assume that $\ld |h_1|\gg  1$ since otherwise, if $\ld |h_1|\les 1$, we use
$
f_{\ld}(t-s+h_1)=e^{-\ld|t-s+h_1|} \leq e^{-\ld|t-s|}e^{\ld |h_1| } \les f_{\ld} (t-s),
$
which is an acceptable upper bound that does not depend on $(h_1,h_2)$.

As in \eqref{kery13c}, we have
 \begin{align*}
& \int_{\R\times\T^2} f_{\ld}(t-s+h_1) |K^\star_{ N}(t+h_1,s,x+h_2)||\mc K_{3,\s}(h_1, h_2)|\\
& \qquad\qquad \qquad\qquad\times  \ind_{\min(||t-s| - |x||, |t-s|) \les |(h_1,h_2)| \les 1} \ind_{|h_1|\les 1} dh_1 dh_2 \\
& \quad  \les \sum_{\l=0}^{\infty} \ind_{2^\l \les \max(||t-s| - |x||^{-1}, |t-s|^{-1})} 2^{\l \s} \\
& \quad \qquad  \times \int_{\R\times\T^2} f_{\ld}(t-s+h_1) | K^\star_{N}(t+h_1,s,x+h_2)| 2^{3\l} \ind_{|h_1| \les 2^{-\l}} \ind_{|h_2| \les 2^{-\l}} dh_1 dh_2.
\end{align*}
As $\ld |h_1|\gg 1$, we have $2^{\l} \ll \ld$. In view of \eqref{keryK} and that $|t-s+h_1|\les 1$, since otherwise $K^\star_{N}(t+h_1,s,x+h_2) =0$, we have 
\begin{align*}
&\int_{\R\times\T^2} f_{\ld}(t-s+h_1) | K^\star_{N}(t+h_1,s,x+h_2)| 2^{3\l} \ind_{|h_1| \les 2^{-\l}} \ind_{|h_2| \les 2^{-\l}} dh_1 dh_2  \\
& \quad \les \int_{\R \times \T^2} f_{\ld}(t-s+h_1)  \frac{\jbb{\log(||t-s+h_1| - |x+h_2||)}}{\big|(t-s+h_1)^2 - |x+h_2|^2\big|^{\frac12}} \ind_{|t-s+h_1| \les 1} \\
& \qquad \qquad \quad \times \jb{\log|x+h_2|} \, 2^{3\l} \ind_{|h_1| \les 2^{-\l}} \ind_{|h_2| \les 2^{-\l}} dh_1 dh_2.
\end{align*}
We again denote by $\1_{\l}$ and $\II_{\l}$ the contributions coming from $ \big| |t-s+h_1|-|x+h_2|\big| <2^{-100\l}$ and  $ \big| |t-s+h_1|-|x+h_2|\big| >2^{-100\l}$, respectively. By passing to polar coordinates and using the mean value theorem,  we get 
\begin{align*}
\1_{\l} & \les  2^{3\l} \int_{|t-s+h_1|\les 1} \frac{f_{\ld}(t-s+h_1)}{ |t-s+h_1|^{\frac 12}} \int_{0}^{10} \frac{ \jb{\log r} }{| |t-s+h_1|-r|^{\frac 34}} \ind_{| | t-s+h_1|-r|<2^{-100\l}} dr dh_1 \\
& \les 2^{-24\l}\int_{|t-s+h_1|\les 1} \frac{f_{\ld}(t-s+h_1)}{ |t-s+h_1|^{\frac 12}} dh_{1} \\
&\les 2^{-24\l} \sum_{j=0}^{\infty} 2^{\frac{j}{2}} \int_{2^{-j-1}}^{2^{-j+1}} e^{-\ld h} dh \\
& \les 2^{-24\l} \sum_{j=0}^{\infty} 2^{\frac j2} 2^{-\frac 34 j} \ld^{-\frac 14} \\
& \les 2^{-24\l} \ld^{-\frac 14},
\end{align*}
Thus, 
\begin{align*}
\sum_{\l=0}^{\infty} \ind_{2^l \les \max(||t-s| - |x||^{-1}, |t-s|^{-1})} 2^{\l \s}  \1_{\l} \les \ld^{-\frac 14}.
\end{align*}

We now consider the contribution from $\II_{\l}$. We have
\begin{align*}
\II_{\l} & \les \l\,  2^{3\l} \int_{\T^2}  \frac{ \jb{\log|x+h_2|}}{ |x+h_2|^{\frac 12}} \ind_{|h_2|\les 2^{-\l}}  \int_{|h_1|\les 2^{-\l}} \frac{f_{\ld}(t-s+h_1)}{  | |t-s+h_1|-|x+h_2||^{\frac 12}} dh_1 dh_2
\end{align*}
For the inner integral over $h_1$, we decompose:
\begin{align*}
 \int_{|h_1|\les 2^{-\l}} \frac{f_{\ld}(t-s+h_1)}{  | |t-s+h_1|-|x+h_2||^{\frac 12}} dh_1& \les \sum_{k=0}^{100\l} 2^{\frac{k}{2}} \int_{| r-|x+h_2|| \sim 2^{-k}} f_{\ld}(r) dr \\
 &\les \sum_{k=0}^{100\l} 2^{\frac{k}{2}} 2^{-\frac{k}{2}}\ld^{-\frac 12}   \les \l \ld^{-\frac 12} .
\end{align*}
Here in the second inequality, we split into the cases $r<|x+h_2|$ and $r\geq |x+h_2|$. 
When $r\geq |x+h_2|$, we have rapid decay in $\ld 2^{-k}$: 
\begin{align*}
\int_{| r-|x+h_2|| \sim 2^{-k}} \ind_{r\geq |x+h_2|} f_{\ld}(r) dr =\ld^{-1}f_{\ld}(x+h_2) [ e^{-\frac{1}{4}\ld 2^{-k}} - e^{-4\ld 2^{-k}}] \les \ld^{-\frac 12} 2^{-\frac{k}{2}} f_{\ld}(x+h_2).
\end{align*}
If instead $r<|x+h_2|$, then since $r\geq 0$, we also have $|x+h_2|>2^{2-k}$ and so by the mean value theorem, we only have
\begin{align*}
\int_{| r-|x+h_2|| \sim 2^{-k}} \ind_{r< |x+h_2|} f_{\ld}(r) dr &= \int_{|x+h_2|-2^{2-k}}^{|x+h_2|-2^{-2-k}} f_{\ld}(r)dr    \\
&=\ld^{-1}[ e^{-\ld( |x+h_2|-2^{-2-k})}- e^{-\ld( |x+h_2|-2^{2-k})}] \\
&\les \ld^{-\frac 12}2^{-\frac{k}{2}}.
\end{align*}
Therefore, 
\begin{align*}
\II_{\l} \les \ld^{-\frac 12} \l^2 2^{3\l} \int_{|h_2|\les 2^{-\l}} \frac{ \jb{\log |x+h_2|}}{|x+h_2|^{\frac 12}} dh_2.
\end{align*}
We decompose this integral into two cases if $|x|\les |h_2|$ or $|x|\gg |h_2|$, denoting the contributions by $\II_{\l}^{(1)}$ and $\II^{(2)}_{\l}$, respectively.
 If $|x|\les |h_2|$, then $|x+h_2|\les 2^{-\l}$, so that by a change of variables and moving to polar coordinates
\begin{align*}
\II^{(1)}_{\l} &=\ld^{-\frac 12}\l^2 2^{3\l} \int_{|h_2|\les 2^{-\l}} \ind_{|x|\les |h_2|}  \frac{\jb{\log |x+h_2|}}{|x+h_2|^{\frac 12}} dh_2 \\
& \les \ld^{-\frac 12} \l^2 2^{3\l}\int_{|h|\les 2^{-\l}} \frac{\jb{\log |h|}}{|h|^{\frac 12}}  dh \\
& \les \ld^{-\frac 12}  \l^2 2^{3\l} \sum_{k=\l}^{\infty} 2^{\frac k2}k\int_{r\sim 2^{-k}} 1 dr \\
& \les \ld^{-\frac 12}  \l^2 2^{3\l}  \sum_{k=\l}^{\infty} k 2^{-\frac 32 k} \\
&\les \ld^{-\frac 12} \l^2 2^{(\frac 32  +)\l}.
\end{align*}
 Therefore, by recalling that $\ld \gg 2^{\l}$,  and since $|x|\les 2^{-\l}$ implies that $2^{\l}\les |x|^{-1}$, we have
\begin{align*}
\sum_{\l=0}^{\infty} \ind_{2^l \les \max(||t-s| - |x||^{-1}, |t-s|^{-1})} 2^{\l \s} \II_{\l}^{(1)} & \les \ld^{-\dl} \sum_{\l=0}^{\infty} \l^2 2^{(1+\s+\dl+)\l} \ind_{2^\l \les \max(||t-s| - |x||^{-1}, |t-s|^{-1})}  \\
& \les \ld^{-\dl} |x|^{-\frac 12} \big[ \mc R_{\frac 12+ \s+2\dl, 0,0}(t-s,x) + \mc R_{0,0,\frac12 +\s+2\dl}(t-s,x)]
\end{align*}
for any $0<\dl \leq \frac 12$.

For the contribution $\II_{\l}^{(2)}$, we have $|x|\gg |h_2|$, so
\begin{align*}
\II_{\l}^{(2)}\les \ld^{-\frac 12}\l^2 2^{3\l} \frac{\jb{\log |x|}}{|x|^{\frac 12}} \int_{|h_2|\les 2^{-\l}}1 dh_2 \les \ld^{-\frac 12}\l^2 2^{\l} \frac{\jb{\log |x|}}{|x|^{\frac 12}} \les \ld^{-\dl} \l^2 2^{(\frac 12+\dl)\l} \frac{\jb{\log |x|}}{|x|^{\frac 12}}.
\end{align*}
Thus, 
\begin{align*}
\sum_{\l=0}^{\infty} \ind_{2^l \les \max(||t-s| - |x||^{-1}, |t-s|^{-1})} &2^{\l \s} \II_{\l}^{(2)}  \les \ld^{-\dl} \sum_{\l=0}^{\infty} \l^2 2^{(\frac 12+\s+\dl+)\l} \ind_{2^l \les \max(||t-s| - |x||^{-1}, |t-s|^{-1})} \frac{\jb{\log |x|}}{|x|^{\frac 12}}  \\
& \les \ld^{-\dl} \frac{\jb{\log |x|}}{|x|^{\frac 12}} \big[ \mc R_{\frac 12+ \s+2\dl, 0,0}(t-s,x) + \mc R_{0,0,\frac12 +\s+2\dl}(t-s,x)]
\end{align*}
for any $0<\dl \leq \frac 12$.

\smallskip

\noi
{\bf $\bul$ Subcase 2.2: $|(h_1,h_2)| \ll \min(||t-s| - |x||, |t-s|)$.} 
As $|h_1|\ll |t-s|$, we have 
\begin{align*}
f_{\ld}(t-s+\tau h_1) \leq f_{\ld/100}(t-s) \quad \text{and} \quad |\dt f_{\ld}(t-s+\tau h_1) |\les \ld f_{\ld/100}(t-s)
\end{align*}
 for any $\tau\in [0,1]$.
By \eqref{kery13b}, the mean value theorem, \eqref{kery5}-\eqref{kery6}, \eqref{kerynoh} we have 
\begin{align*}
|\eqref{prodK}| \les 
&  \int_{\R\times\T^2} dh_1 dh_2 \frac{\ind_{|(h_1,h_2)| \ll \min(||t-s| - |x||, |t-s|)}}{|(h_1,h_2)|^{3+\s}} \\
& \quad  \times \int_0^1 \Big( |(\dt f_{\ld})(t-s+\tau h_1)| |K^{\star}_{N}(t+\tau h_1, s,x+\tau h_2)|\\
& \hphantom{XXXXX} +  f_{\ld}(t-s+\tau h_1) |\nb_{t,x} K^{\star}_{N}(t+\tau h_1,s,x+\tau h_2)| \Big)   |(h_1,h_2)| d \tau  \\
& \les \int_{\R\times\T^2} dh_1 dh_2 \frac{\ind_{|(h_1,h_2)| \ll \min(||t-s| - |x||, |t-s|)}}{|(h_1,h_2)|^{2+\s}}\\ &\hphantom{XX}\times  f_{\ld/100}(t-s)\bigg( \ld \mc R_{0,\frac 12, \frac 12}(t-s,x)  +  \big(  \mc R_{1,\frac 12, \frac 12}(t-s,x)+ \mc R_{0,\frac 12, \frac 32}(t-s,x)\big)\bigg) \\
& \les f_{\ld/100}(t-s)\big( \mc R_{\s,\frac 12, \frac 12}(t-s,x) + \mc R_{0,\frac 12,\frac 12+\s}(t-s,x)\big)) \\
&\quad   + \ld f_{\ld/100}(t-s) |t-s|^{1-\s} \mc R_{0,\frac 12,\frac 12}(t-s,x) \\
& \les f_{\ld/200}(t-s)\big( \mc R_{\s,\frac 12, \frac 12}(t-s,x) + \mc R_{0,\frac 12,\frac 12+\s}(t-s,x)\big)),
\end{align*}
where in the last inequality we used $\ld e^{-a\ld}\les a^{-1}$ for any $\ld, a>0$.

Collecting all the cases establishes \eqref{kery12ld} and thus completes the proof.
\end{proof}

\section{Stochastic objects}\label{sec:sto}

\subsection{The stochastic convolution} \label{subsec:stochconv}

We recall the following regularity properties of the stochastic convolution. See \cite[Proposition 2.1]{GKO} for a proof.

\begin{lemma}\label{LEM:stick}
Given any $T,\eps>0$ and finite $p\geq 1$, $\{\Psi_{N}\}_{N\in \N}$ is a Cauchy sequence in $L^p(\O; C([0,T];W^{-\eps, \infty}(\T^2)))$, converging to some limit $\Psi$ in $L^p(\O; C([0,T];W^{-\eps, \infty}(\T^2)))$.
\end{lemma}

Given $N_1,N_{2}\in \N$, we set 
\begin{align}
\G_{N_1,N_2}(t_1, t_2, x_1,x_2)= \E\big[ \Psi_{N_1}(t_1,x_1) \Psi_{N_2}(t_2,x_2)\big] \label{cov}
\end{align}
for any $(t_1,x_1), (t_2,x_2)\in \R_{+}\times \T^2$. As $\Psi_{N}(t,x)$ is spatially homogeneous, we have 
\begin{align*}
\G_{N_1,N_2}(t_1,t_2, x_1,x_2)=  \E\big[ \Psi_{N_1}(t_1,x_1-x_2) \Psi_{N_2}(t_2,0)\big] =\G_{N_1,N_2}(t_1,t_2,x_1-x_2,0).
\end{align*}
With an abuse of notation, we define $\G_{N_1,N_2}(t_1,t_2,y)$ as $\G_{N_1,N_2}(t_1,t_2,y,0)$ for any $(t_1,y), (t_2,y)\in \R_{+}\times \T^2$. Moreover, if $N_1=N_2$, then we define $\G_{N_1}(t_1,t_2, y) = \G_{N_1,N_1}(t_1,t_2,y)$.
We recall the following result from \cite[Proposition 5.5]{Zine3} which establishes sharp estimates on the space-time covariances $\G_{N_1,N_2}$.

\begin{lemma}\label{LEM:GN}
Given $N\in \N$, let $\G_{N}$ be as in \eqref{GN}. Then, 
\begin{align}
\G_{N}(t_1,t_2,x) \approx -\tfrac{1}{2\pi}\log( |t_1-t_2|+|x|+N^{-1}) \label{cov1}
\end{align}
for any $(t_1,t_2,x)\in [0,1]^{2}\times \T^2$. Moreover, given $N_{1},N_2 \in \N$, let $\G_{N_1,N_2}$ be as in \eqref{cov}. Then, 
\begin{align}
\G_{N_1,N_2}(t_1,t_2, x) \approx -\tfrac{1}{2\pi} \log( | t_1-t_2|+|x|+(N_1 \wedge N_2)^{-1}) \label{cov2}
\end{align}
and 
\begin{align}
\begin{split}
|\G_{N_j}(t_1,t_2, x) - \G_{N_1,N_2}(t_1,t_2, x)| 
 \les \min\big(&  1 \vee (-\log(|t_1-t_2|+|x|+(N_1\wedge N_2)^{-1})), \\
&\, (N_1\wedge N_2)^{-\frac 12}|x|^{-\frac 12}) + O( (N_1\wedge N_2)^{-1})\big)
\end{split} \label{cov3}
\end{align}
for any $(t_1,t_2,x)\in [0,1]^{2}\times \T^2$, $N_{1},N_{2}\in \N$ and $j=1,2$.
\end{lemma}

\subsection{The Gaussian multiplicative chaos} \label{subsec:gmc}

We recall the following regularity and convergence properties of the Gaussian multiplicative chaoses $\Dr$.

\begin{lemma}\label{LEM:GMC} The following statements hold:

\smallskip

\noi
\textup{(i)} Given $0<\be^{2}<8\pi$, let $1\leq p<\frac{8\pi}{\be^2}$ and define $\al=\al(p)$ such that 
\begin{align}
\frac{(p-1)\be^{2}}{4\pi} < \al <2. \label{alpha}
\end{align}
Then, given any $T>0$, the sequence of stochastic processes $\Dr_{N}$ is uniformly bounded in $L^{p}(\O; L^{p}([0,T];W^{-\al,p}(\T^2)))$.

\smallskip

\noi
\textup{(ii)} Given $0<\be^2<4\pi$, let $1\leq p<\frac{8\pi}{\be^2}$ and $\al(p)$ be as in \eqref{alpha}. Then, given any $T>0$, $\{\Dr_{N}\}_{N\in \N}$ is a Cauchy sequence in $L^{p}(\O; L^{p}([0,T];W^{-\al,p}(\T^2)))$ and converges to a limit $\Dr$. 
\end{lemma}

\noi
For a proof of Lemma~\ref{LEM:GMC}, we refer to \cite[Proposition 1.12]{ORW}. It is clear that the same properties in Lemma~\ref{LEM:GMC} hold for the $\cosh$ and $\sinh$ variants $\{\Dr^{a}\}_{a\in \{c,s\}}$ in \eqref{GMCcs} 
One of the main ingredients in the proof of Lemma~\ref{LEM:GMC} is the following convexity result, originally due to Kahane.

\begin{lemma} \label{LEM:Kahane}
Given $n\in \N$, let $\{X_{j}\}_{j=1}^{n}$, $\{Y_{j}\}_{j=1}^{n}$ be two centered Gaussian vectors satisfying 
\begin{align*}
\E[ X_{i} X_{j}]\leq \E[Y_{j}Y_{k}]
\end{align*}
for all $j,k=1,\ldots, n$. Then, for any sequence $\{p_{j}\}_{j=1}^{n}$ of non-negative numbers and any convex function $F:[0,\infty)\to \R$ with at most polynomial growth at infinity, it holds that 
\begin{align*}
\E \bigg[ F \bigg(  \sum_{j=1}^{n} p_{j} e^{X_{j}- \frac 12 \E[X_j^2]} \bigg)\bigg] \leq \E \bigg[ F \bigg(  \sum_{j=1}^{n} p_{j} e^{Y_{j}- \frac 12 \E[Y_j^2]}\bigg)\bigg].
\end{align*}
\end{lemma}

Given $A\in \mathcal{B}(\T^2)$, where $\mathcal{B}(\T^2)$ is the Borel sigma-algebra in $L^{2}(\T^2)$, and $t\geq 0$, we define the random measure 
\begin{align*}
\mathcal{M}_{N}(t,A) = \int_{A}\Dr_{N}(t,x)dx.
\end{align*}
Then, the following moment bounds hold. See \cite[Proposition 3.5]{RV} and \cite[Appendix B]{ORW}.

\begin{lemma}\label{LEM:GMCmeas}
For any $0<\be^{2}<8\pi$ and $1\leq p <\frac{8\pi}{\be^2}$, it holds that 
\begin{align}
\sup_{t\in \R_{+}, A\in \mathcal{B}(\T^2), N\in \N} \E \big[ \mathcal{M}_{N}(t,A)^{p}\big] <\infty. \label{GMCmeas1}
\end{align}
Moreover, it holds that 
\begin{align}
\sup_{t\in \R_{+}, x\in \T^2, N\in \N} \E \big[ \mathcal{M}_{N}(t,B(x,r))^{p}\big] \les r^{(2+\frac{\be^2}{4\pi})p-\frac{\be^2}{4\pi}p^2},
\label{GMCmeas2}
\end{align} 
for any $r\in (0,1)$, where $B(x,r)$ is the ball centered at $x\in \T^2$ of radius $r\in (0,1)$.
\end{lemma}
\begin{proof}
For \eqref{GMCmeas1}, we refer to \cite[Proposition 3.5]{RV} and \cite[Appendix B]{ORW}.
The estimate \eqref{GMCmeas2} in the case of the ball can be found in \cite[(3.8), p.1318]{ORW}. 
\end{proof}

The next result concerns moments of the GMC measure on annuli, which will be a key input in obtaining uniform bounds on the Gibbs state for the hyperbolic sinh-Gordon model.

\begin{lemma} \label{LEM:GMCannuli}
Let $2\leq p<\infty$ be an even integer and $0<\be^{2}<\frac{4\pi}{p-1}$. Given $0<r_1<r_2<1$, let $A(r_1,r_2)$ be the annulus $\{r_1 \leq |x| \leq r_2\}$. Then,
\begin{align}
\sup_{t\in \R_+}\sup_{N\in \N}\E[ \mathcal{M}_{N}( t, A(r_1,r_2))^{p}] \les (  r_2( r_2-r_1))^{p-\frac{\be^{2}}{8\pi}(p-1)p}.\label{GMCann}
\end{align}
\end{lemma}

Notice that the right-hand side of \eqref{GMCann} agrees with the moments on balls \eqref{GMCmeas2} as one takes $r_1 \to 0^{+}$.

The assumption of even integer $p$ in Lemma~\ref{LEM:GMCannuli} can likely be removed by a more refined analysis based on Kahane's convexity inequality (Lemma~\ref{LEM:Kahane}).
 We chose not to pursue this direction for two reasons: first, the restriction to even integer $p$ will not considerably alter the restriction on $\be^2$ in Lemma~\ref{LEM:PuN}. Second, for even integer $p$, we can provide an elementary argument based on directly computing the even integer moments using the geometric Brascamp-Lieb inequality \cite{BCCT}. Note that the geometric Brascamp-Lieb inequality was first used in the context of the GMC measure in \cite{ORW}, to provide an alternative argument for computing the even moments of the GMC. See \cite[Proposition 3.2]{ORW}.

We use the following version of the geometric Brescamp-Lieb inequality; see~\cite{BCCT} for a proof.

\begin{lemma} \label{LEM:BL}
Let $p\in \N$. Then, we have
\begin{align*}
\int_{(\T^2)^{2p}} &\prod_{1\leq j <k\leq 2p} |f_{j,k}(\pi_{j,k}(x))|^{\frac{1}{2p-1}}dx \les  \prod_{1\leq j<k\leq 2p} \bigg( \int_{(\T^2)^{2}} |f_{j,k}(x_j,x_k)| dx_j dx_k\bigg)^{\frac{1}{2p-1}}
\end{align*}
for any $f_{j,k}\in L^1(\T^2 \times \T^2)$ and where $\pi_{j,k}$ denotes the projection defined by $\pi_{j,k}(x) = \pi_{j,k}(x_1,\ldots, x_{2p}) = (x_j, x_k)$ for $x\in (\T^2)^{2p}$.
\end{lemma}

\begin{proof}[Proof of Lemma~\ref{LEM:GMCannuli}]
Assume that $p=2m$ for some $m\in \N$. Then, by Lemma~\ref{LEM:BL}, we have \begin{align}
\E&\bigg[ \bigg( \int_{\{r_1 \leq |y| \leq r_2\}} e^{\be \Psi_{N}(t,y)-\frac{\be^{2}}{2}\s_N}  dy\bigg)^p \bigg] \notag  \\
 & = e^{-m\be^2 \s_{N}}\E\bigg[ \bigg( \int_{\{r_1 \leq |y| \leq r_2 \}} e^{\be \Psi_{N}(t,y)}  dy\bigg)^{2m} \bigg]   \notag\\
& =e^{-m\be^2 \s_{N}} \int_{ \cap_{j=1}^{2m} \{ r_1 \leq |x_j| \leq r_2\}} \E\Big[ \exp\Big( \be \sum_{j=1}^{2m} \Psi_{N}(t,x_j)\Big) \Big] dx_1 \cdots dx_{2m} \notag \\
& \les  \int_{\R^{2m}}  \prod_{1\leq j<k\leq 2m} \frac{ \ind_{\{r_1 \leq |x_j|\leq r_2\}}\ind_{\{r_1\leq |x_k|\leq r_2\}}  }{ (|x_j-x_k|+N^{-1})^{\frac{\be^2}{2\pi}}} dx_1 \cdots dx_{2m} \notag \\
&\les \prod_{1\leq j<k\leq 2m}  \bigg( \int_{\R^2}   \frac{ \ind_{\{r_1 \leq |x_j|\leq r_2\}}\ind_{\{r_1\leq |x_k|\leq r_2\}}  }{ (|x_j-x_k|+N^{-1})^{\frac{\be^2}{2\pi}(2m-1)}} dx_j dx_k \bigg)^{\frac{1}{2m-1}} \notag \\
&\les   \bigg( \int_{\R^2}   \frac{ \ind_{\{r_1\leq |x|\leq r_2\}}\ind_{\{r_1 \leq |y|\leq r_2\}}  }{ (|x-y|+N^{-1})^{\frac{\be^2}{2\pi}(2m-1)}} dy dx \bigg)^{m}.  \label{GMCann1}
\end{align}
Let $a = |A(r_1,r_2)| =\pi (r_2^2-r_1^2)$. 
To estimate the inner integral, we the bound above by 
\begin{align*}
\int &\ind_{\{r_1\leq |y|\leq r_2\}} \int \frac{\ind_{\{r_1\leq |x|\leq r_2\}}}{|x-y|^{\frac{\be^2}{2\pi}(2m-1)}}dx dy  \\ 
& \les a \sup_{y\in \T^2}  \sum_{j=0}^{\infty} 2^{\frac{\be^2}{2\pi}(2m-1)j}  \int \ind_{\{r_1\leq |x|\leq r_2\}}  \ind_{\{|x-y|\sim 2^{-j}\}}  dx \\
& \les a \sum_{j=0}^{\infty} 2^{\frac{\be^2}{2\pi}(2m-1)j} \min(2^{-2j}, a  ) \\
& \les a \bigg\{  \sum_{j>\frac{1}{2}\log_{2}(1/a)} 2^{(\frac{\be^2}{2\pi}(2m-1)-2)j} + \sum_{j<\frac{1}{2}\log_{2}(1/a)} a \, 2^{\frac{\be^2}{2\pi}(2m-1)j}  \bigg\} \\
&\les a^{2-\frac{\be^2}{4\pi}(2m-1)},
\end{align*}
where we used that $\be^{2}<\frac{4\pi}{2m-1}$ for convergence of the first summation. Inserting this bound into \eqref{GMCann1} then establishes \eqref{GMCann} for even integer $p\geq 2$.
\end{proof} 

\begin{remark}\rm The above argument is valid for any Lebesgue measurable. Namely, we have
\[
\sup_{t\in \R_+}\sup_{N\in \N}\E[ \mathcal{M}_{N}( t, A)^{p}] \les |A|^{p-\frac{\be^{2}}{8\pi}(p-1)p}
\]
for any measurable set $A \subset \R^2$.

We conjecture that the same bound holds for all $2 \leq p < \infty$. We note, however, that this would lead only to a very minor improvement in the admissible range of $\be^2$ in Theorem~\ref{thm:main}, and we therefore restrict our attention to the even integer case, for which a simpler proof is available.
\end{remark}

Next, we construct the limiting Gibbs measure as the limit of the renormalized truncated measure \eqref{gibbsN}. A precise statement is given below and we refer to \cite[Proposition 1.4]{ORW} for a proof.

\begin{lemma}\label{lem:gibbs}
Let $\iota >0$ and $\be \in \R$ with $0 < \be^2 < 4\pi$. Then, for any $1\leq p <\frac{8\pi}{\be^2}$, 
the density 
\begin{align*}
R_{N}(u) = \exp\bigg( - \iota \g_{N}\int_{\T^2}\cosh(\be \Pii_{N}u)dx \bigg)
\end{align*}
is uniformly bounded in $N$  \textup{(}by $1$\textup{)}  and converges to a limit $R$ in $L^p (\muu_1)$ as $N\to \infty$. 
Furthermore, the sequence $\{\rhoo_N\}_{N \in \N}$ converges in total variation to some limiting probability measure $\rhoo$ on $\H^{-\eps}(\T^2)$ for any small $\eps >0$, and the measures $\rhoo$ and $\muu_1$ are mutually absolutely continuous.
\end{lemma}

\subsection{The modified Gaussian multiplicative chaos}

We define the following modified Gaussian multiplicative chaos adapted to the hyperbolic setting:
\begin{align}
\mathcal{P}_{\kk,\al}[\Dr_{N}] (t,x;\be) : =  \int_{\T^2}\int_{0}^{1} \K_{\kk,\al}( t,t',x-y) \Dr_{N}(t',y;\be) dt' dy, \label{HGMC}
\end{align}
where the kernel $\K_{\kk,\al}$ is given by
\begin{align}
\K_{\kk,\al}(t,t'x)=   \frac{1}{\jb{t}^{3}} \frac{1}{|x|^{2-\al}} \bigg[  1+ \frac{1}{|t-t'|^{\kk}} + \frac{1}{ ||t-t'|-|x||^{\kk}}     \bigg] \label{K}
\end{align}
for $\kk \geq0 $ and $0< \al < 2$ and $t\in \R$. 
The form of \eqref{HGMC} is motivated by the kernel $\mc R_{\s}$ appearing in \eqref{subop} and Proposition~\ref{prop:kery}.
Indeed, by Young's inequality, we have
\begin{align}
\mc R_{\s}(t,t',x) & = \jb{t}^{-3-\s} \big(1 +   |t-t'|^{-\frac12-\s} |x|^{-\frac12}  \notag\\
& \qquad + (|t-t'| + |x|)^{-\frac12} ||t-t'| - |x||^{-\frac12-\s}\big)  \jbb{\log(||t-t'|- |x||)}^2 \, \jb{\log(|x|)}^3 \notag \\
&\les_{\eps} \K_{\frac{1}{2}+\s+\eps, \frac{3}{2}-\eps}(t,t',x) \label{Rsbd}
\end{align}
for any $(t,t',x) \in \R \times [0,1] \times \T^2$ and $\eps>0$. 
We show that $\mathcal{P}_{\kk, \al}[\Dr_{N}]$ is uniformly bounded in $L^{p}(\R\times \T^2)$ for the natural range of $p$ depending on $(\be, \al)$ and converges in $L^{p}(\O; L^{p}(\R\times \T^2))$ as $N\to \infty$. 

\begin{proposition}\label{PROP:I1theta}
The following statements hold:

\smallskip

\noi
\textup{(i)} Given $0<\be^{2}<8\pi$, let $1\leq p<\frac{8\pi}{\be^2}$, $\al>0$  be as in \eqref{alpha} and $0<\kk<1$.
Then, the sequence of stochastic processes $\mathcal{P}_{\kk,\al}[\Dr_{N}]$ is uniformly bounded in $L^{p}(\O; L^{p}(\R\times \T^2))$. 

\smallskip

\noi
\textup{(ii)} Given $0<\be^{2}<4\pi$, let $1<p<\frac{8 \pi}{\be^2}$ and $\al$ be as in \eqref{alpha}. Then, given any $T>0$ and $0<\kk<1$, the sequence $\{\mathcal{P}_{\kk,\al}[\Dr_{N}]\}_{N\in \N}$ is Cauchy in $L^{p}(\O; L^{p}(\R\times \T^2))$ and thus converges to a limit, which we denote by $\mathcal{P}_{\kk, \al}[\Dr]$, in $L^{p}(\O; L^{p}(\R\times \T^2))$. 
%Moreover, it holds that 
%\begin{align}
%\E\Big[ \sup_{N\in 2^{\N_0}\cup \{\infty\}} \| \H_{\kk,\al}[\Dr_{N}]\|_{L^{p}_{t,x}([0,1]\times \T^2)}^{p}  \bigg] \leq C. \label{unifHbd}
%\end{align}
\end{proposition}

Before we give a proof of Proposition~\ref{PROP:I1theta}, we first establish the following preliminary results.

\begin{lemma}\label{LEM:I1theta}
Let $N\in \N$, $0\leq \kk<1$, $1\leq p <\frac{8\pi}{\be^{2}}$ and $\al$ satisfy \eqref{alpha}. Then, 
\begin{align}
\sup_{N\in \N} \E \Big[ \big\|\mathcal{P}_{\kk, \al}[\Dr_{N}] \big\|_{L^{p}_{t,x}(\R\times \T^2)}^{p} \Big] <\infty. \label{I1theta}
\end{align}
Furthermore, for $0<\be^{2}<4\pi$, it holds that 
\begin{align}
\E \Big[ \big\|\mathcal{P}_{\kk,\al}[\Dr_{N_1}]-\mathcal{P}_{\kk, \al}[\Dr_{N_2}] \big\|_{L^{2}_{t,x}(\R\times \T^2)}^{2} \Big] \leq C (N_1 \wedge N_2)^{-\ta}. \label{I1thetadiff}
\end{align}
for  some $\ta=\ta(\be,\al)>0$ and any $N_{1},N_{2}\in \N$. 
\end{lemma}
\begin{proof}
The case $p=1$ follows from Fubini's theorem, the fact that 
\begin{align*}
\E[ |\Dr_{N}(t,x)|] = 1
\end{align*}
uniformly in $N\in \N$ and for any $t\in \R_{+}$ and $x\in \T^2$, and that $0\leq \kk<1$ and $\al>0$. Thus, we assume that $p>1$.
By the triangle and Minkowski's inequality, \eqref{I1theta} is directly implied by the following results: 
\begin{align}
\sup_{x\in \T^2, t\in \R, N\in \N} \E \bigg[ \bigg( \int_{\T^2}\int_{0}^{1}  \frac{1}{| |t-t'|-|x-y||^{\kk}} \frac{\Dr_{N}(t',y)}{|x-y|^{2-\al}} dt' dy \bigg)^{p} \bigg] <\infty, \label{I1thetagoal1} \\
\sup_{x\in \T^2, t\in \R, N\in \N} \E \bigg[ \bigg( \int_{\T^2}\int_{0}^{1}  \frac{1}{ |t-t'|^{\kk}} \frac{\Dr_{N}(t',y)}{|x-y|^{2-\al}} dt' dy \bigg)^{p} \bigg] <\infty. \label{I1thetagoal2}\\
\sup_{x\in \T^2, t\in \R, N\in \N} \E \bigg[ \bigg( \int_{\T^2}\int_{0}^{1}   \frac{\Dr_{N}(t',y)}{|x-y|^{2-\al}} dt' dy \bigg)^{p} \bigg] <\infty. \label{I1thetagoal3}
\end{align}
We only provide details for the (slightly) more involved \eqref{I1thetagoal1}. Similar and simpler arguments show \eqref{I1thetagoal2} and \eqref{I1thetagoal3}.

Fix $x\in \T^2$ and $t\in \R$. By the positivity of $\Dr_{N}$ and Minkowski's inequality, we have 
\begin{align}
& \E \bigg[ \bigg( \int_{\T^2}\int_{0}^{1}  \frac{1}{| |t-t'|-|x-y||^{\kk}} \frac{\Dr_{N}(t',y)}{|x-y|^{2-\al}} dt' dy \bigg)^{p} \bigg] \notag \\
 & \les   \E \bigg[ \bigg( \sum_{ s, \tau \geq 0  } 2^{(2-\al)s} 2^{\kk \tau} \int_{|x-y|\sim 2^{-s}}\int_{0}^{1} \ind_{\{ | |t-t'|-|x-y||\sim 2^{-\tau} \}} \Dr_{N}(t',y)dt' dy \bigg)^{p} \bigg] \notag \\
 & \les \bigg\{  \sum_{ s, \tau \geq 0  } 2^{(2-\al)s} 2^{\kk \tau} 
  \E \bigg[ \bigg(  \int_{|x-y|\sim 2^{-s}}\int_{0}^{1} \ind_{\{ | |t-t'|-|x-y||\sim 2^{-\tau} \}} \Dr_{N}(t',y)dt' dy \bigg)^{p} \bigg]^{\frac{1}{p}} \bigg\}^{p}.  \label{I1thetacomp0}
\end{align}
Now, by a change of variables, a rescaling, and a Riemann sum approximation, we have 
\begin{align}
  \E& \bigg[ \bigg(  \int_{|x-y|\sim 2^{-s}}\int_{0}^{1} \ind_{\{ | |t-t'|-|x-y||\sim 2^{-\tau} \}} \Dr_{N}(t',y)dt' dy \bigg)^{p} \bigg] \notag \\ 
  & \leq  2^{-2sp} \E \bigg[ \bigg(  \int_{\T^2}\int_{0}^{1} \ind_{\{ | |t-t'|-|2^{-s}y+x| |\sim 2^{-\tau} \}} \Dr_{N}(t',2^{-s}y+x)dt' dy \bigg)^{p} \bigg] \notag \\
  & = 2^{-2sp} \lim_{J\to \infty} \E \bigg[ \bigg(  \sum_{j,k,h=1}^{J} \frac{4\pi^2}{J^{3}} \ind_{\{ || t-t_{h}'|-|2^{-s}y_{j,k}+x| |\sim 2^{-\tau} \}} \Dr_{N}(t_{h}',2^{-s}y_{j,k}+x) \bigg)^{p} \bigg],
  \label{I1thetacomp1}
\end{align}
where $y_{j,k}$, $j,k=1,\ldots, J$ is defined as $y_{j,k}= ( -\pi +\frac{2\pi}{J}(j-1), -\pi +\frac{2\pi}{J}(k-1))\in \T^2 \equiv [-\pi,\pi)^{2}$ and $t'_{h}$, $h=1,\ldots, J$, is defined as $t'_{h}= \frac{1}{J}(h-1) \in [0,1]$.
Now, by \eqref{cov1}, 
\begin{align}
\begin{split}
\E[ &\Psi_{N}(t_{h_1}', 2^{-s}y_{j_1,k_1}+x)\Psi_{N}(t_{h_2}', 2^{-s}y_{j_2,k_2}+x)] \\
& \leq -\tfrac{1}{2\pi} \log\big( |t_{h_1}'-t_{h_2}'| + 2^{-s}|y_{j_1,k_1}-y_{j_2,k_2}|+N^{-1} \big)  +C \\
& \leq -\tfrac{1}{2\pi} \log( 2^{s}|t_{h_1}'-t_{h_2}'| +|y_{j_1,k_1}-y_{j_2,k_2}|+2^{s} N^{-1})+\tfrac{1}{2\pi}\log (2^{s}) +C \\
& \leq -\tfrac{1}{2\pi} \log( 2^{s}|t_{h_1}'-t_{h_2}'| +|y_{j_1,k_1}-y_{j_2,k_2}|+ N^{-1})+\tfrac{1}{2\pi}\log (2^{s})+C \\
& \leq -\tfrac{1}{2\pi} \log(|y_{j_1,k_1}-y_{j_2,k_2}|+  N^{-1})+\tfrac{1}{2\pi}\log (2^{s})+C \\
& = \E[ (\Psi_{N}(1, y_{j_1,k_1}) +g_s)(\Psi_{N}(1,y_{j_2,k_2})+g_{s})],
\end{split} \label{logs}
\end{align}
where $g_{s}$ is a collection of mean-zero Gaussian random variables, independent of $\Psi_{N}$, and with variance $\frac{1}{2\pi}\log(2^{s})+C$.

Now, with the convex function $x\mapsto x^{p}$, we apply Lemma~\ref{LEM:Kahane}, independence, and Lemma~\ref{LEM:GMCmeas}  to obtain
\begin{align*}
\eqref{I1thetacomp1} &\leq 
2^{-2sp} \lim_{J\to \infty} \E \bigg[ \bigg(  \sum_{j,k,h=1}^{J} \frac{4\pi^2}{J^{3}} \ind_{\{ | |t-t_{h}'|-|2^{-s}y_{j,k}+x| |\sim 2^{-\tau} \}} e^{\be( \Psi_{N}(1, y_{j_1,k_1}) +g_s) -\frac{\be^2}{2}(\s_{N}+\E[g_{s}^{2}])} \bigg)^{p} \bigg] \\
& = 2^{-2sp} \, \E \bigg[ \bigg( \int_{\T^2} \int_{0}^{1} \ind_{\{ | |t-t'|-|2^{-s}y+x| |\sim 2^{-\tau} \}} e^{\be( \Psi_{N}(1, y) +g_s) -\frac{\be^2}{2}(\s_{N}+\E[g_{s}^{2}])} dt' dy \bigg)^{p} \bigg]  \\
& = 2^{-2sp} \, \E \Big[ \big( e^{\be g_s -\frac{\be^{2}}{2} \E[g_{s}^2]}\big)^p\Big]
\E \bigg[ \bigg( \int_{\T^2} \int_{0}^{1} \ind_{\{ | |t-t'|-|2^{-s}y+x| |\sim 2^{-\tau} \}} e^{\be \Psi_{N}(1, y)  -\frac{\be^2}{2}\s_{N}} dt' dy \bigg)^{p} \bigg]  \\
& \les 2^{-2sp} \, 2^{s(p^2-p)\frac{\be^{2}}{4\pi}} \, 2^{-p\tau} \, \E \bigg[ \bigg( \int_{\T^2}  e^{\be\Psi_{N}(1, y) -\frac{\be^2}{2}\s_N} dy \bigg)^{p} \bigg]  \\
& = 2^{-2sp +s(p^2-p)\frac{\be^{2}}{4\pi}}\, 2^{-p \tau} \E\big[ \mathcal{M}_{N}(1,\T^{2})^{p}] \\
& \les 2^{-2sp +s(p^2-p)\frac{\be^{2}}{4\pi}}\, 2^{-p \tau}
\end{align*}
uniformly in $N\in \N$. Returning this bound to \eqref{I1thetacomp0}, we find 
\begin{align*}
\eqref{I1thetacomp0} \les  \bigg\{  \sum_{s,\tau \geq0 } 2^{(2-\al-2+(p-1)\frac{\be^2}{4\pi})s}  
\, 2^{-(1-\kk)\tau}    \bigg\}^{p} \les 1,
\end{align*}
uniformly in $N\in \N$, since $\kk<1$ and $\al>(p-1)\frac{\be^2}{4\pi}$.  This completes the proof of \eqref{I1thetagoal1} and thus also the proof of \eqref{I1theta}.

We now move onto the difference estimate \eqref{I1thetadiff}.
We first claim that
\begin{align}
\begin{split}
|\E&\big[  \big(  \Dr_{N_1}(t_1,y_1)-\Dr_{N_2}(t_1,y_1)  \big) \big(  \Dr_{N_1}(t_2,y_2)-\Dr_{N_2}(t_2,y_2) \big)\big]| \\
&  \les  (N_1 \wedge N_2)^{-\eps}(|y_1-y_2|+(N_1 \wedge N_2)^{-1})^{-\frac{\be^2}{2\pi}}  |y_1 -y_2|^{-2\eps}
\end{split} \label{covcov0}
\end{align}
for any $\eps>0$ and $(t_1,t_2,y_1,y_2)\in [0,1]^{2}\times (\T^{2})^{2}$.
By a direct computation, we have
\begin{align}
\begin{split}
\text{LHS} \text{ of } \eqref{covcov0} &  =e^{\be^{2} \G_{N_1}(t_1,t_2,y_1-y_2)} - e^{\be^{2} \G_{N_1,N_2}(t_1,t_2,y_1-y_2)} \\
& \hphantom{XXX}  +  e^{\be^{2} \G_{N_2}(t_1,t_2,y_1-y_2)}  -e^{\be^{2} \G_{N_2,N_1}(t_1,t_2,y_1-y_2)}.
\end{split} \label{covcov}
\end{align}
It suffices to consider the first term on RHS of \eqref{covcov}. By the mean value theorem, and \eqref{cov2}, we have
\begin{align}
&|e^{\be^{2} \G_{N_1}(t_1,t_2,y_1-y_2)} - e^{\be^{2} \G_{N_1,N_2}(t_1,t_2,y_1-y_2)}|  \notag \\
& = |\G_{N_1}(t_1,t_2,y_1-y_2) -  \G_{N_1,N_2}(t_1,t_2,y_1-y_2)| \notag \\
& \hphantom{XX} \times \int_{0}^{1} \be^2 \exp\big( \be^2 ( \tau \G_{N_1}(t_1,t_2, y_1-y_2) + (1-\tau)  \G_{N_1,N_2}(t_1,t_2, y_1-y_2))\big) d\tau \notag  \\
& \les (|t_1-t_2|+|y_1-y_2|+(N_1 \wedge N_2)^{-1})^{-\frac{\be^2}{2\pi}}|\G_{N_1}(t_1,t_2,y_1-y_2) -  \G_{N_1,N_2}(t_1,t_2,y_1-y_2)| \notag  \\
& \les (|y_1-y_2|+(N_1 \wedge N_2)^{-1})^{-\frac{\be^2}{2\pi}}|\G_{N_1}(t_1,t_2,y_1-y_2) -  \G_{N_1,N_2}(t_1,t_2,y_1-y_2)|. \label{covcov1}
\end{align}
Now combining \eqref{cov3}, the elementary inequality 
\begin{align*}
|\log(y+z) | \leq C_{\eps} |y+z|^{-\eps} \leq C_{\eps} |z|^{-\eps}
\end{align*}
for any $ 0<z,y \les 1$ and $\eps>0$, and an interpolation argument, we further obtain
\begin{align*}
\eqref{covcov1} \les (|y_1-y_2|+(N_1 \wedge N_2)^{-1})^{-\frac{\be^2}{2\pi}} (N_1 \wedge N_2)^{-\eps} |y_1 -y_2|^{-2\eps}
\end{align*} 
for any $\eps>0$, which establishes \eqref{covcov0}. 
We write 
\begin{align*}
\mathcal{P}_{\kk,\al}[\Dr_{N}] (t,x;\be) =  \mathcal{P}^{(1)}_{\kk,\al}[\Dr_{N}] (t,x;\be) + \mathcal{P}^{(2)}_{\kk,\al}[\Dr_{N}] (t,x;\be),
\end{align*}
where
\begin{align*}
 \mathcal{P}^{(2)}_{\kk,\al}[\Dr_{N}] (t,x;\be)&:=  \frac{1}{\jb{t}^3} \int_{\T^2}\int_{0}^{1}   \frac{1}{| |t-t'|-|x-y||^{\kk}}  \frac{\Dr_{N}(t',y;\be)}{|x-y|^{2-\al}} dt' dy, \\
 \mathcal{P}^{(1)}_{\kk,\al}[\Dr_{N}] (t,x;\be)& : =\mathcal{P}_{\kk,\al}[\Dr_{N}] (t,x;\be)-\mathcal{P}^{(2)}_{\kk,\al}[\Dr_{N}] (t,x;\be).
\end{align*}

Now, by Fubini's theorem, for any fixed $t\in \R$ and $x\in \T^2$, we have by \eqref{covcov0},
\begin{align}
\E&\big[  | \mathcal{P}_{\kk,\al}^{(2)}[\Dr_{N_1}](t,x) - \mathcal{P}_{\kk,\al}^{(2)}[\Dr_{N_2}](t,x) |^{2}\big] \notag \\
& \les \frac{1}{(N_1\wedge N_2)^{\eps}}\frac{1}{\jb{t}^{6}} \int_{(\T^2)^2} \int_{0}^{1} \int_{0}^{1} \prod_{j=1}^{2} \frac{1}{ | |t-t_j|-|x-y_j||^{\kk} |x-y_j|^{2-\al} } \frac{1}{|y_1-y_2|^{2\eps+\frac{\be^{2}}{2\pi}}} dt_1 dt_2 dy_1 dy_2  \notag \\
& \les \frac{1}{(N_1\wedge N_2)^{\eps}}\frac{1}{\jb{t}^{6}}  \int_{(\T^2)^2} \frac{1}{|x-y_1|^{2-\al} |x-y_2|^{2-\al} |y_1-y_2|^{2\eps+\frac{\be^{2}}{2\pi}}} dy_1 dy_2 \notag \\
& \sim\frac{1}{(N_1\wedge N_2)^{\eps}}\frac{1}{\jb{t}^{6}}\int_{(\T^2)^2} \frac{1}{|y_1|^{2-\al} |y_2|^{2-\al} |y_1-y_2|^{2\eps+\frac{\be^{2}}{2\pi}}} dy_1 dy_2  \notag \\
& \les  \frac{1}{(N_1\wedge N_2)^{\eps}}\frac{1}{\jb{t}^{6}} \label{covcov2}
\end{align}
provided that $\be^{2}<4\pi \min(1-\frac{\eps}{\pi}, \al -\eps) $, uniformly $x\in \T^2$. Note that this latter integral can be estimated by considering separately the cases $|y_1|\sim |y_2|\ges  |y_1-y_2|$ and $|y_1-y_2| \sim \max(|y_1|,|y_2|)$.
Similarly, we have
\begin{align}
\E&\big[  | \mathcal{P}_{\kk,\al}^{(1)}[\Dr_{N_1}](t,x) - \mathcal{P}_{\kk,\al}^{(1)}[\Dr_{N_2}](t,x) |^{2}\big] 
 \les  \frac{1}{(N_1\wedge N_2)^{\eps}}\frac{1}{\jb{t}^{6}} \label{covcov3}
\end{align}
Finally, \eqref{I1thetadiff} follows from Fubini's theorem, \eqref{covcov2} and \eqref{covcov3}.
\end{proof}

We point out that the fourth inequality in \eqref{logs}, which just uses the elementary fact that $x\mapsto -\log x$ is a decreasing function, played a key role. Indeed, it allowed us to avoid potential complications arising from correlations in time. 

%\begin{remark}\rm
%We will apply Proposition~\ref{PROP:I1theta} with $p=6+O(\eps)$ and $\al=\frac{3}{2}-\eps$ which enforces the restriction $\be^{2} < \frac{6 \pi}{5}$. See Proposition~\ref{PROP:LWP} below.
%\end{remark}

\begin{proof}[Proof of Proposition~\ref{PROP:I1theta}]
The uniform boundedness in part (i) is a direct consequence of \eqref{I1theta} in Lemma~\ref{LEM:I1theta}. For part (ii), we interpolate the difference estimate \eqref{I1thetadiff} in $L^{2}(\O; L^{2}(\R\times \T^2))$ with the uniform bound \eqref{I1theta} in $L^{p+\eps}(\O; L^{p+\eps}(\R\times \T^2))$ for some small $\eps>0$ so that the numerology in \eqref{alpha} continues to hold. This gives 
\begin{align*}
\E \Big[ \big\|\mathcal{P}_{\kk,\al}[\Dr_{N_1}]-\mathcal{P}_{\kk, \al}[\Dr_{N_2}] \big\|_{L^{p}_{t,x}(\R\times \T^2)}^{p} \Big] \leq C (N_1 \wedge N_2)^{-\ta_0}
% \label{Hbd0}
\end{align*}
for some $\ta_0>0$ and any $N_1,N_2\in \N$. 
\end{proof}

Finally, for the globalisation argument, we need estimates with respect to the time-weighted norm $\Ld_{p,\ld}^{s,b}(I)$, for which we use Proposition~\ref{prop:keryld}. The form of the kernel $\mc R_{\s,\ld}$ in \eqref{Rsld} motivates defining the following time-weighted version of \eqref{HGMC}:

\begin{align}
\begin{split}
\mathcal{P}_{\kk,\al,\ld}[\Dr_{N}] (t,x;\be) & : =   \int_{\T^2}\int_{0}^{1}  e^{-\frac{\ld}{200}|t-t'|}  \K_{\kk,\al}(t,t',x-y)  \Dr_{N}(t',y;\be)dt' dy, 
\end{split}
\label{HMCld}
\end{align}
for $\kk \geq0 $ and $0< \al < 2$ and $t\in \R$.

\begin{remark}\rm \label{RMK:gamma}
The positivity of $\Dr$ implies that $\kk \in [0,1) \mapsto \mathcal{P}_{\kk,\al,\ld}[\Dr](t,x)$ is increasing for every fixed $(t,x)$. This implies that control $\mathcal{P}_{\kk,\al,\ld}[\Dr]$ in $L^{p}_{t,x}$ implies control on $\mathcal{P}_{\kk',\al,\ld}[\Dr]$ in $L^{p}_{t,x}$ for any $0\leq \kk' \leq \kk$.
\end{remark}

The additional exponential time in \eqref{HMCld} can be exploited to gain an additional decay in $\ld$.

\begin{lemma}\label{LEM:I1thetaLd}
Let $N\in \N$, $0\leq \kk<1$, $1\leq p <\frac{8\pi}{\be^{2}}$ and $\al(p)$ satisfy \eqref{alpha}. Then, 
\begin{align}
\sup_{N\in \N} \E \bigg[   \bigg( \sum_{a\in \{c,s\}}\sum_{\ld \in 2^{\N_0}} \ld^{p(1-\kk-\eps_0)} \big\|\mathcal{P}_{\kk, \al;\ld}[\Dr^{a}_{N}] \big\|_{L^{p}_{t,x}(\R\times \T^2)}\bigg)^{p} \bigg] \leq C,\label{I1thetaLD}
\end{align}
for any $\ld>0$ and $0<\eps_0 <1-\kk$.
%Furthermore, for all $0<\be^{2}<4\pi$, it holds that 
%\begin{align}
%\E \Big[ \big\|\mathcal{P}_{\kk,\al}[\Dr_{N_1}]-\mathcal{P}_{\kk, \al}[\Dr_{N_2}] \big\|_{L^{2}_{t,x}([0,1]\times \T^2)}^{2} \Big] \leq C (N_1 \wedge N_2)^{-\ta}. \label{I1thetadiff}
%\end{align}
%for  some $\ta=\ta(\be,\al)>0$ and any $N_{1},N_{2}\in \N$. 
\end{lemma}
\begin{proof}
For \eqref{I1thetaLD}, we run the exact same argument as for proving \eqref{I1theta}, and end up needing to show that
\begin{align*}
2^{\kk \tau} \sup_{a\in \R}\int_{0}^{1\wedge t} \ind_{\{ | t-t'-a|\sim 2^{-\tau} \}}  e^{-\ld |t-t'|}dt' \les 2^{-\eps_0 \tau} \ld^{-(1-\kk-\eps_0)}.
\end{align*}
This follows readily from H\"{o}lder's inequality with $1=\frac{1}{q}+\frac{1}{q'}$ and $\frac{1}{q}= \kk+\eps_0 \in (0,1)$.
\end{proof}

\section{Well-posedness}\label{sec:wp}

\subsection{A deterministic local well-posedness result}

Here we study the system:
\begin{align}
\begin{split}
X(t)  & = \Phi_{1}(X,Y)  := - \iota \tfrac{1}{2} \be \I_{\textup{wave}} \big(  f^{+}(z+X+Y) \Dr^{+}-f^{-} (z+X+Y) \Dr^{-} \big)  \\
Y(t)  & = \Phi_{2}(X,Y)  := -\iota \tfrac{1}{2} \be \I_{\textup{KG} - \textup{wave}} \big(  f^{+}(z+X+Y) \Dr^{+}-f^{-}(z+X+Y) \Dr^{-} \big),
\end{split} \label{liouville6}
\end{align}
where $\mathcal{I}_{\text{wave}}$ and $\mathcal{I}_{\text{KG-wave}}$ are defined in \eqref{duhamels}, $f^{\pm}$ are defined in \eqref{fcfs}, $\{\Dr^{a}\}_{a\in \{\pm\}}$ are fixed positive distributions satisfying certain regularity and integrability conditions given below in \eqref{Drcond},
$z=\mathcal{D}(t)(v_0,v_1)\in C([0,T];\H^{s_3}(\T^2))$ for some $s_3>1$, and  both $X$ and $Y$ have zero initial data and belong to the space
\begin{align*}
\mc Z^{s_1,s_2}([0,T])  :=  \mathcal{X}^{s_1,s_2}(T)\times \mathcal{Y}^{s_2}(T)
\end{align*} where
\begin{align*}
 \mathcal{X}^{s_1,s_2}(T) & : = \Ld^{s_1,b}_{p}(T) \cap C([0,T];H^{s_2}(\T^2)) \cap C^{1}([0,T];H^{s_2-1}(\T^2)), 
  \\
 \mathcal{Y}^{s_2}(T) & : =  C([0,T];H^{s_2+1}(\T^2)) \cap C^{1}([0,T];H^{s_2-1}(\T^2)).  
\end{align*}
We abbreviate $\mc Z^{s_1,s_2}([0,T])$ as $\mc Z^{s_1,s_2}(T)$.

\begin{proposition}\label{PROP:LWP}
Let $0<\be^{2} < \tfrac{6 \pi}{5}$ and $\iota\in \R \setminus \{0\}$.
There exists $\eps_0>0$ such that if 
\begin{align}
p=6+6\eps, \quad s_1=\tfrac{1}{3}-\tfrac{1}{6}\eps, \quad b=\tfrac{1}{6}-\tfrac{1}{12}\eps, \quad \al=\tfrac{3}{2}-\eps
\label{params}
\end{align}
for any $0<\eps<\eps_0$, $0<s_2<1-\frac{\be^2}{4\pi}$, and $s_3>1$, the following holds true: 
Let $\{\Dr^{a}\}_{a\in \{\pm\}}$ be two positive distributions satisfying
\begin{align}
\max_{a\in \{\pm\}} \big(  \| \Dr^{a}\|_{L^{2}([0,1];H^{s_2-1}_{x}(\T^2))}  + \|\mathcal{P}_{ \frac{1}{2}+s_1+b+\frac{1}{8}\eps, \al}[\Dr^{a}]  \|_{L^{p}(\R\times \T^2)} \big) \leq R \label{Drcond}
\end{align}
for any $R\geq 1$, and $(v_0,v_1)\in \H^{s_3}(\T^2)$ with 
\begin{align*}
\| (v_0,v_1)\|_{\H^{s_3}(\T^2)} \leq K
\end{align*} 
for $K>0$, then for any $0 <r<1$, there exists $0<\dl<1$ and $c_0>0$ such that for
\begin{align*}
T = T(R,K)= (CR)^{-\frac{1}{\dl}}e^{-c_0 K}
\end{align*}
there is a unique solution $(X,Y)\in \mc Z^{s_1,s_2}(T)$ to \eqref{liouville6} with 
\begin{align*}
\| (X,Y)\|_{\mc Z^{s_1,s_2}(T)} \leq 1.
\end{align*}
Furthermore, given $\Ta^{+}, \Ta^{-}$ positive distributions satisfying \eqref{Drcond}, 
suppose that there exists sequences $\{\Dr^{a}_{N}\}_{N\in \N}$ for $a\in\{\pm\}$, satisfying the following properties: for every $N\in \N$, $\Dr_{N}$ is a smooth non-negative function and
there exists $\eta, A>0$, uniform in $N$, such that
\begin{align}
\max_{a\in \{\pm\}} \| \Dr^{a}_{N} - \Dr^{a} \|_{L_{t}^{2}([0,1];H^{s_2-1}_{x}(\T^2))}& \leq AN^{-\eta} \quad \text{for any} \quad N\geq 1, \label{drN1}\\
\max_{a\in \{\pm\}}\sup_{N\in \N \cup\{\infty\}} \big\| \mathcal{P}_{\frac 12+s_1+b+\frac{1}{8}\eps, \al}[\Dr^{a}_{N}]\big\|_{L_{t,x}^{p}(\R\times \T^2)}& \leq A,\label{drN2}
\end{align}
where we have abused notation by defining for $N=\infty$, $\Dr^{a}_{\infty}:=\Dr^{a}$, for $a\in \{\pm\}$. 
Then, denoting by $(X_{N},Y_{N})$ the solutions to \eqref{liouville6} with forcing terms $\{\Dr^{a}_{N}\}_{a\in \{\pm\}}$, and $(X,Y)$ the solutions to \eqref{liouville6} with forcing terms $\{\Dr^{a}\}_{a\in \{\pm\}}$, we have 
 \begin{align}
\lim_{N\to \infty} \big( \| X_{N}-X\|_{\mathcal{X}^{s_1,s_2}(T)} +\|Y_{N}-Y\|_{\mathcal{Y}^{s_2}(T)}  \big)=0. \label{convXY}
\end{align}
\end{proposition}

\begin{proof}
We note that for $s_1,b$ as in the statement of Proposition~\ref{PROP:LWP}, it holds that 
\begin{align*}
\tfrac 12 + s_1+ b + \tfrac{1}{8}\eps <1.
% \label{gamma1}
\end{align*}
Let $0<T<\frac 34$
and $B\subset \mathcal{Z}^{s_1,s_2}_{T}$ denote the ball of radius $1$ in $\mc Z^{s_1,s_2}_{T}$.   
Let $b'=b+\frac{1}{8}\eps$.
We set 
$K=\| (v_0,v_1)\|_{\H^{s_3}}$ and
\begin{align*}
R:= \max_{a\in \{\pm\}}\big( \| \mathcal{P}_{\frac{1}{2}+s_1+b', \al}[\Dr^{a}]\|_{L^{p}(\R\times \T^2)} + \| \Dr^{a}\|_{L^{2}([0,1]; H^{s_2-1}(\T^2)}\big).
\end{align*}

We first show the following mapping property:
\begin{align}
(X,Y) \in  B \mapsto \Phi(X,Y) : = (\Phi_{1}(X,Y), \Phi_{2}(X,Y)) \in B \label{ball2ball}
\end{align}
provided that $T$ is chosen sufficiently small, where $(\Phi_1,\Phi_2)$ are defined in \eqref{liouville6}.
Note that for $(X,Y)\in B$, Lemma~\ref{LEM:Sobolev}, the boundedness of $\mathcal{D}(t)$ on $\H^{s_3}$ for every $t\in \R_{+}$, and the Sobolev embedding imply $X,Y,z \in C([0,T]\times \T^2)$ and satisfy
\begin{align}
\| X\|_{L^{\infty}_{T,x}}+ \|Y\|_{L^{\infty}_{T,x}} + \|z\|_{L^{\infty}_{T,x}} \leq  C(1+ K). \label{Linftycontrol}
\end{align}

%Next, we give a general bilinear estimate that will use often. Combining \eqref{iso}, Proposition~\ref{prop:kery}, \eqref{Rsbd}, we have
%\begin{align}
%\| \Phi_{1}(X,Y)\|_{
%\end{align}

Note that for $0<T<\frac 34$, $\mathcal{I}_{\text{wave}}[F]=\mathcal{I}^{\star}_{\text{wave}}[\ind_{[0,T]}F]$ on $[0,T]$.
Thus, by \eqref{timegain2}, \eqref{iso}, Proposition~\ref{prop:kery}, \eqref{Rsbd}, \eqref{Linftycontrol}, and the positivity of $\Dr^{\pm}$, we have
\begin{align}
\| \Phi_{1}(X,Y)\|_{\Ld^{s_1,b}_{p}(T)}  & \les T^{\frac{\eps}{8}} \|\Phi_{1}(X,Y)\|_{\Ld^{s_1,b'}_{p}(T)} \notag \\
& \les T^{\frac{\eps}{8}} \Big( \| \mc I^{\star}_{\text{wave}} [ \ind_{[0,T]} f^{+}(z+X+Y)\Dr^{+}  ]\|_{\Ld^{s_1,b'}_{p}(\R\times \T^2)} \notag \\
& \qquad \quad  +\| \mc I^{\star}_{\text{wave}} [ \ind_{[0,T]} f^{-}(z+X+Y)\Dr^{-}  ]\|_{\Ld^{s_1,b'}_{p}(\R\times \T^2)}    \Big) \notag \\
 & \les T^{\frac{\eps}{8}} \big\{ \| f^{+} (z+X+Y)\|_{L^{\infty}_{T,x}} \| \mathcal{P}_{\frac 12+s_1+b', \al}[\Dr^{+}]\|_{L^{p}_{t,x}} \notag \\
& \hphantom{XXXX} + \| f^{-} (z+X+Y)\|_{L^{\infty}_{T,x}} \| \mathcal{P}_{\frac 12+s_1+b', \al}[\Dr^{-}]\|_{L^{p}_{t,x}}  \big\} \notag  \\
& \leq C T^{\frac{\eps}{8}}e^{C\be(1+K)}R. \label{bdX}
\end{align}
 Next, by Lemma~\ref{LEM:diffprop} (both (i) and (ii)), the Cauchy-Schwarz inequality, Lemma~\ref{LEM:pos}, and \eqref{Linftycontrol}, 
\begin{align}
\| \Phi_{2}(X,Y)\|_{\mathcal{Y}^{s_2}(T)} & \les \| f^{+}(z+X+Y) \Dr^{+}\|_{L^{1}_{T}H^{s_2-1}_{x}} +\| f^{-}(\be(z+X+Y) \Dr^{-}\|_{L^{1}_{T}H^{s_2-1}_{x}} \notag \\
& \les T^{\frac 12} \big\{  \| f^{+} (z+X+Y)\|_{L^{\infty}_{T,x}}\| \Dr^{+}\|_{L^{2}_{T}H^{s_2-1}_{x}} \notag   \\
&\hphantom{XXXX} +\| f^{-}(z+X+Y)\|_{L^{\infty}_{T,x}}\| \Dr^{-}\|_{L^{2}_{T}H^{s_2-1}_{x}} \big)\notag\\
& \leq C T^{\frac 12} e^{C\be (1+K)} R. \label{bdY}
\end{align}
Finally, arguing as in \eqref{bdY} we also have
\begin{align}
\begin{split}
\|\Phi_1(X,Y)\|_{C_{T}H^{s_2}_{x}\cap C^{1}_{T} H^{s_2-1}_{x}}& \les \sum_{a\in \{\pm\}} \| f^{a}(z+X+Y)\Dr^{a}\|_{L^1_{T}H^{s_2-1}_{x}}  \\
&\les T^{\frac 12} e^{\be\|z+X+Y\|_{L^{\infty}_{T,x}}}\max_{a\in \{\pm\}}\|\Dr^{a}\|_{L^{2}_{T}H^{s_2-1}_{x}}.
\end{split} \label{XC1part}
\end{align}
The mapping property \eqref{ball2ball} now follows from \eqref{bdX},~\eqref{bdY} and~\eqref{XC1part} and by choosing $T=T(R,K)>0$ sufficiently small.

Now we establish that the map $\Phi(X,Y)$ defined in \eqref{ball2ball} is a contraction by reducing $T$, if necessary. Let $(X_1,Y_1), (X_2,Y_2)\in B$. Then, by the triangle inequality,
\begin{align}
\|& \Phi(X_1,Y_1) -\Phi(X_2,Y_2)\|_{\mc Z^{s_1,s_2}(T}) \notag  \\
& \leq  \| \Phi_1(X_1,Y_1) -\Phi_1(X_2,Y_1)\|_{\mathcal{X}^{s_1,s_2}(T)}+ \| \Phi_1(X_2,Y_1) -\Phi_1(X_2,Y_2)\|_{\mathcal{X}^{s_1,s_2}(T)} \label{diff1} \\
& \hphantom{X}  + \| \Phi_2(X_1,Y_1) -\Phi_2(X_2,Y_1)\|_{\mathcal{Y}^{s_2}(T)}+ \| \Phi_2(X_2,Y_1) -\Phi_2 (X_2,Y_2)\|_{\mathcal{Y}^{s_2}(T)} . \label{diff2}
\end{align}
We now estimate \eqref{diff1} and \eqref{diff2}. 
By the fundamental theorem of calculus, we have 
\begin{align}
\begin{split}
f^{a}(v_1)-f^{a}( v_2)   =\be (v_1-v_2) \int_{0}^{1} (f^{a})'( \ta v_1 + (1-\ta) v_2) d\ta, \quad a\in \{\pm\}.
 \label{expdiff}
\end{split}
\end{align}
Thus, proceeding as in \eqref{bdX} and additionally using \eqref{expdiff}, we have
\begin{align}
 \| &\Phi_1(X_1,Y_1) -\Phi_1(X_2,Y_1)\|_{\Ld^{s_1,b}_{p}(T)} \notag \\
 & \les T^{\frac{\eps}{8}} \big(  \|f^{+}(z+X_1+Y_1)-f^{+}(z+X_2+Y_1)\|_{L^{\infty}_{T,x}} \| \mathcal{P}_{\frac{1}{2}+s_1+b', \al}[\Dr^{+}] \|_{L^{p}(\R\times \T^2)} \notag  \\
 & \hphantom{XXXX} + \|f^{-}(z+X_1+Y_1)-f^{-}(z+X_2+Y_1)\|_{L^{\infty}_{T,x}} \| \mathcal{P}_{\frac{1}{2}+s_1+b', \al}[\Dr^{-}] \|_{L^{p}(\R\times \T^2)} \big) \notag \\
 & \les T^{\frac{\eps}{8}} e^{C\be(1+K)} R\| X_1 -X_2\|_{L^{\infty}_{T,x}} \notag \\
 & \leq CT^{\frac{\eps}{8}} e^{C\be(1+K)} R \| X_1-X_2\|_{\Ld_{p}^{s_1,b}(T)}. \label{Phi1Xdiff}
\end{align}
Similarly, we obtain
\begin{align}
 \| &\Phi_1(X_2,Y_1) -\Phi_1(X_2,Y_2)\|_{\Ld^{s_1,b}_{p}(T)} \leq CT^{\frac{\eps}{8}} e^{C\be(1+K)} R \| Y_1-Y_2\|_{L^{\infty}_{T} H^{s_2+1}_{x}} \label{Phi1Ydiff}
\end{align}
For the terms \eqref{diff2}, we follow the arguments leading to \eqref{bdY} which then yield
\begin{align}
\eqref{diff2} \leq CT^{\frac 12} e^{C\be (1+K)}R\big\{ \| X_1-X_2\|_{\Ld^{s_1,b}_{p}(T)}+ \|Y_1-Y_2\|_{L^{\infty}_{T}H^{s_2+1}_{x}}\big\}. \label{Phi2diff}
\end{align}
Moreover, by following the computations for \eqref{XC1part}, we have 
\begin{align}
& \| \Phi_1(X_1,Y_1) -\Phi_1(X_2,Y_1)\|_{C_{T}H^{s_2}_{x}\cap C^{1}_{T} H^{s_2-1}_{x}}+ \| \Phi_1(X_2,Y_1) -\Phi_1(X_2,Y_2)\|_{C_{T}H^{s_2}_{x}\cap C^{1}_{T} H^{s_2-1}_{x}} \notag \\
& \les CT^{\frac 12} e^{C\be (1+K)}R\big\{ \| X_1-X_2\|_{\Ld^{s_1,b}_{p}(T)}+ \|Y_1-Y_2\|_{L^{\infty}_{T}H^{s_2+1}_{x}}\big\}. \label{Phi1C1}
\end{align}
Combining \eqref{Phi1Xdiff}, \eqref{Phi1Ydiff}, \eqref{Phi2diff}, \eqref{Phi1C1}, \eqref{diff1} and \eqref{diff2}, establishes that 
\begin{align*}
\|  \Phi(X_1,Y_1) -\Phi(X_2,Y_2)\|_{\mc Z^{s_1,s_2}(T)} \leq CT^{\ta}e^{C\be(1+K)}R\|(X_1,Y_1)-(X_2,Y_2)\|_{\mc Z^{s_1,s_2}(T)}
\end{align*}
for some $\ta>0$. Thus by reducing $T=T(K,R)>0$ if necessary, we obtain
\begin{align}
\|  \Phi(X_1,Y_1) -\Phi(X_2,Y_2)\|_{\mc Z^{s_1,s_2}_{T}} \leq \tfrac{1}{2}\|(X_1,Y_1)-(X_2,Y_2)\|_{\mc Z^{s_1,s_2}_{T}} \label{diffcontract}
\end{align}
which shows that $\Phi$ is a contraction on the ball $B\subset \mc Z^{s_1,s_2}(T)$. By the Banach fixed point theorem, we obtain a unique solution $(X,Y)\in \mc Z^{s_1,s_2}(T)$ to \eqref{liouville6} up to the random stopping time $T=T(K,R)>0$.  By standard arguments, the uniqueness of the solution $(X,Y)$ can be extended from the ball $B$ to the whole space $\mc Z^{s_1,s_2}(T)$. We omit the details.

It remains to establish the convergence property for the solution map $\Phi(X,Y; \Dr)$ in terms of  $\Dr$, namely~\eqref{convXY}.
We begin with some preliminary estimates, which hold for each fixed $t\in [0,T]$. Define $\wt{Y}:=z+Y$ and $\s_0 = \min(s_2,s_3)$.
Let $M\in \N$ and 
\begin{align}
0<\dl < \min( \tfrac{1}{100}\eps, \tfrac{\s_0}{2}). \label{dl}
\end{align}
$0<\dl<\s_0/2<1$. 
By the Bernstein inequality,  Lemma~\ref{LEM:Lip} and Lemma~\ref{LEM:prod} (i), we have
\begin{align}
\|& \P_{M} f^{a}( \wt{Y})\|_{L^{\infty}_{x}} \notag\\
&  \les  M^{-\dl} \| f^{a}( \wt{Y})\|_{H^{1+\dl}_{x}} \notag\\
& \les M^{-\dl}\big(  \| f^{a}( \wt{Y})\|_{L^2_x} + \|\nb( f^{a}( \wt{Y}))\|_{H^{\dl}_{x}}  \big)  \notag\\
& \les M^{-\dl}\big( \|f^a(\wt Y)\|_{L_x^2} + \| (\nb \wt{Y})  (f^{a})'( \wt{Y})\|_{H^{\dl}_{x}}\big)\notag \\
& \les M^{-\dl}\big( e^{\be \|\wt{Y}\|_{L^{\infty}_{x}}} + \|(f^{a})'( \wt{Y})\|_{L^{\frac{2(2+\dl)}{\dl}}_{x}} \|\nb \wt{Y}\|_{W^{\dl,2+\dl}_{x}} + \|(f^{a})'( \wt{Y})\|_{W^{\dl,{\frac{2(2+\dl)}{\dl}}}_{x}}  \| \nb \wt{Y}\|_{L^{2+\dl}_{x}}  \big) \notag \\
& \les M^{-\dl} e^{\be\|\wt{Y}\|_{L^{\infty}_{x}}}(1+\|\wt{Y}\|_{H^{1+\s_0}_{x}}), \label{YPN}
\end{align}
for any $a\in\{\pm\}$.
 $q>p$ be finite and $\dl>0$ be as in \eqref{dl}. Then, by the Bernstein inequality, Lemma~\ref{LEM:Lip}, and Sobolev embedding, we have 
\begin{align}
\| \P_{M}f^{a}(X)\|_{L^{\infty}_{x}}& \les M^{-\dl}\|f^{a}(X)\|_{W^{\frac{2}{q}+\dl,q}_{x}} \notag \\
& \les M^{-\dl} \big( e^{\be \|X\|_{L^{\infty}_{x}}} + \| |\nb|^{\frac{2}{q}+\dl}(f^{a}(X))\|_{L^{q}}\big) \notag \\
& \les M^{-\dl}e^{\be \|X\|_{L^{\infty}_{x}}} \big( 1 + \| X\|_{W^{\frac{2}{q}+\dl, q(1+\dl)}_{x}} \big)\notag
\\
&\les M^{-\dl}e^{\be \|X\|_{L^{\infty}_{x}}} \big( 1 + \| X\|_{W^{\frac{2}{q}+\dl+ \frac{2}{p}-\frac{2}{q(1+\dl)}, p}_{x}} \big) \notag \\
& \les M^{-\dl}e^{\be \|X\|_{L^{\infty}_{x}}} \big( 1 + \| X\|_{W^{s_1,p}_{x}}), 
\label{XPN}
\end{align}
for any $a\in \{\pm\}$, since, recalling \eqref{params} and \eqref{dl}, we have
\begin{align*}
\tfrac{2}{q}+\dl+ \tfrac{2}{p}-\tfrac{2}{q(1+\dl)} = \tfrac{2}{p} + \dl\big( \tfrac{1}{1+\dl}+\tfrac{2}{q}\big) \leq \tfrac{2}{p} + 3\dl  \leq s_1.
\end{align*}
 We now prove the convergence property \eqref{convXY}. Let $(X,Y)\in B\subset \mathcal{Z}_{T}^{s_1,s_2}$ be the solution with noise terms $\vec{\Dr}=(\Dr^{+},\Dr^{-})$ and $(X_{N},Y_{N})\in B$ be the solution with noise term $\vec{\Dr}_{N}=(\Dr_{N}^{+},\Dr_{N}^{-})$. 
Note that the assumption \eqref{drN1} implies that
\begin{align*}
\max_{a\in \{\pm\}}\sup_{N\in \N} \|\Dr^{a}_N\|_{L^{2}([0,1];H^{s_2-1})} \leq A+R =: A_0.
\end{align*}
By the triangle inequality and \eqref{diffcontract} (reducing $T$ if necessary), we have 
\begin{align*}
\|(X_N, Y_N) -(X,Y)\|_{\mc Z^{s_1,s_2}(T)} 
&=\| \Phi(X_N,Y_N,\vec{\Dr}_{N}) - \Phi(X,Y,\vec{\Dr})\|_{\mc Z^{s_1,s_2}(T)}  \\
&\leq   \| \Phi(X_N,Y_N,\vec{\Dr}_{N}) - \Phi(X,Y,\vec{\Dr}_{N})\|_{\mc Z^{s_1,s_2}(T)} \\
& \hphantom{XX}+  \| \Phi(X,Y,\vec{\Dr}_{N}) - \Phi(X,Y,\vec{\Dr})\|_{\mc Z^{s_1,s_2}(T)} \\
& \leq \tfrac{1}{2}\|(X_N, Y_N) -(X,Y)\|_{\mc Z^{s_1,s_2}(T)} \\
&\hphantom{XX}+ \| \Phi(X,Y,\vec{\Dr}_{N}) - \Phi(X,Y,\vec{\Dr})\|_{\mc Z^{s_1,s_2}(T)}.
\end{align*}
It then remains to control the second term on the right-hand side.
In the following, we prove: there exists $\ta_0 =\ta_0(s,b,p,s_1,s_2, \eta)>0$ such that
\begin{align} 
\begin{split}
\|\Phi_{1}(X,Y;\vec{\Dr}_N)-\Phi_{1}(X,Y;\vec{\Dr})\|_{\Ld^{s_1,b}_{p}(T)} & \leq C e^{C\be (1+K)}A_0 N^{-\ta_0}  \\ 
\|\Phi_{1}(X,Y;\vec{\Dr}_N)-\Phi_{1}(X,Y;\vec{\Dr})\|_{C_{T}H^{s_1}_{x}\cap C^{1}_{T} H^{s_2-1}_{x}} & \leq Ce^{C\be (1+K)}A_0 N^{-\ta_0} \\ 
\|\Phi_{2}(X,Y;\vec{\Dr}_N)-\Phi_{2}(X,Y;\vec{\Dr})\|_{C_{T}H^{s_2+1}_{x}} & \leq Ce^{C\be (1+K)}A_0 N^{-\ta_0} , \label{convergence}
\end{split}
\end{align}
for $N\in \N$ and $\Dr^{+}, \Dr^{-}$ positive distributions satisfying  \eqref{Drcond} and, where $\{\vec{\Dr}_{N}\}_{N\in \N}$ is a sequence of positive distributions associated to $\vec{\Dr}$ satisfying \eqref{drN1} and \eqref{drN2}. We will prove \eqref{convergence} by an interpolation argument. We will omit the arguments required for obtaining the second bound in \eqref{convergence} as these essentially follows from those same arguments used for the third bound.
First, by repeating the estimates in \eqref{bdX} using the positivity of $\vec{\Dr}_N$ and Lemma~\ref{LEM:Sobolev}, we have the uniform bound
\begin{align}
\begin{split}
\sup_{N\in \N}\|\Phi_{1}&(X,Y;\vec{\Dr}_N)-\Phi_{1}(X,Y;\vec{\Dr})\|_{\Ld^{s_1+\frac{1}{8}\eps,b}_{p}(T)} \\
& \leq \sup_{N \in \N}\|\Phi_{1}(X,Y;\vec{\Dr}_N)\|_{\Ld^{s_1+\frac{1}{8}\eps,b}_{p}(T)}  +\|\Phi_{1}(X,Y;\vec{\Dr})\|_{\Ld^{s_1+\frac{1}{8}\eps,b}_{p}(T)} \leq CA_0.
\end{split}
 \label{conv1}
\end{align}

Given $N_1,N_2 \in 2^{\N_0}$ dyadic, we define 
\begin{align*}
\Phi_{1; N_1,N_2}&(X,Y;\vec{\Dr}) \\
& : =  -\iota \tfrac{1}{2} \be \mathcal{I}_{\text{wave}} \Big[ \P_{N_1}(f^{+}(X)) \P_{N_2}(f^{+}(\wt{Y}))    \Dr^{+} -\P_{N_1}(f^{-}(X)) \P_{N_2}(f^{-}(\wt{Y}))\Dr^{-}  \Big] \\
\Phi_{2; N_1,N_2}&(X,Y;\vec{\Dr})&   \\
 &: =-\iota\tfrac{1}{2} \be \mathcal{I}_{\text{KG-wave}}  \Big[ \P_{N_1}(f^{+}(X)) \P_{N_2}(f^{+}(\wt{Y}))    \Dr^{+} -\P_{N_1}(f^{-}(X)) \P_{N_2}(f^{-}(\wt{Y}))\Dr^{-}  \Big] 
\end{align*}
so that we have the decomposition
\begin{align*}
\Phi_{j}(X,Y;\vec{\Dr}) = \sum_{ N_{1},N_{2}\in 2^{\N}} \Phi_{j; N_1,N_2}(X,Y;\vec{\Dr}) 
\end{align*}
for $j=1,2$. 

Let $\ta_1=\ta_1(p,\dl,\eta)>0$ be small enough so that 
\begin{align}
\ta_1 \max(  \tfrac 2p  + 1-\dl, 1-s_2+\dl) < \tfrac{\eta}{4}. \label{ta1}
\end{align}

\smallskip
\noi
$\bullet$ \textbf{Case 1:} $N_1 \vee N_2  \geq N^{\ta_1}$

\noi
By repeating the arguments we used to prove \eqref{bdX}, and using \eqref{YPN}, \eqref{XPN}, and Lemma~\ref{LEM:Sobolev}, we have 
\begin{align}
\|& \Phi_{1; N_1,N_2}(X,Y;\vec{\Dr}_{N})- \Phi_{1; N_1,N_2}(X,Y;\vec{\Dr})  \|_{\Ld^{s_1,b}_{p}(T)}  \notag \\
& \les\| \Phi_{1; N_1,N_2}(X,Y;\vec{\Dr}_{N})\|_{\Ld^{s_1,b}_{p}(T)} +\| \Phi_{1; N_1,N_2}(X,Y;\vec{\Dr})  \|_{\Ld^{s_1,b}_{p}(T)}   \notag \\
& \les \sum_{a\in \{\pm\}}  \| \P_{N_1}f^{a}(X) \|_{L^{\infty}_{T,x}} \| \P_{N_2}f^{a}(\wt{Y}) \|_{L^{\infty}_{T,x}} \sup_{N\in \N\cup\{\infty\}}  \|\mathcal{P}_{\frac 12 +s_1+b, \al}[\Dr^{a}_{N}]\|_{L^{p}(\R\times \T^2)} \notag \\
& \les (N_1 \vee N_2)^{-\dl} Ce^{C\be(1+K)} A_0 \notag \\
& \les (N_1 \vee N_2)^{-\frac{\dl}{2}} N^{-\frac{\dl \ta_1}{2}}Ce^{C\be(1+K)} A_0 . \label{conv12}
\end{align}
The first term allows us to perform the dyadic summations over $N_1,N_2\in 2^{\N}$, while the second term provides a convergence rate in $N$.  Similarly, using the argument for \eqref{bdY} and \eqref{YPN}, and \eqref{XPN}, we have 
\begin{align}
\| &\Phi_{2; N_1,N_2}(X,Y;\vec{\Dr}_{N})  - \Phi_{2; N_1,N_2}(X,Y;\vec{\Dr})  \|_{C_{T}H^{s_2+1}_{x}}  \notag \\
& \les\| \Phi_{2; N_1,N_2}(X,Y;\vec{\Dr}_{N})\|_{C_{T}H^{1+s_2}_{x}} +\| \Phi_{2; N_1,N_2}(X,Y;\vec{\Dr})  \|_{C_{T}H^{s_2+1}_{x}}  \notag  \\
& \les   \sum_{a\in \{\pm\}}  \| \P_{N_1}f^{a}(X) \|_{L^{\infty}_{T,x}} \| \P_{N_2}f^{a}(\wt{Y}) \|_{L^{\infty}_{T,x}} \sup_{N\in \N\cup\{\infty\}}  \|\Dr^{a}_N \|_{L^{1}([0,1];H^{s_2-1}_{x})}   \notag\\
& \les (N_1 \vee N_2)^{-\frac{\dl}{2}} N^{-\frac{\dl \ta_1}{2}}Ce^{C\be(1+K)} A_0. \label{ctyPhi21}
\end{align}

\smallskip
\noi
$\bullet$ \textbf{Case 2:} $N_1 \vee N_2  \leq N^{\ta_1}$

\noi
In this case, we need to make use of the assumption \eqref{drN1} and thus can no longer exploit positivity to handle the distributional term $\vec{\Dr}_N-\vec{\Dr}$. Nonetheless, the assumption in this case allows us to proceed crudely for $\Phi_{2}$. First, by Lemma~\ref{LEM:prod} (ii), Bernstein inequality, \eqref{drN1}, and \eqref{ta1} we have 
\begin{align}
\|&  \Phi_{2; N_1,N_2}(X,Y;\vec{\Dr}_{N})  - \Phi_{2; N_1,N_2}(X,Y;\vec{\Dr})  \|_{C_{T}H^{s_2+1}_{x}} \notag \\
&  = \| \Phi_{2; N_1,N_2}(X,Y;\vec{\Dr}_{N}-\vec{\Dr})  \|_{C_{T}H^{s_2+1}_{x}}  \notag\\
& \les \sum_{a\in \{\pm\}} \| \jb{\nb_x}^{1-s_2}\big( \P_{N_1}(f^{a}(X)) \P_{N_2}(f^{a}( \wt{Y})) \big) \|_{L^{\infty}_{T,x}} \| \Dr_{N}^{a} -\Dr^{a}\|_{L^{1}([0,1];H^{s_2-1}_{x})} \notag \\
& \les  (N_1\vee N_2)^{1-s_2} N^{-\eta} e^{C\be(1+K)} A_0  \notag \\
& \les N^{-\frac{\eta}{2}} (N_1 \vee N_2)^{-\dl} e^{C\be(1+K)}A_0.\label{ctyPhi22}
\end{align}
Combining \eqref{ctyPhi21} and \eqref{ctyPhi22} establishes the estimate for $\Phi_{2}$ in \eqref{convergence}. 
Unfortunately, such an argument does not suffice for $\Phi_{1}$ as we need to control the time derivatives $\jb{\dt}^{b}$ coming from the norm $\Ld_p^{s_1,b}$.
We establish the following bound: there exists $\eta_0>0$ sufficiently small such that for all $0<\eta\leq \eta_0$, we have
\begin{align}
\| \Phi_{1; N_1,N_2}(X,Y;\vec{\Dr}_{N})  - \Phi_{1; N_1,N_2}(X,Y;\vec{\Dr})  \|_{L^{p}_{T,x}} \leq C e^{C\be(1+K)}A_0 N^{-\frac{\eta}{2}} (N_1\vee N_2)^{-\dl}.  \label{conv2}
\end{align}
Interpolating \eqref{conv2} with \eqref{conv1}, combining this result with \eqref{conv12} and performing the dyadic summations over $N_1,N_2\in 2^{\N_0}$, we arrive at the first estimate in \eqref{convergence}.

We have 
\begin{align*}
\| &\Phi_{1; N_1,N_2}(X,Y;\vec{\Dr}_{N})  - \Phi_{1; N_1,N_2}(X,Y;\vec{\Dr})  \|_{L^{p}_{T,x}}  \\
& \leq  \sum_{N_3 \in 2^{\N_0}} \| \P_{N_3}\big[ \Phi_{1; N_1,N_2}(X,Y;\vec{\Dr}_{N})  - \Phi_{1; N_1,N_2}(X,Y;\vec{\Dr})  \big] \|_{L^{p}_{T,x}},
\end{align*}
and we split this summation into two cases. If $N_3 \geq N^{\ta_1}$, then by the Bernstein inequality, and repeating the arguments for \eqref{bdX}, we have
\begin{align*}
\|& \P_{N_3}\big[ \Phi_{1; N_1,N_2}(X,Y;\vec{\Dr}_{N})  - \Phi_{1; N_1,N_2}(X,Y;\vec{\Dr})  \big] \|_{L^{p}_{T,x}} \\
& \les N_{3}^{-s_1} \big\{ \| \Phi_{1; N_1,N_2}(X,Y;\vec{\Dr}_{N})\|_{\Ld^{s_1,b}_{p}(T)} +  \|  \Phi_{1; N_1,N_2}(X,Y;\vec{\Dr})\|_{\Ld^{s_1,b}_{p}(T)} \big\} \\
& \les N^{-\frac{\ta_1 s_1}{2}}  N_3^{-\frac{s_1}{2}} (N_1\vee N_2)^{-\dl}e^{C\be(1+K)}A_0.
\end{align*} 
We can then sum this bound over the dyadics $N_1,N_2,N_3 \in 2^{\N_0}$ and have the final rate $N^{-\frac{\ta_1 s_1}{2}}$. 
Now assume that $N_3 <N^{\ta_1}$. By Sobolev embeddings, Lemma~\ref{LEM:prod} (ii)
\begin{align*}
\| \P_{N_3}& \big[\Phi_{1; N_1,N_2}(X,Y;\vec{\Dr}_{N})  - \Phi_{1; N_1,N_2}(X,Y;\vec{\Dr}) \big]  \|_{L^p_{T,x}} \\
& \les  N_{3}^{\frac{2}{p}}  \| \Phi_{1; N_1,N_2}(X,Y;\vec{\Dr}_{N}-\Dr)  \|_{L^{\infty}_{T}L^{2}_{x}} \\
& \les N_{3}^{\frac{2}{p}} \sum_{a\in \{\pm\}} \| \jb{\nb_x}\big[ \P_{N_1}(f^{a}(X)) \P_{N_2}( f^{a}(\wt{Y})) \|_{L^{\infty}_{T,x}} \| \Dr^{a}_{N}-\Dr^{a_3}\|_{L^{1}_{T}H^{-1}_{x}} \\
& \les N_{3}^{\frac{2}{p}} (N_1\vee N_2)^{1-\dl} N^{-\eta} e^{C\be(1+K)}A_0\\
& \les (N_1 \vee N_2 \vee N_3)^{-\frac{1}{\ta_1}( \ta_1 (\frac 2p +1-\dl)-\frac{\eta}{4})}N^{-\frac{\eta}{4}} e^{C\be(1+K)}A_0,
\end{align*}
which is a negative power of $N_1 \vee N_2 \vee N_3$ in view of \eqref{ta1}.  This now completes the proof of \eqref{conv2} and thus the claimed convergence property of the map $\vec{\Dr} \mapsto \Phi(X,Y;\vec{\Dr})$.
\end{proof}

\begin{remark}\rm
In our approach, we crucially use the positivity of the distributions $\Dr^{+}$ and $\Dr^{-}$ and as such
we are not able to obtain continuous dependence with respect to these inputs. However, with the above interpolation argument, we only require a quantitative rate of convergence for $\Dr^{\pm}$ and not for $\mathcal{P}_{\frac 12+s_1+b+\frac{1}{8}\eps,\al}[\Dr_{N}^{\pm}]$; compare \eqref{drN1} and \eqref{drN2}.
\end{remark}

We now prove Theorem~\ref{thm:0}.

\begin{proof}The proof is a straightforward consequence of the deterministic local well-posedness result in Proposition~\ref{PROP:LWP}. See for instance~\cite[Proof of Theorems 1.2 and 1.6]{ORTz} for a similar argument.
\end{proof}

\subsection{Proof of Theorems~\ref{thm:main} and~\ref{thm:2}} \label{subsec:gwp}

In the following, we fix $\iota>0$.
%We adapt the arguments from \cite[Section 5]{ORTz}.
The Gaussian free field $\vec{\mu}_1=\mu_{1} \otimes \mu_0$ and hence the (truncated) Gibbs measure \eqref{gibbsN}, are independent of the law of the space-time white noise $\xi$. Thus, we may decompose the probability space as $\O=\O_1 \times \O_2$, where for $\o\in \O$ we write $\o=(\o_1,\o_2)$ and the Gaussian random Fourier series in \eqref{series} depends upon $\o_1$, while the Wiener process $W$ depends only upon $\o_2$. We have a corresponding decomposition of the underlying probability measure $\PP$ as $\PP=\PP_1 \otimes \PP_2$.

We then write 
\begin{align*}
\Psi(t; \vu_0, \o_2) = \dt \D(t) u_0 + \D(t) (u_0+u_1) +\sqrt{2} \int_{0}^{t} \D(t-t')  d\mathcal{B}(t',\o_2)
\end{align*}
for $\vu_0 \in \H^{-\eps}(\T^2)$ and $\o_2 \in \O_2$.  Given $N\in \N$, we then set 
\begin{align*}
\Psi_{N}(t;\vu_0, \o_2) = \Pii_{N}\Psi(t; \vu_0, \o_2) .
\end{align*}
By the fact that $\mathcal{D}(0)=0$ (see \eqref{prop1}), it can be shown that $\Psi_{N}(t;\vu_0, \o_2)$ is continuously differentiable in time\footnote{More precisely, $\dt \Psi_N \in C(\R; H^{-1-\eps})$ for any $\eps >0$.} and has time derivative 
\begin{align}
\begin{split}
\dt \Psi_{N}(t; \vu_0, \o_2) =& \dt^2 \D(t) \Pi_{N} u_0 + \dt \D(t) \Pi_{N}(u_0+ u_1) \\
&+\sqrt{2} \Pi_{N}\int_{0}^{t} (\dt\D)(t-t')  d\mathcal{B}(t',\o_2).
\end{split} \label{dtPsiN}
\end{align}
Next, we define the truncated Gaussian multiplicative chaos as in \eqref{GMC}:
\begin{align*}
\Dr_{N}^{\pm}(t; \be,\vu_0, \o_2) =  \g_{N}(\be)e^{\pm \be \Psi_{N}(t;\vu_0, \o_2)}.
\end{align*}
We set $\vec{\Dr}_{N} = (\Dr_{N}^{+}, \Dr_{N}^{-})$.
For $0\leq \kk<1$, we then define the sequence of weighted-in-time hyperbolic GMC object:
\begin{align*}
\mathcal{P}_{\kk, \al, \ld}[\Dr^{a}_{N}](t; \vu_0,\o_2) = \int_{\T^2}\int_{0}^{1}  e^{-\frac{\ld}{200} |t-t'|}   \K_{\kk,\al}(t,t',x-y) \Dr^{a}_{N}(t',y;\vu_0,\o_2) dt' dy 
\end{align*}
 where $a\in \{\pm\}$, $\ld\in 2^{\N_0}\cup \{0\}$, $\al=\frac 32 -\eps$ and $\K$ is as in \eqref{K}. Note that when $\ld=0$, 
\begin{align*}
\mathcal{P}_{\kk,\al, 0}[\Dr^{a}_{N}]= \mathcal{P}_{\kk, \al}[\Dr^{a}_{N}] 
\end{align*} 
where $\mathcal{P}_{\kk, \al}[\Dr^{a}_{N}] $ is the unweighted in time stochastic object from \eqref{HGMC} which was used in the general local well-posedness result (Proposition~\ref{PROP:LWP}). 
In the following, we will use the following single value of $\kk$: 
\begin{align}
\kk = \tfrac{1}{2}+s_1+b+\tfrac{\eps}{8} + \tfrac{\eps}{32} = \tfrac{1}{2}+\s+\tfrac{\eps}{8} +2\dl=1-\tfrac{\eps}{16},  \label{kkchoice}
%\label{g}
\end{align}
where we have chosen $\dl:=\frac{\eps}{64}$ for future application of \eqref{kerygoalld}.
%We will need the more general process $\wt{\mathcal{P}}_{\kk, \frac 32, 0}[\Dr^{a}_{N}]$ for the stability argument as well as increased regularity for $\Dr^{a}$ when controlling the $Y$-part.
With these definitions in hand, we define the (truncated) enhanced data set:
\begin{align}
\Xi_{N}(\vu_0; \o_2)  = \Big( \Psi_{N}, \vec{\Dr}_{N},\{ \mathcal{P}_{\kk,\al, \ld}[\Dr^{a}_N]\}_{\ld \in \{0\}\cup 2^{\N_0}, a\in \{\pm\}}\Big) \label{XiN}
\end{align}
We next define the (untruncated) enhanced data set 
\begin{align}
\Xi(\vu_0; \o_2)  = \Big( \Psi, \vec{\Dr}, \{ \mathcal{P}_{\kk,\al, \ld}[\Dr^{a}]\}_{\ld \in \{0\}\cup 2^{\N_0}, a\in \{\pm\}}\Big) \label{Xi}
\end{align}
where the components are well-defined as the $\rhoo \otimes \PP_{2}$ almost sure limits, guaranteed by Proposition~\ref{PROP:I1theta} and Lemma~\ref{LEM:I1thetaLd}, of the truncated enhanced data set \eqref{XiN}.
We also define 
the space
\begin{align*}
\mathcal{E}^{s_1,s_2}(T) = C_{T}H^{s_0}_x \times ( L^{2}_{T}H^{-1+s_2+\frac{1}{8}\eps}_{x} \times L^{2}_{T}H^{-1+s_2+\frac{1}{8}\eps}_{x})\times \l_{\ld}^{p,1-\kk-\frac{\eps}{32}}L^{p}_{t,x}
\end{align*}
where $\l^{p,s}( \{0\}\cup 2^{\N_0})$ denotes the sub-space of the sequence space $\l^p (\{0\}\cup 2^{\N_0})$ for which $a(\ld)\in\l^{p,s}( \{0\}\cup 2^{\N_0})$ if and only if  $\jb{\ld }^{s}a(\ld)\in \l^p (\{0\}\cup 2^{\N_0})$. We equip $\mathcal{E}^{s_1,s_2}(T) $ with the norm
\begin{align}
\begin{split}
\|\Xi(\vu_0; \o_2) \|_{\mathcal{E}^{s_1,s_2}(T)}   =& \| \Psi\|_{C_{T}H^{s_0}_x} + \max_{a\in \{\pm\}}\| \Dr^{a}\|_{L^{2}_{T}H^{-1+s_2+\frac{1}{8}\eps}_{x}} \\
&  + \max_{a\in \{\pm\}}\bigg(\sum_{\ld\in \{0\}\cup 2^{\N_0}}\jb{\ld}^{p (1-\kk-\frac{\eps}{32})}\| \mathcal{P}_{\kk, \al, \ld}[\Dr^{a}]\|_{L^{p}_{t,x}}^{p} \bigg)^{\frac 1p}.
\end{split} \label{Enorm}
\end{align}
Note that on any event for which $\|\Xi(\vu_0; \o_2) \|_{\mathcal{E}^{s_1,s_2}(T)} <\infty$, it follows from \eqref{kkchoice} and \eqref{Enorm} that 
\begin{align}
 \max_{a\in \{\pm\}}\| \mathcal{P}_{\kk,\al,\ld}[\Dr^{a}]\|_{L^{p}_{t,x}} \les  \jb{\ld}^{-\frac{\eps}{32}}, 
 \label{Pldgain}
\end{align}
with implicit constant uniform in $\ld \in \{0\}\cup 2^{\N_0}$.

We first establish good long time bounds for a finite-dimensional truncated flow and invariance of this truncated flow with respect to the corresponding truncated Gibbs measure. 
Given $N\in 2^{\N_0}$, we define the truncated dynamics: 
\begin{align}
\begin{cases}
\dt^2 u_{N} + \dt u_{N} + (1- \Dl)  u_{N}   + \iota \be \g_{N} \Pii_{ N}\sinh(\be \Pii_{ N} u) = \sqrt{2}\xi\\
(u_{N}, \dt u_{N}) |_{t = 0} = (u_0, u_1).
\end{cases}
\label{liouvilleNM}
\end{align}
Notice that we do not truncate the initial data.

\begin{proposition} \label{PROP:PhiNM}
For any $N\in \N$ and $s_0<0$, \eqref{liouvilleNM} defines a flow map
\begin{align}
\Phi^{N}(t) : \H^{s_0}(\T^2) \times \Omega \mapsto \H^{s_0}(\T^2) \label{PhiNM}
\end{align}
which is $\muu_1 \otimes \PP$-almost surely globally well-posed. Moreover, the truncated Gibbs measure $\rhoo_{N}$ in \eqref{gibbsN} 
is invariant under the flow \eqref{PhiNM} in the sense that 
\begin{align}
\E_{\rhoo_N \otimes \PP}\big[ F(\Phi^{N}(t)(\vu_0;\o))\big] = \E_{\rhoo_{N} }[ F(\vu_0)]  \label{invariance}
\end{align}
for any $F\in C_{b}(\H^{s_0}(\T^2))$ and $t\geq 0$.
\end{proposition}

\begin{proof}
We first discuss the global existence of the flow \eqref{PhiNM} associated to \eqref{liouvilleNM}.
 We write 
\begin{align*}
u_{N} = \Psi_{N}+ v_{N}
\end{align*}
with $v_{N}$ the solution to 
\begin{align}
\dt^2 v_{N}   + \dt v_{N}  +(1-\Dl)  v_{N}  = - \iota \be \g_{N}  \Pii_{ N} \sinh( \be \Psi_N +\be \Pii_{ N} v_{N}),
\label{liouvilleNM2}
\end{align} 
with $(v_{N},\dt v_{N})|_{t=0}=(v_0,v_1)\in \H^{1}$. Here we have included general initial data $(v_0,v_1)$ in order to perform an iteration later on. Note that 
\begin{align*}
(\text{Id}-\Pii_{2N})v_{N}(t) = (\text{Id}-\Pii_{2N})[ \dt \mathcal{D}(t)v_0 + \mathcal{D}(t)(v_0+v_1)],
\end{align*}
which is the solution to the linear damped Klein-Gordon equation with initial data $( (\text{Id}-\Pii_{2N}) v_0,( (\text{Id}-\Pii_{2N})v_1)$.

By a simple contraction mapping argument, we see that \eqref{liouvilleNM2} is unconditionally locally-well posed in $C_{T_{N}}\H^{1}$, where  $T_{N}(\vu_0, \vec{v}_0, \o)>0$ is some stopping time.
Therefore, the flow map
\begin{align*}
\Phi^{N}(t) : (\vec{u}_0, \o) \mapsto ( \Psi(\vu_0, \o) + v_{N}(t), \dt \Psi(\vu_0, \o) +  \dt v_{N}(t)) 
%\label{PhiNMmap}
\end{align*}
is well-defined almost surely on $[0,T_{N}]$ 
We now establish the global well-posedness for the map $\Phi^{N}(t)$.
To this end, we define the functional
\begin{align*}
E(t) = \tfrac{1}{2}\|v_{N}(t) \|_{H^{1}}^{2} + \tfrac{1}{2}\| \dt v_{N}(t)\|_{L^2}^{2}. 
%\label{Efunc}
\end{align*}
Using \eqref{liouvilleNM2}, we compute 
\begin{align*}
\frac{d}{dt} E(t) & = \int_{\T^2} \dt v_{N} ( \dt^2 v_{N}+(1-\Dl)v_{N}) dx \\
&   =- \| \dt v_{N}(t)\|_{L^{2}}^{2} -\iota \int_{\T^2} \dt v_{N} \be \Pii_{N}\big[  
\g_{N} \sinh(\be \Psi_{N}+\be \Pii_{N} v_{N}) \big]dx  \\
& \leq -\iota\int_{\T^2} \dt v_{N} \be  \Pii_{N}\big[ \g_{N} \sinh(\be \Psi_{N}+\be\Pii_{N}v_N) \big]dx.
\end{align*}
 Integrating in time then gives
\begin{align*}
E(t) 
& \leq E(0) -\iota \int_{0}^{t} \int_{\T^2} \partial_{t'}\Pii_{N} v_{N}(t') \cdot\be \g_{N} \sinh(\be \Psi_{N}(t')+\be\Pii_{N}v_{N}(t'))dx dt' \\
 & =E(0)-\iota \frac{1}{2} \int_{0}^{t} \int_{\T^2} \partial_{t'}( f^{+}(\Pii_{N}v_{N}(t'))) \Dr^{+}_{N}(t')+\partial_{t'}( f^{-}(\Pii_{N}v_{N}(t'))) \Dr^{-}_{N}(t') dx dt'.
\end{align*}
Recalling from \eqref{dtPsiN} that $\Psi_{N}$ is $C^1$ in time a.s. with $\Psi_{N}(0)=0$, integration by parts in time yields 
\begin{align}
\begin{split}
E(t)  & \leq E(0)  -\iota \int_{\T} \g_{N} \cosh(\be (\Pii_{N} v_N (t) +\Psi_{N}(t))dx  + \iota \g_{N}\int_{\T^2}\cosh(\be \Pi_{N}v_0)dx  \\
& \hphantom{X}  +\iota \be \int_{0}^{t} \int_{\T^2} \partial_{t'} \big(  \Pii_{N}\Psi_{N} \big) \g_{N}\sinh( \be \Psi_{N}(t') +\be \Pii_{N}v_{N}(t')) dx dt'.
\end{split} \label{ENMbd}
\end{align}
By setting 
\begin{align*}
Q(t) : = E(t) + \iota \int_{\T^2}\g_{N}\cosh(\be \Psi_{N}(t)+\be v_{N}(t)) dx
\end{align*}
and noting that from $\iota>0$ and the trivial inequality $\sinh(x) \leq \cosh(x)$ for any $x\in \R$, 
\begin{align}
\max\Big( E(t), \iota \int_{\T} \g_{N} \sinh( \be \Psi_{N}(t)+\be \Pii_{N}v_{N}(t))  dx \Big) \leq  Q(t) \label{EQ}
\end{align}
 for every $t>0$. Now from \eqref{ENMbd} we obtain
\begin{align*}
Q(t) \leq E(0)+ \iota \g_{N}\int_{\T^2} \cosh(\be \Pi_{N}v_0)dx+ \iota \be \int_{0}^{t}  \| \dt  \Pii_{N}\Psi_{N}\|_{L^{\infty}_{x}} Q(t')dt'.
\end{align*}
Gronwall's inequality and \eqref{EQ} then imply that 
\begin{align}
\sup_{t\in [0, T_{N})} E(t) \leq C(N,T_{N})<\infty.
\label{vNMbd2}
\end{align}
%We write $v_{N} = \Pii_{2N}v_{N}+(\text{Id}-\Pii_{2N})v_{N}$, and note that (i) $ \Pii_{2N}v_{N}$ is supported on frequencies $\{|n|\leq 4N\}$, and (ii) $(\text{Id}-\Pii_{2N})v_{N}=$, Bernstein's inequality then implies 
%\begin{align}
%\sup_{t\in [0, T_{N})}\| v_{N}(t)\|_{L^{\infty}_{x}}  \les N\sup_{t\in [0, T_{N})} E(t) <\infty \label{vNMbd2}
%\end{align}
%almost surely on the event $\{ T_{N}<\infty\}$. 
Thus, we can iterate the local well-posedness result for $v_{N}$ and establish that $v_N$, and hence $u_N$, exist globally in time. With $(v_0,v_1)=(0,0)$, we define 
\begin{align}
\begin{split}
\wt{X}_{N}(t)  &= - \iota\tfrac{1}{2} \be \I_{\textup{wave}} \big(   \Pii_{N} [ f^{+}(\Pii_{N}v_N) \Ta^{+}_{N}-f^{-}(\Pii_{N}v_N) \Ta^{-}_{N}  ]\big)  \\
\wt{Y}_{N}(t)  & = -\iota \tfrac{1}{2} \be \I_{\textup{KG} - \textup{wave}} \big(  \Pii_{N}[ f^{+}(\Pii_{N}v_N) \Ta^{+}_{N}-f^{-}(\Pii_{N}v_{N}) \Ta^{-}_{N}] \big),
\end{split} \label{liouvilleNM32}
\end{align}
which are globally well-defined and satisfy $\wt{X}_{N}(t) + \wt{Y}_{N}(t)=v_N (t)$, where $v_N$ is the unique global solution to \eqref{liouvilleNM2} with zero initial data.

Meanwhile, a simple contraction mapping argument establishes the unconditional local well-posedness for the system $(X_{N},Y_{N})$:
\begin{align}
\begin{split}
X_{N}(t)  &= - \iota\tfrac{1}{2} \be \I_{\textup{wave}} \big(   \Pii_{N} [ f^{+}(\Pii_{N} (X_N+Y_N) ) \Ta^{+}_{N}-f^{-}(\Pii_{N}(X_N+Y_N)) \Ta^{-}_{N}  ]\big)  \\
Y_{N}(t)  & = -\iota \tfrac{1}{2} \be \I_{\textup{KG} - \textup{wave}} \big(  \Pii_{N}[ f^{+}(\Pii_{N}(X_N+Y_N)) \Ta^{+}_{N}-f^{-}(\Pii_{N}(X_N+Y_N)) \Ta^{-}_{N}] \big),
\end{split} \label{liouvilleNM3}
\end{align}
in $C_{T'_{N}}\H^{1}\times C_{T'_{N}}\H^{1}$ for some almost surely positive stopping time $T'_{N}>0$. Then, we see that $X_N+Y_N$ solves \eqref{liouvilleNM32} with zero initial data and hence $X_N+Y_N=v_N$ on $[0,T_{N}']$. In particular, by comparing the formulas \eqref{liouvilleNM32}
and \eqref{liouvilleNM3}, we conclude that $X_{N}=\wt{X}_{N}$ and $Y_{N}=\wt{Y}_N$, and hence $(X_N,Y_N)$ exists globally-in-time.

We now discuss the invariance of $\rhoo_{N}$ under the truncated flow $\Phi^{N}$.
Let $\mathbf{S}_{N}$ be the sharp Fourier truncation to frequencies $\{ |n|\leq N\}$. We then decompose 
\begin{align*}
u_{N}=\mathbf{S}_{2N}u_{N}+(\text{Id}-\mathbf{S}_{2N})u_{N}=u^{(1)}+u^{(2)},
\end{align*}
and note that $u^{(1)}$ solves \eqref{liouvilleNM} with the noise $\xi$ replaced by $\mathbf{S}_{2N}\xi$ while $u^{(2)}$ solves the linear stochastic damped wave equation:
\begin{align*}
(\dt^{2} +\dt +1-\Dl)u^{(2)} = \sqrt{2} (\text{Id}-\mathbf{S}_{2N})\xi.
\end{align*}
As $ (\text{Id}-\mathbf{S}_{2N})_{\ast} \rhoo_{N}=  (\text{Id}-\mathbf{S}_{2N})_{\ast} \muu_{1}$, the high frequency part $u^{(2)}$ preserves $(\text{Id}-\mathbf{S}_{2N})_{\ast} \rhoo_{N}$. The low frequency part then preserves $(\mathbf{S}_{2N})_{\ast} \rhoo_{N}$ as it is formally given as the superposition of the deterministic nonlinear wave equation:
\begin{align*}
\dt^{2}u^{(1)} + (1-\Dl)u^{(1)} +\iota \be \Pii_{ N} \sinh(\be \Pii_{ N} u^{(1)})=0
\end{align*}
and the Ornstein-Uhlenbeck process for $\dt u^{(1)}$: 
\begin{align*}
\dt( \dt u^{(1)})+\dt u^{(1)} = \sqrt{2}\mathbf{S}_{2N}\xi.
\end{align*}
See for example \cite[Proposition 5.1]{ORTz} for similar details.
\end{proof}

\begin{lemma}[Control on enhanced data set] \label{LEM:enhanced}
Let $T>0$, $0<\be^{2}<\frac{6 \pi}{5}$ and $(s_1,s_2,b, \eps_0, p)$ be as in Proposition~\ref{PROP:LWP} and fix $0<\eps<\eps_0$. Then, we have 
\begin{align}
\sup_{N\in 2^{\N_0}} \int \E_{\mathbb{P}_{2}}\Big[ \| \Xi_{N}(\vu_0,\o_2)\|_{\mathcal{E}^{s_1,s_2}(T)}^p \Big] d\rhoo_{N}(\vu_0)& \leq C(T), \label{unifrhoN1}\\
\sup_{N\in 2^{\N_0}}  (\vec{\rho}\otimes \mathbb{P}_{2})\big(\max_{a\in \{\pm\}} N^{\eta}\| \Dr^{a}-\Dr^{a}_{N}\|_{L^{2}_{T}H^{-1+s_2+\frac{\eps}{8}}_{x}} >A \big) &\leq (C(T)A)^{-p}. \label{unifrhoN2}
\end{align}
In particular, 
$\vec{\Dr}_{N}(\vu_0,\o_2)$ converges to $\vec{\Dr}(\vu_0,\o_2)$ almost surely and in probability with respect to $\rhoo\otimes \mathbb{P}_{2}$.

\end{lemma}
\begin{proof}
In view of Lemma~\ref{lem:gibbs}, for any $F$ integrable with respect to $\rhoo\otimes \mathbb{P}_{2}$, we have
\begin{align*}
\int \E_{\mathbb{P}_{2}}[F(\vu_0 ,\o_2)] d\rhoo_{N}(\vu_0) \les \int \E_{\mathbb{P}_{2}}[F(\vu_0 ,\o_2)] d\muu_{1}(\vu_0)
\end{align*}
uniformly in $N\in \N$. Now \eqref{unifrhoN1} and \eqref{unifrhoN1} follow from Proposition~\ref{PROP:I1theta}. The convergence property of $\vec{\Dr}_{N}(\vu_0,\o_2)$ is then a consequence of the convergence of $\rhoo_{N}$ to $\rhoo$ in total variation. 
\end{proof}

The goal now is to prove uniform in $N$ bounds on the solutions $(X_N, Y_N)$ to the truncated equation \eqref{liouvilleNM3}. This is an improvement over the bounds obtained via \eqref{vNMbd2} which are not uniform in $N$. We rely on the invariance of $\rho_{N}$ under the flow $\Phi^{N}(t)$. In the case of the hyperbolic sinh-Gordon equation \eqref{liouville}, we need to establish control on $X_{N}\in L^{\infty}_{T,x}$ \emph{without} using an iteration argument. This is responsible for the worse restriction $\be^{2}<\frac{2\pi}{11}$ as compared to the restriction for local well-posedness of $\be^{2}<\frac{6\pi}{5}$ in Proposition~\ref{PROP:LWP}.

\begin{lemma}\label{LEM:PuN}
Let $T\geq 1$, $0<\be^{2}<\frac{2\pi}{11}$, and $\iota>0$. Let $u_{N}$ be the solution to the truncated system \eqref{liouvilleNM} on $[0,T]$ with respect to the truncated enhanced data $\Xi_{N}(\vu_0,\o)$, and let $u_{N} = \Psi_{N} +X_{N}+Y_{N}$, where $(X_N,Y_N)$ solves \eqref{liouvilleNM3} on $[0,T]$.
 Then, given any $\eta>0$, there exists $C_0=C_0 (T,\eta)\gg 1$ such that 
\begin{align}
\rhoo_{N}\otimes \PP \big( \|X_{N} \|_{L^{\infty}_{T,x}} >C_0\big) <\eta \label{unifPuN}
\end{align}
uniformly in $N\in 2^{\N_0}$.
 \end{lemma}

\begin{proof}
From \eqref{liouvilleNM3}, we see that $X_{N}$ satisfies 
\begin{align}
X_{N} = -\iota \be \mathcal{I}_{\text{wave}}[ \g_{N}\sinh(\be \Pii_{N}u_{N})] . \label{unifXN1}
\end{align}
Let $q\geq  6$ be an even integer to be chosen later,  set $s_q=\frac{2}{q}+\dl$ and $b_{q}=\frac{1}{q}+\dl$ for some sufficiently small $\dl>0$ so that $\kk_q:= \frac{1}{2}+s_q+b_q$ satisfies that $\kk_q<1$. Similarly to \eqref{bdX}, by \eqref{sobolev2}, \eqref{iso}, Proposition~\ref{prop:kery}, and \eqref{unifXN1},  we have
\begin{align*}
\|X_{N}\|_{L^{\infty}_{T,x}} & \les \|X_{N} \|_{\Ld^{s_q,b_q}_{q}(T)}   \les  \sum_{a\in  \{\pm\}} \| \mathcal{P}_{\kk_q, \al}[ \g_{N}e^{ a\be u_{N}}]\|_{L^{q}_{t,x}}.
\end{align*}
where we recall that $\al=\frac{3}{2}+\eps$.
Then, in order to prove \eqref{unifPuN}, it suffices to combine Chebyshev's inequality with control on the moments
\begin{align}\E[ \| \mathcal{P}_{\kk_q,\al}[ \g_{N}e^{a\be \Pii_{N}u_N}]\|_{L^{q}_{t,x}(\R\times \T^2)}^{q}]. \label{unifXN2}
\end{align} 
for $a\in \{\pm\}$.
 Without loss of generality, we fix $a=+$. Then, recalling \eqref{HGMC}, we have
\begin{align*}
\eqref{unifXN2}
 &= \int_{\R} \int_{\T^2} \frac{1}{\jb{t}^{3q}} \E \bigg[  \bigg(  \int_{0}^{1} \int_{\T^2}  \Big[ 1+\frac{1}{|t-t'|^{\kk_q}}+ \frac{1 }{  | |t-t'|-|x-y||^{\kk_q}}  \Big] \frac{\g_{N}e^{\be \Pii_{N}u_N(t',y)}}{|x-y|^{2-\al}} dydt'   \bigg)^{q}     \bigg] dx dt.
\end{align*}
We heavily rely on the invariance of the truncated measure $\vec{\rho}_{N}\otimes \PP$ with respect to the truncated flow $\Phi^N$ in \eqref{invariance}.  We first reduce to an essential part. 

By Jensen's inequality and invariance, we have 
\begin{align}
 \int_{\R} \int_{\T^2} \frac{1}{\jb{t}^{3q}} &\E \bigg[  \bigg(  \int_{0}^{1} \int_{\T^2}  \frac{\g_{N}e^{\be \Pii_{N}u_N(t',y)}}{|x-y|^{2-\al}} dydt'   \bigg)^{q}     \bigg] dx dt \notag \\
 & \les \sup_{x\in \T^2}  \int_{0}^{1} \E \bigg[  \bigg( \int_{\T^2}  \frac{\g_{N}e^{\be \Pii_{N}u_N(t',y)}}{|x-y|^{2-\al}} dy   \bigg)^{q}     \bigg]  dt' \notag \\
 &  \les  \sup_{x\in \T^2}  \int_{0}^{1} \E \bigg[  \bigg( \int_{\T^2}  \frac{\g_{N}e^{\be \Pii_{N}u_0(y)}}{|x-y|^{2-\al}} dy   \bigg)^{q}     \bigg]  dt'  <\infty \label{PuN1}
 \end{align}
where we noted that $\g_{N}e^{\be \Pii_{N}u_0(y)}= \Dr_{N}(0,y)$ to obtain the uniform in $N$ boundedness.

Now we consider the contribution of
\begin{align*}
\int_{\R} \int_{\T^2} \frac{1}{\jb{t}^{3q}} \E \bigg[  \bigg(  \int_{0}^{1} \int_{\T^2}  \Big[ \frac{1}{|t-t'|^{\kk_q}}+ \frac{1 }{  | |t-t'|-|x-y||^{\kk_q}}  \Big] \frac{\g_{N}e^{\be \Pii_{N}u_N(t',y)}}{|x-y|^{2-\al}} dydt'   \bigg)^{q}     \bigg] dx dt.
\end{align*}
If $|t|\gg 1$, then $|t-t'|\sim |t|$ and $||t-t'|-|x-y||\sim |t|$, so we have
\begin{align}
\int_{|t|\gg 1} \int_{\T^2}& \frac{1}{\jb{t}^{(3+\kk_q)q}} \E \bigg[  \bigg(  \int_{0}^{1} \int_{\T^2}   \frac{\g_{N}e^{\be \Pii_{N}u_N(t',y)}}{|x-y|^{2-\al}} dydt'   \bigg)^{q}     \bigg] dx dt
\end{align}
and we then argue as in \eqref{PuN1}.
 Thus it remains to consider the case when $|t|\les 1$. By a change of variables and that $|t|\les 1$, we have 
\begin{align*}
\sup_{\substack{ |t|\les 1 \\ x\in \T^2  }} \,  &\E \bigg[  \bigg(  \int_{0}^{t} \int_{\T^2}  \Big[ \frac{1}{|t-t'|^{\kk_q}}+ \frac{1 }{  | |t-t'|-|x-y||^{\kk_q}}  \Big] \frac{\g_{N}e^{\be \Pii_{N}u_N(t',y)}}{|x-y|^{2-\al}} dydt'   \bigg)^{q}     \bigg] \\
&=\sup_{\substack{ |t|\les 1 \\ x\in \T^2  }} \E \bigg[  \bigg(  \int_{0}^{t} \int_{\T^2}  \Big[ \frac{1}{|t'|^{\kk_q}}+ \frac{1 }{  | |t'|-|x-y||^{\kk_q}}  \Big] \frac{\g_{N}e^{\be \Pii_{N}u_N(t',y)}}{|x-y|^{2-\al}} dydt'   \bigg)^{q}     \bigg] \\
&\les\sup_{\substack{ |t|\les 1 \\ x\in \T^2  }}\E \bigg[  \bigg(  \int_{0}^{10} \int_{\T^2}  \Big[ \frac{1}{|t'|^{\kk_q}}+ \frac{1 }{  | |t'|-|x-y||^{\kk_q}}  \Big] \frac{\g_{N}e^{\be \Pii_{N}u_N(t-t',y)}}{|x-y|^{2-\al}} dydt'   \bigg)^{q}     \bigg].
\end{align*}
For the term with $|t'|^{-\kk_q}$, we use Jensen's inequality, noting that since $\kk_q<1$, $\frac{dt}{|t|^{\kk_q}}$ can be normalised to be a probability measure on $(0,10]$, and then the invariance property leading to the same term as in \eqref{PuN1}.

It remains to control the term with the ``hyperbolic singularity".  
Fix $x\in \T^2$. Then, by Minkowski's inequality and the translation invariance of the law of $\Dr_{N}(0,\cdot)$, we have
\begin{align*}
\E& \bigg[  \bigg( \int_{0}^{10}  \int_{\T^2}  \frac{1 }{  | |t'|-|x-y||^{\kk_q}} \frac{\g_{N}e^{\be \Pii_N u_N(t-t',y)}}{|x-y|^{2-\al}} dy dt'   \bigg)^{q}     \bigg] \\
&\les  \E \bigg[ \bigg( \sum_{\tau,\l=1}^{\infty} 2^{(2-\al)\l} 2^{\kk_q \tau}  \int_{0}^{10}  \int_{\T^2} \ind_{\{ |x-y|\sim 2^{-\l }\}} \ind_{\{ ||t'|-|x-y||\sim 2^{-\tau}  \}}\g_{N}e^{\be \Pii_N u_N(t-t',y)}dy \bigg)^q \bigg] \\
&\les \bigg(  \sum_{\tau,\l=1}^{\infty} 2^{(2-\al)\l} 2^{\kk_q \tau} \int_{0}^{10}  \E\bigg[ \bigg(  \int _{\T^2}\ind_{\{ |x-y|\sim 2^{-\l }\}} \ind_{\{ ||t'|-|x-y||\sim 2^{-\tau}  \}} \g_{N}e^{\be \Pii_N u_N(t-t',y)}dy  \bigg)^{q}\bigg]^{\frac{1}{q}}  dt' \bigg)^{q} \\
&\les \bigg(  \sum_{\tau,\l=1}^{\infty} 2^{(2-\al)\l} 2^{\kk_q \tau} \int_{0}^{10}  \E\bigg[ \bigg(  \int_{\T^2} \ind_{\{ |x-y|\sim 2^{-\l }\}} \ind_{\{ ||t'|-|x-y||\sim 2^{-\tau}  \}}\Dr_{N}(0,y)dy  \bigg)^{q}\bigg]^{\frac{1}{q}}  dt' \bigg)^{q} \\
&\sim \bigg(  \sum_{\tau,\l=1}^{\infty} 2^{(2-\al)\l} 2^{\kk_q \tau} \int_{0}^{10}  \E\bigg[ \bigg(  \int_{\T^2} \ind_{\{ |y|\sim 2^{-\l }\}} \ind_{\{ ||t'|-|y||\sim 2^{-\tau}  \}} \Dr_{N}(0,y)dy  \bigg)^{q}\bigg]^{\frac{1}{q}} dt' \bigg)^{q}.
\end{align*} 
Consider the inner integral over $y$.
If $|t'| \gg |y|$ or $|t'|\ll |y|$, then we must have that $|t'|\les 2^{-\tau}$. Thus, for these contributions, it follows from \eqref{GMCmeas2} that 
\begin{align*}
&\bigg(  \sum_{\tau,\l=1}^{\infty} 2^{(2-\al)\l} 2^{\kk_q \tau} \int_{|t'|\les 2^{-\tau}}  \E\bigg[ \bigg(  \int_{\T^2} \ind_{\{ |y|\sim 2^{-\l }\}} \ind_{\{ ||t'|-|y||\sim 2^{-\tau}  \}} \Dr_{N}(0,y)dy  \bigg)^{q}\bigg]^{\frac{1}{q}} dt' \bigg)^{q}  \\
& \les \bigg(  \sum_{\tau,\l=1}^{\infty} 2^{(2-\al)\l} 2^{\kk_q \tau} \int_{|t'|\les 2^{-\tau}}  \E\bigg[ \mathcal{M}_{N}(0,B(0,2^{-\l}))^{q} \bigg]^{\frac{1}{q}} dt' \bigg)^{q} \\
&\les  \bigg(  \sum_{\tau,\l=1}^{\infty} 2^{(2-\al)\l} 2^{(\kk_q-1) \tau} (2^{-\l})^{ (2+\frac{\be^2}{4\pi}) - \frac{\be^2}{4\pi}q} \bigg)^{q} <\infty,
\end{align*}
provided that $(q-1)\frac{\be^2}{4\pi}<\al$ which holds for $\be^{2}<\frac{6\pi}{5}$.
Now consider the contribution from $|t'|\sim |y|$. In this case, we have $|t'|\sim 2^{-\l}$. 
If $\tau \leq \l+5$, then we can proceed as in the previous case since the gain $2^{-\l}$ from the $t'$-integral can be transferred to gain a factor $2^{- \tau}$.

It remains to consider the case when $\l+ 5<\tau$. 
As $|t'|\sim |y| \gg 2^{-\tau}$, we can write $\{ 2^{-\tau-1}<| |t'|-|y||< 2^{-\tau+1}\}$ as $|t'| + 2^{-\tau-1} < |y| < |t'|+2^{-\tau+1}$ or $|t'| - 2^{-\tau+1} < |y| < |t'|-2^{-\tau-1}$, depending on if $|y|>|t'|$ or $|y|<|t'|$. In either case, we obtain a (spatial) annular region which has an area $\approx 2^{-\tau-\l}$.
Thus applying \eqref{GMCann}, we have
\begin{align*}
&\bigg(  \sum_{\substack{\tau, \l \ge 1 \\ \tau\geq \l+5}} 2^{(2-\al)\l} 2^{\kk_q \tau} \int_{|t'|\les 2^{-\l}}  \E\bigg[ \bigg(  \int_{\T^2} \ind_{\{ ||t'|-|y||\sim 2^{-\tau}  \}} \Dr_{N}(0,y)dy  \bigg)^{q}\bigg]^{\frac{1}{q}} dt' \bigg)^{q}  \\ 
& \les \bigg(  \sum_{\substack{\tau, \l \ge 1 \\ \tau\geq \l+5}} 2^{-\al\l} 2^{\kk_q \tau}  2^{-(1-\frac{\be^{2}}{8\pi} (q-1))\tau} 2^{ \frac{\be^2}{8\pi}(q-1)\l}     \bigg)^{q} 
 <\infty,
\end{align*}
provided that  $(q-1)\frac{\be^2}{8\pi} < \min(1-\kk_q,1, \al)=1-\kk_q$, since $\kk_q=\frac{1}{2}+\frac{3}{q}+2\dl$.
Thus, we need 
\begin{align*}
\be^{2} <\frac{4\pi (q-6)}{q(q-1)}=:h(q),
\end{align*}
where $\dl>0$ was chosen sufficiently small.
The function $h(q)$ has a maximum around $q\approx 11.48$ and the even integer which gives the largest range of $\be^2$ is $q=12$ for which we get 
\begin{align*}
\be^{2} < \frac{2\pi}{11} \approx 0.182\pi.
\end{align*}
 Putting the cases together then establishes \eqref{unifPuN}. 
\end{proof}

%\justin{ $\frac{2}{11}= 0.181818...$ while if we evaluate $h(q)$ at its maximum (for $q>0$) we get  $q_{\ast}=6+\sqrt{30}\approx 11.477$ and $h(q_{\ast})= 44-8\sqrt{30}\approx  0.1822$, which is not that different from just using $q=12$.    BTW if we put $q=10$, we get $h(q)=\frac{8}{45} \approx  0.17777$, which is worse than $q=12$.  Also no point asking for $q=11$ as $h(11)=\frac{2}{11}=h(12)$! }

Using Lemma~\ref{LEM:PuN}, we can now establish the uniform in $N$ bounds on $(X_N,Y_N)$ in the norm $\mathcal{Z}^{s_1,s_2}(T)$ for long times $T\geq 1$.

\begin{lemma}\label{LEM:unifNMbds}
Let $T\geq 1$, $0<\be^{2}<\frac{2\pi}{11}$, $s_1,s_2$ be as in Proposition~\ref{PROP:LWP} and $\iota>0$. Then, given any $\eta>0$, there exists $C_0=C_0 (T,\eta)\gg 1$ such that 
\begin{align}
\rhoo_{N}\otimes \PP \big( \| (X_{N},Y_{N})\|_{\mc Z^{s_1,s_2}(T)} >C_0\big) <\eta \label{unifNMXY}
\end{align}
uniformly in $N\in 2^{\N_0}$, where $(X_{N},Y_{N})$ is the solution to the truncated system \eqref{liouvilleNM3} on $[0,T]$ with respect to the truncated enhanced data $\Xi_{N}(\vu_0,\o)$. 
\end{lemma}

\begin{proof}
Let $(u_{N},\dt u_N) = \Phi^{N}(t)(\vec{u}_0 ,\o_2)$ be a global solutions constructed in Proposition~\ref{PROP:PhiNM}, which in particular satisfies the invariance property \eqref{invariance}. 
We write $v_{N}:= u_{N}-\Psi_{N}$ which solves \eqref{liouvilleNM2} with zero initial data,
and decompose $v_{N}= X_{N}+Y_{N}$, where $(X_{N},Y_{N})$ solve \eqref{liouvilleNM3}.
We rewrite  the equation for $Y_N$ as follows:
\begin{align*}
Y_{N}(t) = -\iota \be \mathcal{I}_{\text{KG-wave}}[ \Pii_{N} \g_{N}\sinh( \be \Pii_{N} u_{N})] 
\end{align*}
Then, as in \eqref{bdY}, by Minkowski's inequality and the invariance in \eqref{invariance}, we have
\begin{align*}
\big\| \|Y_{N}\|_{L^{\infty}_{T}H^{s_2+1}_x} \big\|_{L^{p}_{\vec{u}_0, \o}(\vec{\rho}_{N}\otimes \PP)}   & \les  T^{\frac 12} \big\|  \| \g_{N}\sinh(\be \Pii_{N} u_{N}(t))\|_{L^{2}_{T}H^{s_2-1}_x}   \big\|_{L^{p}_{\vec{u}_0, \o}(\vec{\rho}_{N}\otimes \PP)}  \\
& \les T^{\frac 12}  \Big\| \|  \jb{\nb_x}^{-1+s_2} \g_{N}\sinh(\be \Pii_{N} u_{N}) (t,x) \big\|_{L^{p}_{\vec{u}_0, \o}(\vec{\rho}_{N}\otimes \PP)  }   \Big\|_{L^2_{T,x}} \\
& \les T^{\frac 12}  \Big\| \| \jb{\nb_x}^{-1+s_2} \g_{N}\sinh(\be \Pii_{N}u_0)(x) \big\|_{L^{p}_{\vec{u}_0, \o}(\vec{\rho}_{N}\otimes \PP)  }   \Big\|_{L^2_{T,x}} \\
&\les C(T,p) <\infty,
\end{align*}
uniformly in $N\in 2^{\N_0}$. A similar computation applies for $\| \dt Y_{N}\|_{L^{\infty}_{T}H^{s_2}_x}$.
Therefore, it follows by Chebyshev's inequality, that 
\begin{align}
\rhoo_{N}\otimes \PP\big(  \|Y_{N}\|_{L^{\infty}_{T}H^{s_2+1}_x}+\|\dt Y_{N}\|_{L^{\infty}_{T}H^{s_2}_x} > C_{1}  \big) < \frac{\eta}{4}. \label{YNprob}
\end{align}
where $C_{1}:=(8 C(T,p))^{\frac 1p} (\eta)^{-\frac 1p}$.  
Similarly, by using the formula
\begin{align}
X_{N}(t) = -\iota \be \mathcal{I}_{\text{wave}}[ \Pii_{N} \g_{N}\sinh( \be \Pii_{N} u_{N})], 
\label{XNeq0}
\end{align}
and arguing as in the first inequality in \eqref{XC1part}, and by the invariance \eqref{invariance}, there exists $C_2=C_2(T,\eta)>0$, uniform in $N\in 2^{\N_0}$, such that 
\begin{align}
\rhoo_{N}\otimes \PP\big(\|X_{N}\|_{L^{\infty}_{T}H^{s_2}_x}+ \|\dt X_{N}\|_{L^{\infty}_{T}H^{s_2-1}_x}  > C_{2}  \big) < \frac{\eta}{4}. \label{XNprob0}
\end{align}
It remains to show that $X_{N}\in \Ld^{s_1,b}_{p}(T)$ with high-probability.
By Proposition~\ref{PROP:I1theta} and Chebyshev's inequality, there exists $c(T)>0$ such that
\begin{align}
\rhoo_{N}\otimes \PP \big( \| \Xi_{N}(\vu_0,\o)\|_{\mathcal{E}^{s_1,s_2}([0,T])} >A\big) \leq \frac{c(T)}{A^{p}} <\frac{\eta}{4}. \label{unifNM1}
\end{align}
where the second inequality follows by choosing $A=A(\eta,T)$ so that $A^{p}=16c(T)(\eta)^{-1}$, and where $c(T)$ and hence $A$ are uniform in $N\in 2^{\N_0}$.

Using Lemma~\ref{LEM:PuN}, there exists $C_3=C_3(T,\eta)>0$, uniform in $N\in 2^{\N_0}$, such that 
\begin{align}
\rhoo_{N}\otimes \PP\big(\|X_{N}\|_{L^{\infty}_{T,x}}  > C_{3}  \big) < \frac{\eta}{4}. \label{XNprob}
\end{align}
We now work on the event that 
\begin{align}
\| X_{N}\|_{L^{\infty}_{T,x}} \leq C_3, \quad \|Y_N\|_{L^{\infty}_{T,x}} \leq C_1, \quad \| \Xi_{N}(\vu_0,\o)\|_{\mathcal{E}^{s_1,s_2}([0,T])} \leq A. 
\label{XNprob2}
\end{align}
Then, by \eqref{unifXN1} and arguing similar to \eqref{bdX} with \eqref{XNprob2}, we have
\begin{align*}
\|X_{N}\|_{\Ld^{s_1,b}_{p}(T)}  \leq C e^{\be( \|X_{N}\|_{L^{\infty}_{T,x}} + \|Y_{N}\|_{L^{\infty}_{T,x}})}    \| \mathcal{P}_{\kk,\al}[ \g_{N}\cosh(\be \Psi_{N} )]\|_{L^{p}_{t,x}} &\leq C  e^{\be(C_3+C_1)} A. 
\end{align*}
Thus, on the event \eqref{XNprob2}, we control $X_{N}$ in $\Ld^{s_1,b}_{p}(T)$ uniformly in $N$.
Finally, by combining \eqref{YNprob},~\eqref{XNprob0}, \eqref{XNprob}, and \eqref{XNprob2}, and choosing $C_0=\frac{1}{3}\max(C_1,C_2,C_3, Ce^{\be(C_2+C_3)}A)$, we conclude the proof of \eqref{unifNMXY}.
\end{proof} 
\begin{remark}[On the defocusing Liouville model]\rm \label{rmk:liouvilleUS}
For the defocusing Liouville model, the analogue of the result of Lemma~\ref{LEM:unifNMbds} holds in the larger region $\be^2 <\frac{6\pi}{5}$, matching our local well-posedness result in Proposition~\ref{PROP:LWP}. 
Let us explain the modifications and differences. 
We consider the smoothed equation \eqref{liouvilleUS} and the decomposition $u_N = \Psi_{N}+X_N+Y_N$ where $(X_N,Y_N)$ satisfy \eqref{XYorw}. The main difference with the sinh-case \eqref{liouville} is that we no longer need Lemma~\ref{LEM:PuN} which was responsible for the further restriction $\be^2<\frac{2\pi}{11}$.

By the same arguments using invariance, we obtain both \eqref{YNprob} and \eqref{XNprob}. It remains to bound $X_{N}$ in $\Ld^{s_1,b}_{p}(T)$. 
In particular, on the event that
\begin{align}
 \|Y_{N}\|_{L^{\infty}_{T}H^{s_2+1}_x}  \leq C_{1}, \label{Yprob1}
\end{align}
 Sobolev embedding implies
\begin{align}
 \|Y_{N}\|_{L^{\infty}_{T,x}}  \leq C_{\text{sob}}C_{1} \label{Yprob2}
\end{align}
for some constant $C_{\text{sob}}>0$, where $C_{\text{sob}}C_{1}$ is uniform in $N\in 2^{\N_0}$. 

We now work on the intersection of the events for which \eqref{Yprob1} holds and for which we have 
\begin{align}
\| \Xi_{N}(\vu_0,\o)\|_{\mathcal{E}^{s_1,s_2}([0,T])} \leq A; \label{databd}
\end{align}
see \eqref{unifNM1}.
Due to the positivity of $\Q_{N}e^{\be \Q_{N} u_{N}}$ and the positivity of the kernel of $\mathcal{I}_{\text{wave}}$ in \eqref{S}, we see that $\be X_{N} \leq 0$ a.s. Thus, for $F$ a smooth bounded function such that $F(x)= e^{x}$ for $x\leq 0$ and $F\vert_{\R_{+}}\in C^{\infty}(\R_+,\R_+)$, we may write
\begin{align*}
X_{N}& = -\iota \be \Q_{N} \mathcal{I}_{\text{wave}}[e^{\be\Q_{N}(X_N +Y_N)} \Dr_{N} ] = -\iota \be \Q_{N} \mathcal{I}_{\text{wave}}[ F(\be \Q_{N}X_N)  e^{\be\Q_{N}Y_N}\Dr_{N} \big].
\end{align*}
Then, by arguing similar to \eqref{bdX} and using \eqref{Yprob2} and \eqref{databd}, we have 
\begin{align*}
\|X_{N}\|_{\Ld^{s_1,b}_{p}(T)}& \les \| F(\be \Q_{N}X_N)e^{\be \Q_{N}Y_N}\|_{L^{\infty}_{T,x}} \|\mathcal{P}_{\kk,\al}[\Dr_{N}]\|_{L^{p}_{t,x}} \\
&\les e^{\be \| Y_N \|_{L^{\infty}_{T,x}}}A \\
&\les e^{\be C_{\text{sob}}C_1}A,
\end{align*}
which is uniform in $N\in 2^{\N_0}$. We thus obtain \eqref{unifNMXY} without having to use Lemma~\ref{LEM:PuN}.
\end{remark}

For the stability statements, we need to define a slightly enlarged enhanced data set which takes into account the exponentially weighted in time object $\mathcal{P}_{\kk,\al;\ld}[\Dr]$.

\begin{lemma}[Stability] \label{LEM:stab}
Let $T\geq 1$, $K\geq 0$ and $C_0 \gg 1$. 

\noi
\textup{(a)}
There exist $N_0 (T,K,C_0)\in 2^{\N}$ such that the following statements hold:
suppose that for some $N\geq N_0$, we have
\begin{align}
\| \Xi_{N}(\vu_0, \o_2)\|_{\mathcal{E}^{s_1,s_2}(T)} \leq K \label{Stab1}
\end{align}
and
\begin{align}
\| (X_{N},Y_{N})\|_{\mc Z^{s_1,s_2}(T)} \leq C_0, \label{Stab2}
\end{align}
which is the solution to $(X_{N},Y_{N})$ to the truncated system \eqref{liouvilleNM3} with respect to the enhanced data set $\Xi_{N}(\vu_0;\o_2)$.
Furthermore, suppose there exists $\eta>0$ and a constant $A>0$ such that
\begin{align}
\max_{a\in \{\pm\}}\| \Dr^{a} -\Dr^{a}_{N}\|_{L^2([0,T]; H^{s_2-1}_{x})} \leq AN^{-\eta}  \label{Stab3}
\end{align}
Then, there exists a solution $(X,Y)$ to the truncated system on $[0,T]$ with zero initial data and the enhanced data set $\Xi(\vu_0;\o_2)$, satisfying 
\begin{align*}
\| (X,Y)\|_{\mc Z^{s_1,s_2}(T)} \leq C_0+1.
\end{align*}

\noi
\textup{(b)} Suppose that 
\begin{align*}
\| \Xi(\vu_0, \o_2)\|_{\mathcal{E}^{s_1,s_2}(T)} \leq K 
%\label{Stab4}
\end{align*}
and that the untruncated system \eqref{liouville6} with zero initial data and enhanced data set $\Xi(\vu_0; \o_2)$ has a solution $(X,Y)$ on $[0,T]$ satisfying 
\begin{align*}
\| (X,Y)\|_{\mc Z^{s_1,s_2}(T)} \leq C_0.
\end{align*}
Assume also that \eqref{Stab3} holds true. Then, there exists $N_0=N_0(C_0,K,T)\gg 1$ such that for all $N\geq N_0$, there is a solution $(X_{N},Y_{N})$ to the truncated system \eqref{liouvilleNM3} on $[0,T]$
with enhanced data set $\Xi_{N}(\vu_0;\o_2)$ and such that
\begin{align}
\| (X_{N},Y_{N})-(X,Y)\|_{\mc Z^{s_1,s_2}(T)} \leq C(K,C_0,T) N^{-\ta} \label{Stab5}
\end{align}
for some $\ta>0$ and $C(K,C_0,T)>0$. 
\end{lemma}

\begin{proof}
We adapt the proof of \cite[Proposition 6.3.1]{OOTol2}, beginning with part (a). Fix $T\geq 1$. For $\ld\geq 1$ to be chosen later, we define the space $\mc Z_{\ld}^{s_1,s_2}(T)$ by 
\begin{align*}
\| (X,Y)\|_{\mc Z^{s_1,s_2}_{\ld}(T)}= \| (e^{-\ld t}X,e^{-\ld t}Y)\|_{\mc Z^{s_1,s_2}(T)}.
\end{align*}
To simplify the notation, we write $Z_{N}=(X_{N},Y_{N})$, $Z=(X,Y)$, $\vec{\Dr}_{N}=\vec{\Dr}_{N}(\vu_0,\o_2)$, and $\vec{\Dr}=\vec{\Dr}(\vu_0,\o_2)$. We assume the bounds \eqref{Stab1}, \eqref{Stab2} and \eqref{Stab3}. In particular, from \eqref{Stab1}, we have
\begin{align*}
\max_{a\in \{\pm\}}\| \Dr^{a}(\vu_0,\o_2)\|_{L^{2}_{T}H^{-1+s_2+\frac{\eps}{8}}} \leq K + AN^{-\eta} \leq K+1=:K_0.
\end{align*} 
Thus, from \eqref{Stab1} again, we have
\begin{align*}
\| \Xi(\vu_0,\o_2)\|_{\mathcal{E}^{s_1,s_2}(T)} \leq K_0.
\end{align*}
We write the Duhamel formulation of the system \eqref{liouvilleNM3}
 with zero data 
 \begin{align}
 \begin{split}
X_{N}(t)  & = \G_{1}^{N}(X_{N},Y_{N}, \vec{\Dr}_{N}) \\
&  :=  -\iota \tfrac{1}{2} \be \I_{\textup{wave}} \big(  \Pii_{N} [f^{+}(\Pii_{N}(X_N+Y_N)) \Dr^{+}_{N} -f^{-}(\Pii_{N}(X_N+Y_N)) \Dr^{-}_{N} ]\big)  \\
Y_{N}(t)  & =\G_{2}^{N}(X_{N},Y_{N}, \vec{\Dr}_{N}) \\
& := -\iota \tfrac{1}{2} \be \I_{\textup{KG} - \textup{wave}} \big(   \Pii_{N}[f^{+}(\Pii_{N}(X_N+Y_N)) \Dr^{+}_{N} -f^{-}(\Pii_{N}(X_N+Y_N)) \Dr^{-}_{N}] \big),
\end{split} \label{Gsyst}
\end{align}
We want to construct a solution $Z=(X,Y)$ to the system \eqref{Gsyst} when $N=+\infty$:
 \begin{align}
 \begin{split}
X(t)  & = \G_{1}^{\infty}(X,Y, \vec{\Dr})   :=  -\iota\tfrac{1}{2} \be \I_{\textup{wave}} \big(   [f^{+}(X+Y) \Dr^{+} -f^{-}(X+Y) \Dr^{-} ]\big)  \\
Y(t)  & =\G_{2}^{\infty}(X,Y, \vec{\Dr}) := -\iota \tfrac{1}{2}\be \I_{\textup{KG} - \textup{wave}} \big(   f^{+}(X+Y) \Dr^{+} -f^{-}(X+Y) \Dr^{-}] \big),
\end{split} \label{Gsyst2}
\end{align}
This amounts to constructing the difference $Z-Z_{N}=(X-X_{N}, Y-Y_{N})$. We write $\dl Z_{N}=Z-Z_N$, $\dl X_{N}=X-X_N$ and $\dl Y_{N} = Y-Y_{N}$, with 
\begin{align*}
Z=\dl Z_{N} + Z_N, \quad X=\dl X_{N} + X_{N}\quad \text{and} \quad Y = \dl Y_{N} +Y_{N}.
\end{align*}
We thus consider the following system in the variables $(\dl X_{N}, \dl Y_{N})$: 
\begin{align}
\begin{split}
\dl X_{N} = \wt{\G}_{1}(\dl X_{N} , \dl Y_{N})=\G_{1}^{\infty}(\dl X_{N} + X_{N},\dl Y_{N} + Y_{N}, \vec{\Dr}) - \G_{1}^{N}(X_{N},Y_{N}, \vec{\Dr}_{N}), \\
\dl Y_{N} =\wt{\G}_{2}(\dl X_{N} , \dl Y_{N})=\G_{2}^{\infty}(\dl X_{N} + X_{N},\dl Y_{N} + Y_{N}, \vec{\Dr}) - \G_{2}^{N}(X_{N},Y_{N}, \vec{\Dr}_{N}).  
\end{split} 
\label{deltasyst}
\end{align}
Our goal is to prove that the map $\wt{\G}=(\wt{\G}_{1}, \wt{\G}_{2})$ is a contraction on a small ball in $\mc Z^{s_1,s_2}_{\ld}(T)$ by choosing $\ld=\ld(K_0,C_0, T)\gg 1$.
As in \cite{OOTol2}, it is convenient first to establish bounds on $\wt{\G}$ in $\mc Z^{s_1,s_2}_{\ld}(T)$ for $\dl Z_{N}$ in a closed ball $B_{1}$ of radius $1$ in the $\mc Z^{s_1,s_2}(T)$-norm. 
Under the assumption that $\dl Z_{N} \in B_1$, \eqref{Stab2} implies 
\begin{align*}
\| Z\|_{\mc Z^{s_1,s_2}(T)} \leq \| \dl Z_{N} \|_{\mc Z^{s_1,s_2}(T)} + \| Z_{N}\|_{\mc Z^{s_1,s_2}(T)} \leq 1+ C_0.
\end{align*}

We essentially repeat the computations from the proof of Proposition~\ref{PROP:LWP} but taking into account the presence of the time weight $e^{-\ld |t|}$ and the difference structure in \eqref{deltasyst}. 
We begin by establishing bounds for the first equation in \eqref{deltasyst}. As the arguments for the $C_{T}H^{s_2}_{x}\cap C^{1}_{T}H^{s_2-1}_{x}$ parallel the arguments for establishing the control on $\wt{\G}_{2}$, we omit them. Thus, we only prove the bounds on $\wt{\G}_1$ in the space $\Ld^{s_1,b}_{p,\ld}(T)$.
 Notice that the operator $\wt{\G}_{1}(\dl X_{N} , \dl Y_{N})$ can be decomposed into three pieces: 
\begin{align}
\wt{\G}_{1}(\dl X_{N} , \dl Y_{N}) = \sum_{j=1}^{3}\wt{\G}_{1,j}(\dl X_{N} , \dl Y_{N}), \label{G1}
\end{align}
where: the $j=1$ terms correspond to the terms where there appears the difference $\text{Id}-\Pii_{N}$, the $j=2$ terms correspond to a difference in the noise terms $\vec{\Dr}_{N}$ and $\vec{\Dr}$, and the $j=3$ terms correspond to the remaining terms where there is at least one factor of $\dl X_{N}$ or $\dl Y_{N}$ arising from the difference in \eqref{deltasyst} (and not from $Z =\dl Z_{N} + Z_{N}$).

We begin with the $j=1$ case. The main point is that by losing a small amount of spatial regularity we can use the difference $(\text{Id}-\Pii_{N})$ to gain a negative power of $N$.
We consider the representative term:
\begin{align*}
\I_{\textup{wave}} \big( f^{+}(Z_{N}) \Dr^{+}_{N})- \I_{\textup{wave}} \big(  \Pii_{N} [f^{+}(\Pii_{N}(Z_{N})) \Dr^{+}_{N}]).
\end{align*}
We write this as
\begin{align*}
\I_{\textup{wave}} \big( [\text{Id}-\Pii_{N}]  f^{+}(Z_{N}) \Dr^{+}_{N}) + \I_{\textup{wave}} \big(  \Pii_{N}[(f^{+}(Z_{N})-f^{+}(\Pii_{N}Z_{N})) \Dr^{+}_{N}]) =: I_{1} + I_{2}.
\end{align*}
By Bernstein's inequality,  \eqref{Ldweight}, \eqref{ldnorm}, \eqref{iso}, Proposition~\ref{prop:keryld}, the positivity of $\Dr_{N}^{+}$, \eqref{Rsbd}, \eqref{HMCld}, Remark~\ref{RMK:gamma}, \eqref{Pldgain} and \eqref{Stab1}, we get 
\begin{align*}
\| I_{1}\|_{\Ld^{s_1,b}_{p,\ld}(T)}&  \les N^{-\frac{\eps}{32}} \| I_{1}\|_{\Ld^{s_1+\frac{\eps}{32},b}_{p,\ld}(T)}   \\
& \les N^{-\frac{\eps}{32}} \| \mathcal{I}^{\star}_{\text{wave}}[ \ind_{[0,T]} f^{+}(Z_N)\Dr^{+}_{N}] \|_{\Ld^{s_1+\frac{\eps}{32}, b}_{p,\ld}} \\
& \les N^{-\frac{\eps}{32}} \| \mathcal{I}^{\star}_{\text{wave},\ld}[ \ind_{[0,T]} e^{-\ld t} f^{+}(Z_N)\Dr^{+}_{N}] \|_{\Ld^{s_1+\frac{\eps}{32}, b}_{p}} \\
& \les  N^{-\frac{\eps}{32}}     \big(  \| \mathcal{A}_{0,\ld}( \Dr^{+}_{N})\|_{L^{p}_{t,x}}+\| \mathcal{A}_{s_1+\frac{\eps}{32} +b,\ld}( \Dr^{+}_{N})\|_{L^{p}_{t,x}}  \big)   \|e^{-\ld t}f^{+}(Z_{N})\|_{L^{\infty}_{T,x}}   \\
& \les N^{-\frac{\eps}{32}}   \| \mathcal{P}_{s_1+\frac{\eps}{32}+b, \al, \ld}[\Dr_{N}^{+}]\|_{L^{p}_{t,x}}  \|e^{-\ld t}f^{+}(Z_{N})\|_{L^{\infty}_{T,x}} \\
& \les N^{-\frac{\eps}{32}}   \| \mathcal{P}_{\kk, \al, \ld}[\Dr_{N}^{+}]\|_{L^{p}_{t,x}}  \|e^{-\ld t}f^{+}(Z_{N})\|_{L^{\infty}_{T,x}}  \\
& \leq CC_0 N^{-\frac{\eps}{32}} \ld^{-\frac{\eps}{32}}K_0.
\end{align*}
Note that we simply bounded the decaying weight attached to the term $f^{+}(Z_{N})$ above by $1$ and used the assumption that $Z_{N}\in B_1$. We will do the same for the other $j=1$ terms and the $j=2$ terms.
For the term $I_{2}$, we use \eqref{ldnorm}, \eqref{iso}, Proposition~\ref{prop:keryld}, \eqref{expdiff}, Remark~\ref{RMK:gamma}, \eqref{Pldgain}, and Lemma~\ref{LEM:Sobolev}, to get
\begin{align}
\begin{split}
&\| I_{2}\|_{\Ld^{s_1,b}_{p,\ld}(T)}  \\
&\les \| \mathcal{P}_{s_1+b, \al, \ld}[\Dr_{N}^{+}]\|_{L^{p}_{t,x}}  \|f^{+}(Z_{N})-f^{+}(\Pii_{N}Z_{N})\|_{L^{\infty}_{T,x}}  \\
& \les C_0 K_{0}\ld^{-\frac{\eps}{32}} \| (\text{Id}-\Pii_{N})Z_{N}\|_{L^{\infty}_{T,x}} \\
& \les C_0 K_{0}\ld^{-\frac{\eps}{32}} \big( \| (\text{Id}-\Pii_{N})X_{N}\|_{\Ld^{s_1 -\frac{\eps}{8}, b}_{p}(T)}
+ \|  (\text{Id}-\Pii_{N})Y_{N} \|_{L^{\infty}_{T}H^{1+s_{2}-\frac{\eps}{8}}_{x}}\big) \\
& \les C_0 K_{0}\ld^{-\frac{\eps}{32}}N^{-\frac{\eps}{8}}.
\end{split} \label{lambdagain}
\end{align}
These arguments suffice to control all other terms and we find:
\begin{align}
\| \wt{\G}_{1,1}(\dl X_{N}, \dl Y_{N})\|_{\Ld^{s_1,b}_{p,\ld}(T)} \leq C C_0 K_0 N^{-\frac{\eps}{32}} \ld^{-\frac{\eps}{32}}. \label{G11}
\end{align}

We move onto the $j=2$ terms. As we have a difference in the enhanced data set, we gain a negative power of $N$ from the assumption \eqref{Stab3}. We repeat the argument in proving the first inequality in \eqref{convergence} with any reference to using Proposition~\ref{prop:kery} replaced by using Proposition~\ref{prop:keryld} instead. This procedure yields
\begin{align}
\| \wt{\G}_{1,2}(\dl X_{N}, \dl Y_{N})\|_{\Ld^{s_1,b}_{p,\ld}(T)} \leq C C_0 K_0 N^{-\ta}, \label{G12}
\end{align}
for some $\ta>0$, depending on $\eta$.

Now we consider the $j=3$ terms, where there is at least one of the differences $\dl X_{N}$ or $\dl Y_{N}$ appearing, not coming from $Z=\dl Z_{N}+Z_{N}$. In this case, we need to keep the decaying weight $e^{-\ld t}$ appearing on the right hand side of the function \eqref{ldnorm}. This weight is associated to the difference term after using \eqref{expdiff}, and as in \eqref{lambdagain}, we can recover a small negative power of $\ld$ (but no negative power of $N$). These considerations lead to the bound:
\begin{align}
\| \wt{\G}_{1,3}(\dl X_{N}, \dl Y_{N})\|_{\Ld^{s_1,b}_{p,\ld}(T)} \leq C C_0 K_0  \ld^{-\frac{\eps}{32}} \| \dl Z_{N}\|_{\mc Z^{s_1,s_2}_{\ld}(T)}. \label{G13}
\end{align}
We move onto considering the second equation in \eqref{deltasyst}. In this case, by Minkwoski's inequality,  Lemma~\ref{LEM:diffprop}, and H\"{o}lder's inequality,
\begin{align}
\| e^{-\ld t} \I_{\textup{KG} - \textup{wave}} [ F]\|_{C_{T}H^{1+s_{2}}_{x}} & \leq \bigg\| \int_{0}^{t}  e^{-\ld(t-t')} \| e^{-\ld t'}F(t')\|_{H^{-1+s_{2}}_{x}} dt' \bigg\|_{L^{\infty}_{T}} \notag \\
& \les \ld^{-\frac 12} \| e^{-\ld t} F\|_{L^{2}_{T}H^{-1+s_2}_{x}}. \label{Hld}
\end{align}
We decompose $\wt{\G}_{2}$ into the same three types of terms as in \eqref{G1}, and use \eqref{Hld} to control the norms of each of these. We then obtain
\begin{align}
\|e^{-\ld t} \wt{\G}_{2}(\dl X_{N},\dl Y_{N}) \|_{C_{T}H^{1+s_{2}}_{x}} \leq C(C_0,K_0) \ld^{-\frac 12}(N^{-\ta}+ \| \dl Z_{N}\|_{\mc Z^{s_{1},s_{2}}_{\ld}(T)}). \label{G2bd}
\end{align}
Combining \eqref{G1}, \eqref{G11}, \eqref{G12}, \eqref{G13}, and \eqref{G2bd}, we have shown that 
\begin{align}
\| \wt{\G}(\dl Z_{N})\|_{\mc Z^{s_1,s_2}_{\ld}(T)} \leq C(C_0,K_0,T) \big[ N^{-\ta_1} +\ld^{-\ta_2} \|\dl Z_{N}\|_{\mc Z^{s_1,s_2}_{\ld}(T)} \big], \label{FP1}
\end{align}
for some $\ta_j=\ta_j(\eta,\eps)>0$, $j=1,2$ and any $\dl Z_{N}\in B_1 \subset Z^{s_1,s_2}(T)$. 

For the contraction estimate, we take any $\dl Z^{(1)}_{N}, \dl Z^{(2)}_{N}\in B_1 \subset \mc Z^{s_1,s_2}(T)$, and estimate the terms  
\begin{align*}
\G_{j}^{\infty}(\dl X_{N}^{(1)} + X_{N},\dl Y_{N}^{(1)} + Y_{N}, \vec{\Dr}) -\G_{j}^{\infty}(\dl X_{N}^{(2)} + X_{N},\dl Y_{N}^{(2)} + Y_{N}, \vec{\Dr}), \,\, j\in \{1,2\}
\end{align*}
Then arguing as in the estimate leading to \eqref{G13}, we have
\begin{align}
\| \wt{\G}(\dl Z_{N}^{(1)})-\wt{\G}(\dl Z_{N}^{(2)})\|_{\mc Z^{s_1,s_2}_{\ld}(T)} \leq C(C_0,K_0,T) \ld^{-\ta_2} \|\dl Z^{(1)}_{N}-\dl Z^{(2)}_{N}\|_{\mc Z^{s_1,s_2}_{\ld}(T)}, \label{FP2}.
\end{align}
for any $\dl Z^{(1)}_{N}, \dl Z^{(2)}_{N}\in B_1 \subset \mc Z^{s_1,s_2}(T)$.

Given $r>0$ to be chosen later, let $B_{r,\ld}=\{ \| (X,Y)\|_{\mc Z^{s_1,s_2}_{\ld}(T)}\leq r\}$. Then, it follows from Lemma~\ref{LEM:ldup}, that if $\dl Z_{N}\in B_{r,\dl}$, then by choosing $r=r(\ld, T)$ sufficiently small, we have $\dl Z_{N}\in B_{1}$. 
 Now, let $\dl Z_{N}\in B_{r,\ld}$, with $r$ just chosen. By choosing $\ld=\ld(C_0,K_0,T)\gg 1$ and $N_0 = N_{0}(C_0,K_0,T)\gg 1$, the estimates \eqref{FP1} and \eqref{FP2} show that $\wt{\G}$ is a contraction on $B_{r,\ld}\subset \mc Z^{s_1,s_2}_{\ld}(T)$, and thus has a unique fixed point for any $N\geq N_0$. We then set $Z=\dl Z_{N} +Z_{N}$ and observe that $Z=(X,Y)$ satisfies \eqref{Gsyst2}, since $Z_{N}$ satisfies \eqref{Gsyst}.

For part (b), we write $Z_{N}=Z-(Z-Z_{N})$ and consider the system:
\begin{align*}
\dl X_{N}& = \wt{\G}^{N}_{1}(\dl Z_N)=\G_{1}^{\infty}(X,Y, \vec{\Dr}) - \G_{1}^{N}(X-\dl X_{N},Y -\dl Y_N, \vec{\Dr}_{N}), \\
\dl Y_{N} &= \wt{\G}^{N}_{2}(\dl Z_N)=\G_{2}^{\infty}(X,Y, \vec{\Dr}) - \G_{2}^{N}(X-\dl X_{N},Y -\dl Y_N, \vec{\Dr}_{N}).
\end{align*}
By adapting the arguments for part (a), we see that by choosing $\ld=\ld(C_0,K_0,T)\gg 1$ sufficiently large, the map $(\wt{\G}^{N}_{1}, \wt{\G}^{N}_{2})$ has a unique fixed point in $\mc Z^{s_1,s_2}_{\ld}(T)$.
This constructs the difference $\dl Z_{N}$ and we see that $Z_{N}= Z-\dl Z_{N}$ satisfies \eqref{Gsyst}. 
The convergence property \eqref{Stab5} then follows from Lemma~\ref{LEM:ldup} since
\begin{align*}
\|Z-Z_{N}\|_{\mc Z^{s_1,s_2}(T)} \les \ld e^{2\ld T} \| (\wt{\G}^{N}_{1}, \wt{\G}^{N}_{2})(\dl Z_{N})\|_{\mc Z^{s_1,s_2}_{\ld}(T)} \les C(C_0,K_0,T) N^{-\ta}.
\end{align*}
This completes the proof of Lemma~\ref{LEM:stab}.
\end{proof}

\begin{proof}[Proof of Theorem~\ref{thm:main}]

We begin by proving the almost sure global well-posedness of the (hyperbolic) $\sinh$-Gordon model \eqref{liouville} for $\be^{2}<\frac{6\pi}{5}$.  It suffices to prove the following statement, known as ``almost" almost sure global well-posedness \cite{BO94}: given any $T>0$ and small $\dl_0>0$, there is a measurable event $\Sigma_{T,\dl_0} \subset \H^{s_0}(\T^2)\times \O_2$, with $(\rhoo \otimes \mathbb{P}_{2})(\Sigma_{T,\dl_0}^{c}) <\dl_0$ such that for each $(\vu_0,\o_2)\in \Sigma_{T,\dl_0}$, the solution $(X,Y)$ to the system \eqref{liouville6} with zero initial data and enhanced data set $\Dr(\vu_0,\o_2)$ exists on $[0,T]$.

We first discuss the proof of Theorem~\ref{thm:main}.

Assuming that ``almost" almost sure global well-posedness holds, we construct the global solution $(u,\dt u )\in C(\R_{+}; \H^{s_0}(\T^2))$ to \eqref{liouville}. For $N\in 2^{\N_0}$, let $(X_{N},Y_{N})$ be the (global) solution to the truncated system \eqref{liouvilleNM3} with zero initial data and enhanced data set $\Xi_{N}(\vu_0, \o_2)$ and set 
\begin{align*}
u_{N}(t) = \Psi(\vu_0; \o_2) + X_{N}(t) + Y_{N}(t). 
\end{align*} 
Then, $u_{N}$ solves \eqref{liouvilleNM} with initial data $\vu_0$ and noise $\xi=\xi(\o_2)$. By Lemma~\ref{LEM:enhanced}, the truncated GMCs $\Dr^{a}_{N}(\vu_0,\o_2)$ converges almost surely to $\Dr^{a}(\vu_0,\o_2)$. By part (b) of Lemma~\ref{LEM:stab}, we conclude that $(u_{N},\dt u_N)(\vu_0,\o_2)$ converges to $(u,\dt u)(\vu_0,\o_2)$ in $C([0,T];\H^{s_0}(\T^2))$ for $(\vu_0,\o_2)\in \Sigma_{T,\dl_0}$, where $u$ is defined as
\begin{align}
u(t) =  \Psi(\vu_0; \o_2) + X(t) + Y(t). \label{udecomp}
\end{align}
We then set
\begin{align*}
\Sigma  = \bigcup_{k=1}^{\infty} \bigcap_{j=1}^{\infty} \Sigma_{2^{j},2^{-j}\frac{1}{k}},
\end{align*}
which satisfies 
$$(\rhoo \otimes \mathbb{P}_{2})(\Sigma^{c})=  \lim_{k\to \infty} \sum_{j=1}^{\infty} \frac{1}{2^j k } = 0.$$
Thus, for each $(\vu_0,\o_2)\in \Sigma$, $(u_{N},\dt u_{N})(\vu_0,\o_2)$ converges in $C(\R_{+};\H^{s_0}(\T^2))$ to $(u,\dt u)$.

We now prove the ``almost" almost sure global well-posedness claim. Fix $T>0$ and $\dl_0>0$. Let $\Xi= (\Xi_1, \Xi_{2}, \Xi_3) \in \mathcal{E}^{s_1,s_2}_{T}$ be enhanced data as in \eqref{Xi}, and $(X,Y)[\Xi]$ be the corresponding solution to the system \eqref{liouville6} with enhanced data $\Xi$. Let $C_0 (T,\frac{\dl_0}{3})>0$ be as in Lemma~\ref{LEM:unifNMbds}.  By Lemma~\ref{LEM:enhanced}, there exists $K=K(T,\dl_0)>0$ and $A=A(T,\dl_0)>0$ sufficiently large such that 
\begin{align*}
(\rhoo_{N}\otimes \mathbb{P}_{2})\big(  \| \Xi_{N}(\vu_0, \o_2)\|_{\mathcal{E}^{s_1,s_2}(T)} >K\big) \leq \tfrac{\dl_0}{3},   \\
(\vec{\rho}\otimes \mathbb{P}_{2})\big(\max_{a\in \{c,s\}} N^{\eta}\| \Dr^{a}-\Dr^{a}_{N}\|_{L^{2}_{T}H^{-1+s_2+\frac{\eps}{8}}_{x}} >A \big) \leq \tfrac{\dl_0}{3} ,
\end{align*}
for any $N\in 2^{\N_0}$. Now, let $N_0(T,K,C_0) \in 2^{\N}$ be as in Lemma~\ref{LEM:stab}.
Define the event $\Sigma_{C_0, K, A} \subset \mathcal{E}^{s_1,s_2}(T)$ such that for any $\Xi\in \Sigma^{ K, C_0, A}$, the solution $(X,Y)[\Xi]$ to \eqref{liouville6} with zero initial data and enhanced data set $\Xi$ exists on $[0,T]$ and satisfies 
\begin{align*}
\| (X,Y)[\Xi]\|_{\mc Z^{s_1,s_2}(T)} \leq C_0+1.
\end{align*}
Let $N\geq N_{0}$ and define 
\begin{align*}
B^{K,C_0,A}_{N} = \big\{ \Xi\in \mathcal{E}^{s_1,s_2}(T) \,:& \, \|\Xi_{N}\|_{ \mathcal{E}^{s_1,s_2}(T)} \leq K, \, \| (X_{N},Y_{N})[\Xi_{N}]\|_{\mc Z^{s_1,s_2}(T)}\leq C_0,   \\
&\hphantom{XX} \, N^{\eta} \| (\Xi_{2})_{N} -\Xi_{2}\|_{L^{2}_{T}H^{-1+s_2+\frac{\eps}{8}}_{x}\times L^{2}_{T}H^{-1+s_2+\frac{\eps}{8}}_{x}}\leq A\big\}.
\end{align*}
It follows from Lemma~\ref{LEM:stab} (a) that $B_{N}^{K,C_{0},A} \subseteq \Sigma^{K,C_{0},A}$. Then, 
 \begin{align*}
&(\rhoo\otimes \mathbb{P}_{2})( (\Sigma^{K,C_{0},A})^{c})   \\ 
& \leq (\rhoo\otimes \mathbb{P}_{2})( (B_{N}^{K,C_{0},A})^{c})  \\
& \leq \lim_{N\to \infty} \Big[ (\rhoo_{N}\otimes \mathbb{P}_{2})\big( \| \Xi_{N}\|_{\mathcal{E}^{s_1,s_2}(T)} >K\big) +(\rhoo_{N}\otimes \mathbb{P}_{2})\big( \|(X_{N},Y_{N})[\Dr_{N}]\|_{\mc Z^{s_1,s_2}(T)}>C_0\big)  \Big] +\tfrac{\dl_0}{3} \\
& \leq \tfrac{\dl_0}{3}+\tfrac{\dl_0}{3}+\tfrac{\dl_0}{3}=\dl_0.
\end{align*}
Now by setting 
\begin{align*}
\Sigma_{T,\dl_0}  = \big\{ (\vu_0 ,\o_2)\in \H^{s_0}(\T^2)\times \O_2\,:\, \Xi(\vu_0,\o_2)\in \Sigma^{K,C_0,A}\big\}
\end{align*}
 we complete the verification of the ``almost" almost sure global well-posedness. Thus, the limit $u=u(\vu_0,\o_2)$ is a global-in-time solution to \eqref{liouville} with the decomposition \eqref{udecomp} which exists almost surely with respect to $\rhoo\otimes \O_{2}$.
 
Finally, the invariance of $\rhoo\otimes \O_{2}$ is a simple consequence of the (i) convergence of $\rhoo_{N}$ to $\rhoo$ in total variation, (ii) the almost sure convergence of $u_{N}(t)$ to $u(t)$, (iii) the invariance of $\rhoo_{N}\otimes \mathbb{P}_{2}$ under the flow $\Phi^{N}$ in \eqref{PhiNM}. See for instance \cite[p. 214]{ORTz}.
\end{proof}

Lastly, we deal with the proof of Theorem~\ref{thm:2}.

\begin{proof}[Proof of Theorem~\ref{thm:2}] The proof of Theorem~\ref{thm:2} follows that of Theorem~\ref{thm:main}, except that we can bypass the use of Lemma~\ref{LEM:PuN}, as discussed in Remark~\ref{rmk:liouvilleUS}. This yields the improved numerology $\be^2<\frac{6\pi}{5}$ stated in Theorem~\ref{thm:2}. We omit details.
\end{proof}

\begin{ackno}\rm
J.F.~was partially supported by the ARC project FT230100588. Y.Z.~was funded by the chair of probability and PDEs at EPFL. 
\end{ackno}

\end{document}